\documentclass[10pt]{amsart}

\usepackage{amstext}
\usepackage{latexsym}
\usepackage{amssymb,amsbsy}
\usepackage{amsmath,amsthm,amscd}
\usepackage{enumitem}
\usepackage{extarrows}
\usepackage[pdftex, bookmarks=true, linkbordercolor={0 1 0}]{hyperref}
\usepackage{a4wide}
\usepackage[all]{xy}

\newcommand{\R}{\mathbb R}
\newcommand{\Q}{\mathbb Q}
\newcommand{\C}{\mathbb C}
\newcommand{\Z}{\mathbb Z}
\newcommand{\N}{\mathbb N}
\newcommand{\A}{\mathbb A}
\newcommand{\F}{\mathbb F}
\newcommand{\G}{\mathbb G}
\newcommand{\K}{\mathbb K}
\newcommand{\HH}{\mathbb H}
\newcommand{\eps}{\varepsilon}
\newcommand{\minus}{\backslash}

\newcommand{\Hom}{\mathop{\rm Hom}\nolimits}

\newcommand{\tr}{\mathop{\rm tr}\nolimits}
\newcommand{\vol}{\mathop{\rm vol}\nolimits}
\newcommand{\GL}{\mathop{\rm GL}\nolimits}

\newcommand{\Orth}{{\rm O}}
\newcommand{\Unit}{{\rm U}}
\newcommand{\SL}{\mathop{\rm SL}\nolimits}
\newcommand{\PSL}{\mathop{\rm PSL}\nolimits}
\newcommand{\SU}{\mathop{\rm SU}\nolimits}
\newcommand{\id}{\mathop{\rm id}\nolimits}

\newcommand{\res}{\mathop{\rm res}}
\newcommand{\Ind}{\rm Ind}
\newcommand{\diag}{\mathop{\rm diag}}

\newcommand{\supp}{\mathop{\rm supp}}
\newcommand{\Mat}{\mathop{\rm Mat}\nolimits}
\newcommand{\val}{\mathop{\rm val}\nolimits}
\newcommand{\Gal}{\mathop{\rm Gal}\nolimits}

\newcommand{\cpt}{{\bf K}}

\newcommand{\Ad}{\mathop{{\rm Ad}}\nolimits}

\newcommand{\aaa}{\mathfrak{a}}
\newcommand{\bbb}{\mathfrak{b}}
\newcommand{\ooo}{\mathfrak{o}}
\newcommand{\ppp}{\mathfrak{p}}

\newcommand{\nnn}{\mathfrak{n}}

\newcommand{\ggG}{\mathfrak{g}}
\newcommand{\vvv}{\mathfrak{v}}

\newcommand{\fff}{\mathfrak{f}}
\newcommand{\eee}{\mathfrak{e}}
\newcommand{\AAA}{\mathcal{A}}
\newcommand{\HHH}{\mathcal{H}}
\newcommand{\OOO}{\mathcal{O}}
\newcommand{\UUU}{\mathcal{U}}
\newcommand{\LLL}{\mathcal{L}}
\newcommand{\FFF}{\mathcal{F}}
\newcommand{\PPP}{\mathcal{P}}

\newcommand{\BBB}{\mathcal{B}}
\newcommand{\CCC}{\mathcal{C}}
\newcommand{\DDD}{\mathcal{D}}

\newcommand{\bfk}{\mathbf{k}}
\newcommand{\bfc}{\mathbf{c}}
\newcommand{\bfU}{\mathbf{U}}
\newcommand{\One}{\mathbf{1}}

\newcommand{\VVV}{\mathcal{V}}

\newcommand{\RRR}{\mathcal{R}}

\newcommand{\bkappa}{\boldsymbol{\kappa}}

\newcommand{\bxi}{\boldsymbol{\xi}}

\newcommand{\Res}{\mathop{{\rm Res}}\nolimits}
\newcommand{\lev}{\mathop{{\rm lev}}\nolimits}
\newcommand{\sslash}{\mathbin{/\mkern-5mu/}}

 \theoremstyle{plain}
 \newtheorem{theorem}{Theorem}[section]
 \newtheorem{lemma}[theorem]{Lemma}
 \newtheorem{cor}[theorem]{Corollary}
 \newtheorem{proposition}[theorem]{Proposition}
 
 \theoremstyle{definition}
 \newtheorem{example}[theorem]{Example}
 \newtheorem{definition}[theorem]{Definition}
 \newtheorem{rem}[theorem]{Remark}

\begin{document}
\title[Weyl's law for Hecke operators]{Weyl's law for Hecke operators on ${\rm GL}(n)$ over imaginary quadratic number fields}
\author{Jasmin Matz}
\address{Universit\"at Leipzig, Mathematisches Institut, Postfach 100920, 04009 Leipzig, Germany}
\email{matz@math.uni-leipzig.de}
\keywords{11F70, 11F72}
\begin{abstract}
 We prove Weyl's law for Hecke operators acting on cusp forms of ${\rm GL}(n)$ over imaginary quadratic number fields together with an upper bound for the error term depending explicitly on the Hecke operator.  This has applications to the theory of low-lying zeros of families of automorphic $L$-functions.
\end{abstract}

\maketitle
\tableofcontents
\section{Introduction}
The Weyl law in its simplest form gives an asymptotic for the counting function of the Laplace eigenvalues of a compact Riemannian manifold. 
It has been generalised to many non-compact, finite volume locally symmetric spaces, cf.~\cite{DKV79,Mi01,LiVe07,Mu07,LaMu09,Mu08}.
In the automorphic context Weyl's law for the cuspidal spectrum of $\GL_n$ over $\Q$ was proven in~\cite{Mu07} for arbitrary $\cpt_{\infty}$-type. For trivial $\cpt_{\infty}$-type such an asymptotic together with an upper bound on the error term was established in~\cite{LaMu09}. 
In both cases not only the number of Casimir eigenvalues is counted but rather the asymptotic behaviour of the distribution of the infinitesimal characters of cuspidal automorphic representations is established. 
The main tool in both~\cite{Mu07} and~\cite{LaMu09} was Arthur's trace formula for $\GL_n$, but applied to different types of test functions.  

The purpose of this paper is to study the asymptotic distribution of the infinitesimal characters of cuspidal automorphic representation weighted by the eigenvalues of Hecke operators acting on cusp forms for $\GL_n$ over an imaginary quadratic number field. Additionally, we prove an upper bound for the error term in dependence of the Hecke operator in an explicit way. 
Our motivation mainly stems from recent developments in the theory of low-lying zeros of families of automorphic $L$-functions, cf.~\cite{Sa08, Kow11,SaShTe14}.
For $n=2$ and unramified Hecke operators, i.e.\ for Hecke operators on Laplacian eigenfunctions on the quotient $\PSL_2(\OOO)\backslash \HH^3$ of the hyperbolic $3$-space by the lattice $\PSL(\OOO)$ for $\OOO$ the ring of integers of an imaginary quadratic number field, such an asymptotic was proven in~\cite{ImRa10}. However, their bound on the error term does not include the dependence on the Hecke operator.
See also~\cite{Sa87,CDF97,Se97,ILS00,GoKo12,Bl13} for related results for $n=2,3$ over the field $\Q$. In~\cite{GoKo12,Bl13} instead of the Arthur-Selberg trace formula the Kuznetsov trace formula for $\GL_3$ is used.

\subsection*{Setup}
To state our result in more detail, let $G=\GL_n$ for some $n\geq 1$, and let $\F$ be an imaginary quadratic number field with ring of integers $\OOO_{\F}$ and adeles $\A$. The finite part of the adeles is denoted by $\A_f$. Let $A_G\simeq\R_{>0}\subseteq G(\A)$ denote the set of matrices in the centre of $G(\A)$ with positive real entries. Let $\Pi_{\text{cusp}}(G(\A)^1)$ (resp.\ $\Pi_{\text{disc}}(G(\A)^1)$) denote the set of irreducible unitary representations of $G(\A)$ appearing in the cuspidal (resp.\ discrete) part of $L^2(A_GG(\F)\backslash G(\A))$. 
Let $W$ be the Weyl group of $G$ relative to the diagonal torus, and let $\aaa^*=\{\lambda=(\lambda_1, \ldots, \lambda_n)\in\R^n\mid\lambda_1+\ldots+\lambda_n=0\}$.
If $\pi\in\Pi_{\text{disc}}(G(\A)^1)$, we denote by $\lambda_{\pi_{\infty}}\in\aaa^*_{\C}/W=\aaa^*\otimes\C/W $ the infinitesimal character of $\pi_{\infty}$, where $\pi_{\infty}$ is the archimedean component of $\pi=\pi_{\infty}\otimes\pi_f$. Let $\cpt_{\infty}=\cpt_{\C}=U(n)\subseteq G(\C)$, and for a non-archimedean place $v$ of $\F$ let $\cpt_v= G(\OOO_{\F_v})$ denote the usual maximal compact subgroups. Here $\OOO_{\F_v}$ is the ring of integers of the local field $\F_v$. 
Let  $K_f\subseteq \prod_{v<\infty} \cpt_v=\cpt_f$ be a subgroup of finite index, put $K=\cpt_{\C} K_f$, $\cpt:=\cpt_{\C}\cpt_f$. 
Denote by $\HHH_{\pi}^K$ the space of vectors fixed by $K$ in the representation space $\HHH_{\pi}$ of $\pi$, and similarly, let $\HHH_{\pi_{\infty}}^{\cpt_{\C}}$ denote the $\cpt_{\C}$-invariant vectors in $\HHH_{\pi_{\infty}}$. 

The Hecke algebra of smooth compactly supported bi-$K_f$-invariant functions $\tau: G(\A_f)\longrightarrow\C$ defines for each $\pi\in \Pi_{\text{disc}}(G(\A)^1)$ operators on the space of automorphic forms: 
As usual,  $\pi(\tau)$ operates on $\HHH_{\pi}$ by
\[
 \pi(\tau)\varphi=\int_{G(\A_f)} \tau(x)\pi(x)\varphi\,dx
\]
for $\varphi\in\HHH_{\pi}$.
Let $\pi^{\cpt_{\C}}$ denote the restriction of $\pi$ to $\HHH_{\pi}^{\cpt_{\C}}$, and define $\pi^{\cpt_{\C}}(\tau)$ similar to $\pi(\tau)$.

Let $\FFF_K$ be the family of all cuspidal representations $\pi\in\Pi_{\text{cusp}}(G(\A)^1)$ having a non-trivial $K$-fixed vector, i.e.\ those $\pi$ with $\HHH_{\pi}^K\neq\{0\}$. In particular, the $\cpt_{\C}$-type of every $\pi\in \FFF_K$ is trivial, and $\FFF_K$ is a family of automorphic forms of type (IV) in the terminology of~\cite{Sa08}. 

\subsection*{Results}
 Let $\Omega\subseteq i\aaa^*$ be a $W$-invariant bounded domain with piecewise $C^2$-boundary which will be fixed for the rest of the paper. 
For $t\geq1$ put
\[
\FFF_{K, t}=\{\pi\in\FFF_K \mid \lambda_{\pi_{\infty}}\in t\Omega\},
\]
and further define
\[
 \Lambda_0(t)=|W|^{-1} \vol(A_GG(\F)\backslash G(\A)) \int_{t\Omega} \beta(\lambda)\,d\lambda,
\]
where
 $\beta: \aaa^*_{\C}\longrightarrow\C$ is the spherical Plancherel measure for $G(\C)$ (cf.\ Section~\ref{sec:spherical:fcts}).

Our main result is the following:

\begin{theorem}\label{thm:main}
There exist constants $c_1, c_2, c_3\geq0$ depending only on $n$, $\F$, and $\Omega$ such that the following holds: Suppose $K_f\subseteq \cpt_f$ is a subgroup of finite index, and $\Xi\subseteq G(\A_f)$ is a bi-$K_f$-invariant compact set (i.e.\ $k_1\Xi k_2=\Xi$ for all $k_1, k_2\in K_f$). Let $\tau: G(\A_f)\longrightarrow \C$ be the characteristic function of $\Xi$ normalised by $\vol(K_f)^{-1}$. 
Then
\begin{equation}\label{thm_main_eq}
\lim_{t\rightarrow\infty}\Lambda_0(t)^{-1}\sum_{\pi\in\FFF_{K,t}}\tr \pi^{\cpt_{\C}}(\tau) 
= \sum_{z\in Z(\F)} \tau(z),
\end{equation}
where $Z\subseteq G$ denotes the centre of $G$.
Moreover,
\begin{equation}\label{thm_main_error}
\bigg|\sum_{\pi\in\FFF_{K,t}}\tr \pi^{\cpt_{\C}}(\tau) 
- \Lambda_0(t) \sum_{z\in Z(\F)} \tau(z)\bigg|
\leq c_1 [\cpt:K]^{c_2} \Big(\frac{\vol(\Xi)}{\vol(\cpt_f)}\Big)^{c_3} ~t^{d-1} (\log t)^{n+1} 
\end{equation}
for every $t\geq2$, where $d=\dim_{\R} A_G \backslash G(\F_{\infty})/\cpt_{\C}=\dim_{\R}\SL_n(\C)/\SU(n)=n^2-1$. 
\end{theorem}

\begin{rem}\label{rem:about:thm:main}
\begin{enumerate}[label=(\roman{*})]
 \item  The function $\Lambda_0(t)$ is of order $t^d$ so that the upper bound for the error term in~\eqref{thm_main_error} is indeed of lower order than the main term (if it does not vanish). 

\item 
The analogue of the theorem for other number fields $\F$, and in particular $\F=\Q$ is also expected to hold (with a possibly different error estimate). Our restriction here to imaginary quadratic number fields is for technical reasons, and a proof of the above result for $\F=\Q$ is the content of current joint work with N. Templier~\cite{weyl2}. For more general number fields see~\cite{proceedings}.

\item If $K_f=\cpt_f$ and $\Xi=\cpt_f a \cpt_f$  for some $a\in T_0(\A_f)$ (for $T_0(\A_f)\subseteq G(\A_f)$ the maximal torus of diagonal matrices), then
\[
 \sum_{z\in Z(\F)} \tau(z) =\begin{cases}
					      \nu_{\F}\vol(\cpt_f)^{-1}				&\text{if }a\in  Z(\F)(T_0(\A_f)\cap \cpt_f),\\
					      0							&\text{otherwise},
						\end{cases}
\]
where $\nu_{\F}$ denotes the number of multiplicative units $\OOO_F^{\times}$ in $\OOO_F$.
More generally, 
if $K_f$ is an arbitrary finite index subgroup of $\cpt_f$, and $\Xi=K_f a K_f$  for some $a\in T_0(\A_f)$, then
\[
 \sum_{z\in Z(\F)} \tau(z) =\begin{cases}
					      \nu_{\F, K_f}\vol(K_f)^{-1}			&\text{if }a\in  Z(\F)(T_0(\A_f)\cap K_f),\\
					      0							&\text{otherwise},
						\end{cases}
\]
where $\nu_{\F, K_f}$ denotes the number of elements in $Z(\F)\cap K_f$ which is the number of $\OOO_F$-units ``contained'' in $K_f$.

\item  In the case $\Xi=K_f$, we get
\begin{equation}\label{eq:simple:weyl}
 \sum_{\pi\in\FFF_{K,t}}\dim\HHH_{\pi}^K 
\sim  \nu_{\F, K_f} \vol(K_f)^{-1}\Lambda_0(t)
\end{equation}
as $t\rightarrow\infty$. The analogue of~\eqref{eq:simple:weyl} for $\GL_n/\Q$ and $K_f$ contained in a principal congruence subgroup of level at least $3$ was proven in~\cite{LaMu09} together with an upper bound for the error term (but without explicit dependence on the index of $K$ in $\cpt$).

\item
It would be interesting to extend Theorem~\ref{thm:main} to the family $\FFF_{K, \sigma}$ of cuspidal automorphic representations $\pi$ such that $\pi_f$ has a $K_f$-fixed vector and $\pi_{\infty}$ has arbitrary but fixed $\cpt_{\C}$-type $\sigma\in\widehat{\cpt_{\C}}$. In that case, the right hand side of~\eqref{thm_main_eq} gains an additional factor of $\dim\sigma$ (cf.\ also~\cite[Theorem 0.2]{Mu07}).

\end{enumerate}
\end{rem}

\subsection*{Unramified Hecke operators}
In the case that every $\pi\in\FFF_{K}$ is unramified, i.e.\ $K=\cpt$, one can reformulate Theorem~\ref{thm:main} in a slightly different form:
 We identify the set of non-archimedean places of $\F$ with the set of prime ideals $\ppp \subseteq\OOO_{\F}$ and accordingly write $\F_{\ppp}$, $\cpt_{\ppp}$, and so on.
Let $\aaa\subseteq \OOO_{\F}$ be an integral ideal  with prime factorisation $\aaa=\prod \ppp^{e_{\ppp}}$ for suitable integers $e_{\ppp}\geq0$, almost all $e_{\ppp}=0$, and let $\N(\aaa)=\big|\OOO_{\F}/\aaa\big|$ be the norm of $\aaa$. 
Let $T_{\ppp^{e_{\ppp}}}:G(\F_{\ppp})\longrightarrow\C $ be the usual element in the unramified Hecke algebra of $G(\F_{\ppp})$ 
associated with $\ppp^{e_{\ppp}}$, i.e.,
\[
 T_{\ppp^{e_{\ppp}}} =\vol(\cpt_{\ppp})^{-1}\N_{\F/\Q}(\ppp)^{-e_{\ppp}(n-1)/2}\sum_{\xi\in \Lambda_{e_{\ppp}}} \tau_{\xi} 
\]
where $\xi$ runs over all tuples of integers $(\xi_1, \ldots, \xi_n)$ with $\xi_1\geq\ldots\geq\xi_n\geq0$ and $\xi_1+\ldots+\xi_n=e_{\ppp}$,
and $\tau_{\xi}$ is the characteristic function of the double coset 
\[
\cpt_{\ppp} \diag(\varpi_{\ppp}^{\xi_1}, \ldots, \varpi_{\ppp}^{\xi_n})\cpt_{\ppp}\subseteq G(\F_{\ppp})
\]
where $\varpi_{\ppp}\in\OOO_{\F_{\ppp}}$ a uniformising element. Put $T_{\aaa}=\prod T_{\ppp^{e_{\ppp}}}$. Note that with 
\[
\Xi_{\ppp}= \bigsqcup_{\xi\in\Lambda_{e_{\ppp}}} \cpt_{\ppp} \diag(\varpi_{\ppp}^{\xi_1}, \ldots, \varpi_{\ppp}^{\xi_n})\cpt_{\ppp}\subseteq G(\F_{\ppp}),
\]
and $\Xi:=\prod_{\ppp}\Xi_{\ppp}$ the quotient  $\frac{\vol(\Xi)}{\vol(\cpt_f)}$ equals the usual degree of the Hecke operator $T_{\aaa}$, and this degree is bounded by a power of $\N(\aaa)$.
Then Theorem~\ref{thm:main} and together with Remark~\ref{rem:about:thm:main} gives the following:

\begin{cor}
 We have
\[
\lim_{t\rightarrow\infty}\Lambda_0(t)^{-1}\sum_{\pi\in\FFF_{\cpt, t}} \tr \pi^{\cpt_{\C}}(T_{\aaa})
= \nu_{\F} \vol(\cpt_f)^{-1}\delta_n(\aaa)
\]
as $t\rightarrow\infty$, 
where
\[
 \delta_n(\aaa)=\begin{cases}
                 1	 					&\text{if }\aaa=\bbb^{n}\text{ for some principal ideal } \bbb\subseteq \OOO_{\F},\\
		0						&\text{otherwise},
                \end{cases}
\]
and $\nu_{\F}$ is as in Remark~\ref{rem:about:thm:main}.
 In any case, for every $t\geq2$, we have 
\[
\bigg|\sum_{\pi\in\FFF_{\cpt, t}} \tr\pi^{\cpt_{\C}}( T_{\aaa}) 
- \frac{\nu_{\F}\Lambda_0(t)}{\vol(\cpt_f)}\delta_n(\aaa)\bigg|
\leq c_1 \N(\aaa)^{c_2} t^{d-1} (\log t)^{n+1},
\]
where $c_1, c_2>0$ are suitable integers depending only on $n$, $\F$, and $\Omega$ but not on $\aaa$.
\end{cor}

If $\pi\in\FFF_{\cpt}$, let $L(s, \pi)=\prod_{v\not\in S_{\infty}} L_v(s, \pi_v)$ be the standard $L$-function of $\pi$. $L(s, \pi)$ can be written as a Dirichlet series
\[
 L(s, \pi) =\sum_{\aaa\subseteq \OOO_{\F}} a_{\pi}(\aaa) \N(\aaa)^{-s}
\]
for suitable coefficients $a_{\pi}(\aaa)$ 
that satisfy $a_{\pi}(\aaa)=\prod_{\ppp\not\in S_{\infty}} a_{\pi}(\ppp^{e_{\ppp}})$ and $a_{\pi}(\aaa)=\tr\pi^{\cpt_{\C}} (T_{\aaa})$. 
Hence the above corollary takes the following form:

\begin{cor}\label{thm:coeff:l-fcts}
we have
\[
 \lim_{t\rightarrow\infty} \Lambda_0(t)^{-1}\big|\FFF_{\cpt, t}\big|
= \frac{\nu_{\F}}{\vol(\cpt_f)},
\]
and for every ideal $\aaa\subseteq \OOO_{\F}$
\[
\lim_{t\rightarrow\infty} \big|\FFF_{\cpt,t}\big|^{-1} \sum_{\pi\in \FFF_{\cpt,t}} a_{\pi}(\aaa) 
=\delta_n(\aaa).
\]

Further, for every $t\geq2$,
\[
 \bigg|\sum_{\pi\in \FFF_{\cpt, t}} a_{\pi}(\aaa) - \big|\FFF_{\cpt,t}\big| \delta_n(\aaa)\bigg|
\leq c_1 \N(\aaa)^{c_2} ~ t^{d-1} (\log t)^{n+1}
\]
for suitable $c_1, c_2>0$ independent of $\aaa$.
\end{cor}

\begin{rem}
This corollary in particular gives asymptotic formulas for the quantities in  $(1)$ and $(2)$ of~\cite{Sa08} for our family $\FFF_{\cpt}$.
\end{rem}

\subsection*{The ramified case}
Fix an ideal $\nnn\subseteq \OOO_F$ with prime factorisation $\nnn=\prod_{\ppp} \ppp^{n_{\ppp}}$, and let $S_{\nnn}$ be the set of prime ideals dividing $\nnn$. Let $K_f:=K_f(\nnn):=\prod_{\ppp}K_{\ppp}(\ppp^{n_{\ppp}})$ be the principal congruence subgroup of level $\nnn$ with 
\[
K_{\ppp}:= K_{\ppp}(\ppp^{n_{\ppp}}):=\{x\in \cpt_{\ppp} \mid x\equiv \mathbf{1}_n \mod\ppp^{n_{\ppp}}\},
\]
and set $K=K(\nnn)=\cpt_{\C} \cdot K_f$.
In particular, $K_{\ppp}=\cpt_{\ppp}$ if $\ppp\not\in S_{\nnn}$.

If $\pi\in\FFF_{K}$, $\pi$ does not need to be unramified at the places dividing $\nnn$ so its standard $L$-function is only defined outside of $S_{\nnn}$. Let $L(s, \pi)=\prod_{\ppp\not\in S_{\nnn}} L_{\ppp}(s, \pi_{\ppp})$ be this $L$-function. Again writing $L(s, \pi)= \sum_{\aaa} a_{\pi}(\aaa) \N(\aaa)^{-s}$  for its Dirichlet series with $a_{\pi}(\aaa)=0$ if $\aaa$ and $\nnn$ are not coprime, we have 
\[a_{\pi}(\aaa)=\prod_{\ppp|\aaa} \tr \pi_{\ppp}(T_{\ppp^{e_{\ppp}}})\]
for every $\aaa$ coprime to $\nnn$. Let $\tau_{\nnn}:G(\F_{S_{\nnn}})\longrightarrow\C$ be the characteristic function on $\prod_{\ppp\in S_{\nnn}} K_{\ppp}(\ppp^{n_{\ppp}})$ normalised by the inverse of the volume of this set, and put
\[
 \tau^{\aaa}:= \tau_{\nnn}\cdot \prod_{\ppp\not\in S_{\nnn}} T_{\ppp^{e_{\ppp}}}
\]
for every ideal $\aaa$ coprime to $\nnn$. Then $\tr \pi^{\cpt_{\infty}} (\tau^{\aaa})= \dim\HHH_{\pi}^{K(\nnn)} a_{\pi}(\aaa)$, and Theorem~\ref{thm:main} gives:

\begin{cor}
 Let $\aaa\subseteq \OOO_{\F}$ be an ideal coprime to $\nnn$. 
We write, abusing notation,
\[
 \big|\FFF_{K(\nnn), t}\big|:=\sum_{\pi\in\FFF_{K(\nnn),t}}\dim\HHH_{\pi}^{K(\nnn)}.
\]
Then 
\[
 \lim_{t\rightarrow\infty} \Lambda_0(t)^{-1} \big|\FFF_{K(\nnn), t}\big|
= \frac{\nu_{\F, K_f(\nnn)} }{\vol(K_f(\nnn))},
\]
and
\[
\lim_{t\rightarrow\infty} \big|\FFF_{K(\nnn),t}\big|^{-1} \sum_{\pi\in \FFF_{K(\nnn),t}} a_{\pi}(\aaa)  \dim\HHH_{\pi}^{K(\nnn)} 
= \delta_n(\aaa).
\]
Moreover, for every $t\geq2$, 
\[
  \bigg|\sum_{\pi\in \FFF_{K(\nnn), t}} a_{\pi}(\aaa) \dim\HHH_{\pi}^{K(\nnn)}   - \big|\FFF_{K(\nnn),t}\big| \delta_n(\aaa)\bigg|
\leq c_1 \N(\nnn)^{c_2} \N(\aaa)^{c_3} ~ t^{d-1} (\log t)^{n+1}
\]
for suitable $c_1, c_2, c_3>0$ independent of $\nnn$ and $\aaa$.
\end{cor}

\subsection*{Outline of the proof of Theorem \ref{thm:main}}
As in~\cite{Mu07,LaMu09} our main tool is Arthur's trace formula for $G=\GL_n$ (but over $\F$ instead of $\Q$), and most of the paper is concerned with analysing the different parts of the trace formula. 
Recall that Arthur's trace formula is an identity $J_{\text{geom}}(f)=J_{\text{spec}}(f)$ between the so-called geometric and spectral side. These are distributions on a certain space of test functions $f:G(\A)^1\longrightarrow\C$ and can both be written in terms of sums of ``finer'' distributions.
More precisely, the rough structure of the spectral side is 
\begin{equation}\label{eq:spec:exp:intro}
 J_{\text{spec}}(f)=\sum_{\pi\in\Pi_{\text{disc}}(G(\A)^1)} \tr\pi(f) + J_{\text{cts}}(f),
\end{equation}
where $J_{\text{cts}}(f)$ is a certain distribution associated to the continuous spectrum of $G$, and the first sum is the part of $J_{\text{spec}}(f)$ associated with the discrete spectrum. 
The geometric side, on the other hand, can be described by Arthur's fine geometric expansion (cf.~\cite{Ar86} and also Section~\ref{sec:fine:geom:exp}):
Let $S$ be a sufficiently large finite set of valuations of $\F$. There exist coefficients $a^M(\gamma, S)\in\C$ and certain $S$-adic weighted orbital integrals $J_M^G(\gamma, f)$ such that
\begin{equation}\label{eq:fine:geom:intro}
 J_{\text{geom}}(f)
=\sum_{M}\frac{|W^M|}{|W^G|} \sum_{\gamma} a^M(\gamma, S) J_M^G(\gamma, f)
\end{equation}
for every $f\in C_c^{\infty}(G(\A)^1)$ provided $S$ is sufficiently large with respect to  the support of $f$ (cf.~\cite{Ar86}). Here $M$ runs over $\F$-Levi subgroups in $G$ containing the torus of diagonal elements, $\gamma\in M(\F)$ runs over a set of representatives for the $M(\F)$-conjugacy classes in $M(\F)$, and $W^M$ denotes the Weyl group of $M$ with respect to $T_0$.

The first task is to choose good test functions (or rather a family of test functions), cf.\ Section~\ref{sec:test:fcts}: This will be done in the spirit of~\cite{DKV79} from the ``spectral point of view`` such that the discrete part in~\eqref{eq:spec:exp:intro} almost immediately yields the left hand side of~\eqref{thm_main_eq}.
To be more precise, we choose a specific family of test functions $F^{\mu, \tau}$ depending on the spectral parameter $\mu\in\aaa_{\C}^*$ and a function $\tau:G(\A_f)\longrightarrow\C$ as in Theorem~\ref{thm:main} (the ''Hecke operator''). It basically equals the product of $\tau$ with a test function at the archimedean place, namely, a bi-$\cpt_{\C}$-invariant, smooth, compactly supported function $f_{\C}^{\mu}: G(\C)^1=\{g\in G(\C)\mid|\det g|_\C=1\}\longrightarrow\C$.This function is constructed via the Paley-Wiener theorem from Harish-Chandra's elementary spherical function.
With this family of test functions, the integral 
\[
\int_{t\Omega} \sum_{\pi\in\Pi_{\text{disc}}(G(\A)^1)} \tr\pi(F^{\mu, \tau}) \,d\mu
\]
essentially equals $\sum_{\pi\in\FFF_{K, t}} \tr\pi^{\cpt_{\C}}(\tau)$.

The trace formula then gives the identity
\[
\int_{t\Omega} \sum_{\pi\in\Pi_{\text{disc}}(G(\A)^1)} \tr\pi(F^{\mu, \tau}) \,d\mu
= \int_{t\Omega} J_{\text{geom}}(F^{\mu, \tau}) \,d\mu
-\int_{t\Omega} J_{\text{cts}}(F^{\mu, \tau}) \,d\mu.
\]
There are basically two different problems now: We need to show that 
\begin{equation}\label{eq:remaining:terms}
\int_{t\Omega} \big(J_{\text{geom}}(F^{\mu, \tau})-\sum_{z\in Z(\F)}F^{\mu, \tau}(z)\big) \,d\mu
\;\text{ as well as }\;\;\int_{t\Omega} J_{\text{cts}}(F^{\mu, \tau})\,d\mu
\end{equation}
both only contribute to the error term in~\eqref{thm_main_error}.

In the first integral we plug in the fine geometric expansion~\eqref{eq:fine:geom:intro}. Note that the first sum over $M$ in~\eqref{eq:fine:geom:intro} is a priori finite and does not pose any problem. The second sum over $\gamma$ is in fact also finite, because all but finitely many $J_M^G(\gamma, F^{\mu, \tau})$ vanish. One of our first tasks therefore will be to determine which $\gamma$ contribute non-trivially, and how large $S$ has to be, both in dependence on the support of our chosen function $\tau$, see Section~\ref{sec:fine:geom:exp}.
The absolute value of the coefficients $a^M(\gamma, S)$ can be controlled by~\cite{coeff_est} (cf.\ Proposition~\ref{prop:coeff:est}) after we fix a normalisation of measures.

Hence on the geometric side, i.e.\ for the treatment of the first integral in~\eqref{eq:remaining:terms}, our remaining task is to bound the integrals $\int_{t\Omega}J_M^G(\gamma, F^{\mu,\tau})\,d\mu$ for $M\neq G$ or $\gamma$ not central. For this it suffices to consider the unramified case $K_f=\cpt_f$ so that $\tau$ is an element of the unramified Hecke algebra for $G(\A_f)$.
 The terms 
$J_M^G(\gamma, F^{\mu, \tau})$ can be reduced to sum-products of local $v$-adic weighted orbital integrals for $v\in S$. 
Hence we are left with the problem of bounding local weighted orbital integrals, namely, 
\begin{enumerate}[label=(\alph{*})]
\item\label{arch} at the archimedean place: $\int_{t\Omega} J_M^G(\gamma, f_{\C}^{\mu})\,d\mu$ as $t\rightarrow\infty$, and
\item\label{non-arch} at the non-archimedean places in $S$: $J_M^G(\gamma, \tau_{v})$, for $\tau_v$ the characteristic function of the double coset $\cpt_v a\cpt_v$ for some $a\in G(\F_v)$,
\end{enumerate}
both uniformly in $\tau$ and $S$ in the sense that they only contribute to the error term in~\eqref{thm_main_error}. 

For~\ref{arch}, we distinguish two different types of $\gamma$: If $\gamma$ is ``almost unipotent'' by which we mean that all complex eigenvalues of $\gamma$ have the same absolute value, we apply  results on integrals of spherical functions over unipotent orbits from~\cite{LaMu09} to control its contribution. 
If $\gamma$ is not almost unipotent (e.g., regular), we can bound the $G(\C)$-orbit of $\gamma$ effectively away from $\cpt_{\C}$.
This is enough to establish uniform bounds for the product of Harish-Chandra's elementary spherical functions with the Plancherel measure (see Section~\ref{sec:spherical:fcts}, Section~\ref{sec:aux:est:arch}, and Section~\ref{sec:weighted:orb:int:arch}). By construction of $f_{\C}^{\mu}$ this suffices to bound~\ref{arch}.

\begin{rem}\label{rem:diff:spher}
Having a good (pointwise) bound for the spherical function is crucial for this approach. For $\GL_n(\C)$ the fact that the spherical function is expressible as elementary functions makes it possible to give fairly precise information on the decay of the spherical function $\phi_{\lambda}(e^H)$ as $\lambda\rightarrow\infty$ depending on the degeneracy of $\lambda$ and $H$, see Section~\ref{sec:spherical:fcts}. 
On the other hand, the spherical functions on $\GL_n(\R)$ do not have such a simple form (this is due to the fact that the root system of $\GL_n(\R)$ has odd multiplicities), and upper bounds for these functions are much harder to establish. Partial results were given, e.g., in~\cite{DKV83,Mar13}, but they are not sufficiently uniform when $\lambda$ or $H$ become singular. 

Recently, an upper bound for the spherical function on $\GL_n(\R)$ (and more general groups) was established in~\cite{BlPo14} and independently in~\cite{weyl2} which remains valid for degenerate $\lambda$ and $H$. This bound does not give the same amount of information as for $\GL_n(\C)$ in Section~\ref{sec:spherical:fcts}, but is sufficient to obtain an upper bound for the weighted orbital integrals for $\GL_n(\R)$ that is good enough to  establish the analogue of Theorem~\ref{thm:main} for the base field $\Q$. 
The proof of such a result for $\F=\Q$ is the content of current joint work with N. Templier~\cite{weyl2}.
\end{rem}

For~\ref{non-arch} we proceed as follows: Let $\gamma=\sigma\nu=\nu\sigma$ be the Jordan decomposition of $\gamma$ with $\sigma$ semisimple and $\nu$ unipotent.
Using Arthur's description of local weighted orbital integrals (cf.~\cite{Ar88a}), $J_M^G(\gamma, \tau_v)$ equals an orbital integral of $\tau_v$ over the quotient $G_{\sigma}(\F_v)\backslash G(\F_v)$ and over a certain unipotent class in $G_{\sigma}(\F_v)$ depending on $M$ and $\nu$.
The involved measures can be described as products of weight functions against suitable invariant measures. 
We use an extension of the  methods of Shin and Templier in~\cite[\S 7]{ShTe12} to find certain bounded sets depending on the support of $\tau_v$ 
such that we may replace the integral over $G_{\sigma}(\F_v)\backslash G(\F_v)$ and the integral over the unipotent class in $G_{\sigma}(\F_v)$ by these bounded sets (cf.\ Section~\ref{sec:aux:est:non-arch}).

The integral over the unipotent class can be bounded by using the construction of the weight functions from~\cite{Ar88a}. To get uniform bounds one needs to observe that the construction of the weight function basically depends only on the underlying based root datum of $G_{\sigma}/\F_v$ and the unipotent class in $G_{\sigma}(\F_v)$. 
There are only finitely many possibilities for the based root datum and the unipotent class (uniformly over all $v$ and $\sigma$) so that we can make the estimates uniform, see Section~\ref{sec:weighted:orb:int:non-arch}. 

After this, we are essentially left with an invariant orbital integral of some compactly supported bi-$\cpt_v$-invariant function over the $G(\F_v)$-orbit of $\sigma$. 
These invariant integrals were studied in~\cite[\S 7]{ShTe12} where an upper bound for them was established.

On the spectral side, i.e.\ for the treatment of the second integral in~\eqref{eq:remaining:terms}, the remaining distribution belonging to the continuous spectrum only contributes to the error term in~\eqref{thm_main_error}. The analysis of the spectral side in Section~\ref{sec:spec:side} and Section~\ref{sec:spectral:est} basically works the same as the analysis of the corresponding problem in the case of $G/\Q$ and $\Xi=K_f$ in~\cite{LaMu09}. However, we need to be more careful to make the dependence on the index $[\cpt:K]$ explicit.
To sum up, $\int_{t\Omega} J_{\text{spec}}(F^{\mu, \tau})\,d\mu$, and therefore also $\int_{t\Omega} J_{\text{geom}}(F^{\mu, \tau}) \,d\mu$, equals $\sum_{\pi\in\FFF_{\cpt,t}} \tr\pi^{\cpt_{\C}}(\tau)$  as well as $\Lambda_0(t)\sum_{z\in Z(\F)} \tau(z)$ up to an error term as in~\eqref{thm_main_error} as $t\rightarrow\infty$.

\subsection*{Organisation of the paper}
We first fix some notation and conventions in Section~\ref{sec:notation}. In Section~\ref{sec:hecke:alg} we recall some facts about Hecke algebras. In Section~\ref{sec:spherical:fcts} we establish some properties of Harish-Chandra's spherical functions on $G(\C)$ so that we can finally define our global test functions in Section~\ref{sec:test:fcts}. The reduction of the analysis of the geometric side to the problems~\ref{arch} and~\ref{non-arch} above will occupy Section~\ref{sec:fine:geom:exp}. We then solve problem~\ref{non-arch} in Sections~\ref{section_norms}--\ref{sec:weighted:orb:int:non-arch}, and problem~\ref{arch} in Sections~\ref{sec:aux:est:arch}--\ref{sec:weighted:orb:int:arch}. After summarising the results for the geometric side in Section~\ref{sec:summary:geom}, we continue the analysis of the spectral side in Section~\ref{sec:spec:side}--\ref{sec:spectral:est}. Finally, we combine all partial results to finish the proof of Theorem~\ref{thm:main} in Section~\ref{sec:finish:main:res}.

\section{Notation}\label{sec:notation}
We  fix an imaginary quadratic number field $\F$ and an integer $n\geq2$ for the rest of the paper and write $G=\GL_n$.
All constants appearing in any estimate throughout the paper are allowed to depend on $n$ and $\F$ even if this is not explicitly mentioned. If $f,g\in\R$, we write $f\ll g$ if there exists some $a>0$ (possibly depending on $n$ and $\F$ by our convention) such that $f\leq ag$. If this constant $a$ depends on further parameters $\alpha, \beta, \ldots$ but no others we write $f\ll_{\alpha,\beta, \ldots} g$.

If $E$ is any number field we denote by $\OOO_E$ the ring of integers of $E$.
Let $\A$ be the ring of adeles of $\F$, and $S_{\infty}$ the set of archimedean place of $\F$. If $S$ is an arbitrary finite set of valuations of $\F$, we write $\F_S=\prod_{v\in S} \F_v$, $\A^S=\prod_{v\not\in S} \F_v$, $\F_{\infty}=\prod_{v\in S_{\infty}}\F_{v}\simeq\C$, and $\A_f=\A^{S_{\infty}}$.
We denote the norm of elements  in $\F$ over $\Q$ by $\N=\N_{\F/\Q}$. Note that if we view $z\in \F$ as embedded into $\C$ via one of the embeddings, then $|z|_{\C}=\N_{\F/\Q}(z)$ regardless of the chosen embedding. If $x\in\R$ is a real number, we may also write $|x|_\R$ for the usual absolute value on $\R$.

If $F$ is a finite extension of $\Q_p$ for some rational prime $p$, we denote by $\OOO_{F}\subseteq F$ the ring of integers, by $\varpi_F\in \OOO_{F}$ a uniformising element, by $q_F$ the cardinality of the residue field, and by $|\cdot|_{F}$ the norm on $F$ normalised by $|\varpi_F|_F=q_F^{-1}$. If $F=\F_v$ for some non-archimedean valuation $v$ of $\F$, we may also write $\varpi_v=\varpi_{\F_v}$, $q_v=q_{\F_v}$,  and so on.
If $x\in\A^{\times}$, we denote by $|x|=|x|_{\A}$ the adelic norm of $x$, i.e.\ $|x|=\prod_v |x|_v$ with the product running over all valuations of $\F$. We let $\A^1=\{x\in\A^{\times}\mid|x|=1\}$, $G(\A)^1=\{g\in G(\A)\mid|\det g|=1\}$, $G(\C)^1=\{g\in G(\C)\mid |\det g|_{\C}=1\}$, and $G(\R)^1=\{g\in G(\R)\mid |\det g|_{\R}=1\}$.

 We fix non-negative integers $P_v\geq0$ for each non-archimedean place $v$ of $\F$ with $P_v=0$ for almost all $v$ such that the following holds:
 Each element $[\aaa]$ of the Picard group of $\F$ has a representative $\aaa\in[\aaa]$ such that $\aaa=\prod_{v<\infty}\ppp_v^{e_v}$ with $0\leq e_v\leq P_v$ and $\ppp_v\subseteq F_v$ the prime ideal at $v$.

\subsection{Subgroups}
We denote by $T_0\subseteq G$ the maximal torus consisting of diagonal matrices, by $U_0\subseteq G$ the unipotent subgroup of upper triangular matrices, and by $P_0=T_0U_0$ the usual minimal parabolic subgroup in $G$. Let $\LLL$ be the set of $\Q$-Levi subgroups in $G$ containing $T_0$, and if $M\in\LLL$, let $\LLL(M)=\{L\in\LLL\mid M\subseteq L\}$. Further, let $\FFF(M)$ be the set of all $\Q$-parabolic subgroups $P\subseteq G$ containing $M$, and let $\PPP(M)$ be the set of $P\in\FFF(M)$ with $P\not\in \FFF(L)$ for any $L\supsetneq M$. We write $\FFF=\FFF(T_0)$ and denote by $W=W^G$ the Weyl group of $G$ with respect to $T_0$. We identify $W$ with the subgroup of permutation matrices in $G(\F)$ whenever convenient.
If $P\in\FFF$, we write $P=M_PU_P$ for the Levi decomposition of $P$ such that $T_0\subseteq M_P$. Recall that $P\in\FFF$ is called standard if $P_0\subseteq P$. 
We denote by $\FFF_{\text{std}}\subseteq\FFF$ the set of all standard parabolic subgroups.

Let $A_P=A_{M_P}\subseteq M(\R)$ be the identity component of the centre of $M_P(\R)$. We choose maximal compact subgroups in $G(\A)$ and $G(\F_v)$ as follows: If $v$ is an archimedean (hence complex) place of $\F$, we let $\cpt_{\C}=\cpt_v=\Unit(n)$, and if $v$ is non-archimedean, we let $\cpt_{\F_v}=\cpt_v=G(\OOO_{\F_v})$. We also set $\cpt_{\R}=\Orth(n)$. 
We put $\cpt=\prod_{v} \cpt_v$ which is the usual maximal compact subgroup in $G(\A)$. More generally, if $F$ is any local extension of $\Q_p$ for $p$ a rational prime, we put $\cpt_F=G(\OOO_F)$. Then $\cpt_F$ is a hyperspecial maximal compact subgroup of $G(F)$ and admissible relative to the torus $T_0(F)$.

 For every standard parabolic subgroup $P$ we have the Iwasawa decomposition
\[
 G(\A)= P(\A) \cpt= U_P(\A) M_P(\A)^1 A_{M_P} \cpt,
\]
and locally, $G(\F_v)= P(\F_v)\cpt_v=U_P(\F_v)M_P(\F_v) \cpt_v$. Let $\aaa_{M_P}=\aaa_P$ be the Lie algebra of $A_P$, and let
\[
 H_P: G(\A)=U_P(\A) M_P(\A)^1 A_{M_P} \cpt \longrightarrow \aaa_P
\]
be the map characterised by $H_P(ume^Xk)=X$ for $X\in\aaa_P$. We usually write $H_0=H_{P_0}$. 

If $P_1\subseteq P_2$, we have a natural projection $\aaa_{P_1}\longrightarrow \aaa_{P_2}$, and we denote the kernel of this map by $\aaa_{P_1}^{P_2}$. In particular, we set $\aaa:=\aaa_0^G=\aaa_{P_0}^G$. We define $\aaa_{P_1}^*$ (and similarly, $\big(\aaa_{P_1}^{P_2}\big)^*$ and $\aaa^*$) to be the dual space $\Hom_{\R}(\aaa_{P_1}, \R)$. We set $\aaa_0=\aaa_{P_0}$ and $\aaa_0^*=\aaa^*_{P_0}$.

If $H\subseteq G$ is any $\Q$-subgroup, we denote by $\UUU_H$ the variety of unipotent elements in $H$, and if $\gamma\in H(F)$ (with $F$ a local or global field), we denote by $H_{\gamma}(F)=\{g\in H(F)\mid g\gamma=\gamma g\}$ the centraliser of $\gamma$ in $H(F)$.

\subsection{Roots and weights}\label{sec:roots}
Let $\Phi^+$ be the set of positive roots for $(T_0, U_0)$, $\Phi=\Phi^+\cup(-\Phi^+)$, and let $\Delta_0\subseteq \Phi^+$ be the subset of simple roots.
Let $X_*(T_0)\subseteq \aaa_0$ denote the lattice of rational cocharacters, and $X^*(T_0)\subseteq \aaa_0^*$ the lattice of rational characters. We identify the spaces $\aaa_0$ and $\aaa_0^*$ with $\R^n$ in the usual way. In particular, if $p$ be a rational prime and $F/\Q_p$ a finite extension,  $X_*(T_0)$ is identified with $\Z^n$ in the usual way, namely by $\Z^n\ni(\xi_1, \ldots, \xi_n)\mapsto \diag(\varpi_F^{\xi_1}, \ldots, \varpi_F^{\xi_n})\in T_0(F)$.

The set of positive roots $\Phi^+\supseteq \Delta_0$ defines a subset of positive cocharacters $X_*^+(T_0)\subseteq X_*(T_0)$ by saying that $\xi=(\xi_1,\ldots,\xi_n)\in X_*^+(T_0)$ if and only if $\alpha(\xi)\geq0$ for all $\alpha\in\Phi^+$ and $\xi_1+\ldots+\xi_n\ge0$. 
We let $\Delta_0=\{\alpha_1, \ldots, \alpha_{n-1}\}$ be the usual ordering of the positive simple roots, i.e.\ $\alpha_i((\xi_1, \ldots, \xi_n))=\xi_i-\xi_{i+1}$ for every $\xi=(\xi_1, \ldots, \xi_n)\in\aaa_0$.
There is a partial ordering ``$\leq$'' on the set $X_*(T_0)$ defined by $\xi \leq \zeta$ if and only if $\zeta-\xi$ is a non-negative linear combination of positive coroots and $\sum_{i=1}^n (\zeta_i-\xi_i)\ge0$. 
Let $\aaa^+=\{\xi\in\aaa\mid\forall\alpha\in\Phi^+:~ \alpha(\xi)>0\}$ be the positive Weyl chamber in $\aaa$.

We define a Weyl group invariant norm $\|\cdot\|_W$ on $\aaa_0$ by 
\[
\|\xi\|_W=\sup_{i=1,\ldots,n}\sup_{w\in W} |(w \xi)_i|.
\]
 Then, if $\xi=(\xi_1, \ldots, \xi_n)\in\aaa_0^+$, we clearly have
\[
 \|\xi\|_W= \max\{|\xi_1|,|\xi_n|\}.
\]
As in~\cite{CoNe01} we denote by $\langle\cdot, \cdot\rangle$ the inner product on $\aaa$ obtained from the Killing form, i.e., $\langle X, Y\rangle=2n(X_1Y_1+\ldots+X_nY_n)$ for $X=(X_1, \ldots, X_n), Y=(Y_1,\ldots, Y_n)\in\aaa$. 
For $\lambda\in \aaa^*$ let  $X_{\lambda}\in\aaa$ denote the element such that $\lambda(X)=\langle X_{\lambda}, X\rangle$ for all $X\in\aaa^*$. We then define the inner product $\langle\cdot, \cdot, \rangle$ on $\aaa^*$ by $\langle\lambda, \mu\rangle=\langle X_{\lambda}, X_{\mu}\rangle$ and set
\[
 \|\lambda\| = \big(\langle \Re\lambda, \Re\lambda\rangle+ \langle\Im \lambda, \Im\lambda\rangle\big)^{1/2}
\]
for $\lambda\in \aaa_{\C}^*$.

\subsection{Measures in the split case}\label{subsection_measures}
We need to fix measures on certain groups. We do this first in the split case for $\GL_n$ over local fields and the adeles of number fields.

\subsubsection*{Local measures}
Let $F$ be a finite extension of $\Q_p$ for some place $p\leq \infty$ of $\Q$ (with $\Q_p:=\R$ if $p=\infty$).
On the maximal compact group $\cpt_F$ we take the Haar measure which is normalised by giving $\cpt_F$ volume $1$.
Let $\psi_F: F\longrightarrow\C$ be the additive  character given by 
$
\psi_F(a)
=e^{2\pi i j_p\circ \tr_{F/\Q_p}(a)}
$,
where $\tr_{F/\Q_p}:F\longrightarrow \Q_p$ denotes the trace map for the extension $F$ over $\Q_p$ and composed with the homomorphism  $j_p:\Q_p\longrightarrow \Q_p/\Z_p\hookrightarrow\Q/\Z\hookrightarrow\R/\Z$ if $F$ is non-archimedean, cf.~\cite[Chapter XIV, \S 1]{La86}.

We then take the Haar measure on $F$  which is self-dual with respect to $\psi_F$. It is the usual Lebesgue measure if $F=\R$, twice the usual Lebesgue measure if $F=\C$, and gives the normalisation $\vol(\OOO_F)=\N_{F/\Q_p}(\DDD_F)^{-\frac{1}{2}}$ if $F$ is non-archimedean where $\DDD_F\subseteq \OOO_F$ denotes the local different of $F/\Q_p$, and $\N_{F/\Q_p}(\DDD_F)$ denotes the ideal norm of $\DDD_F\subseteq \OOO_F$.
The multiplicative measures on $F^{\times}$ are then fixed by
\[
d^{\times} x_F=
\begin{cases}
\frac{dx_F}{|x_F|_F}												&\text{if } F \text{ is archimedean},\\
\frac{q_F}{q_F-1}\frac{dx_F}{|x_F|_F}										&\text{if } F \text{ is non-archimedean},
\end{cases}
\]
so that $\vol(\OOO_F^{\times})= \N_{F/\Q_p}(\DDD_F)^{-\frac{1}{2}}$ in the non-archimedean case.

\subsubsection*{Global measures}
Let $\K$ be an arbitrary number field  of degree $d=[\K:\Q]$  with ring of adeles $\A_{\K}$.
Let
$\psi:\A_{\K}\longrightarrow\C$
be the product $\psi(x)=\prod_v\psi_{\F_v}(x_v)$, $x=(x_v)_v\in\A_{\K}$ where $v$ runs over all valuations of $\K$. 
Then $\psi$ is trivial on $\K$ so that we get a non-trivial character 
$\psi:\K\backslash \A_{\K}\longrightarrow\C$.
We take the product measures $dx=\prod_v dx_v$ and $d^{\times}x=\prod_v d^{\times}x_v$ on $\A_{\K}$ and $\A_{\K}^{\times}$, respectively.
We can embed $\R_{>0}$ into $\A_{\K}^{\times}$ via $\R_{>0}\ni t\mapsto (t^{1/d}, \ldots, t^{1/d},1,\ldots)\in\A_{\K}^{\times}$ where $(t^{1/d},\ldots,t^{1/d},1,\ldots)$ denotes the idele having the entry $t^{1/d}$ at every archimedean place, and $1$ at every non-archimedean place. Hence we get an isomorphism 
$\A_{\K}^{\times}\simeq \R_{>0}\times\A_{\K}^1$
that also fixes a measure $d^{\times}b$ on $\A_{\K}^1$ via
$d^{\times}x= d^{\times} b \,d^{\times} t$
for $d^{\times} x$ the previously defined measure on $\A_{\K}^{\times}$ and $d^{\times} t=t^{-1}dt$ the usual multiplicative Haar measure on $\R^{\geq0}$.
With this choice of measures we get
\[
\vol(\K\backslash \A_{\K})=1
\;
\text{ and }
\;\;
\vol(\K^{\times}\backslash \A_{\K}^{1})=\res_{s=1}\zeta_{\K}(s)
\]
where $\zeta_{\K}(s)$ is the Dedekind zeta function of $\K$
(cf.~\cite[Chapter XIV, \S 7, Proposition 9]{La86}).
This also fixes measures on $T_0(\A_{\K})$, $T_0(\A_{\K})^1$, $T_0(\K_v)$, $U(\A_{\K})$, and  $U(\K_v)$ for  $U$ the unipotent radical of any semi-standard parabolic subgroup 
by using the bases given by the coordinate entries of the matrices.

The measure on $G(\A_{\K})$ (and any of its Levi subgroups) is then defined via the Iwasawa decomposition $G(\A_{\K})=U_0(\A_{\K}) T_0(\A_{\K})\cpt$ such that for any integrable function $f:G(\A_{\K})\longrightarrow\C$ we have
\begin{align*}
\int_{G(\A_{\K})}f(g)dg
&=\int_{\cpt}\int_{T_0(\A_{\K})}\int_{U_0(\A_{\K})}\delta_{0}(m)^{-1}f(umk)\,du\,dm\,dk\\
&=\int_{\cpt}\int_{T_0(\A_{\K})}\int_{U_0(\A_{\K})}f(muk)\,du\,dm\,dk
\end{align*}
(similarly for $G(\K_v)=U_0(\K_v)T_0(\K_v)\cpt_v$), where  $\delta_0=\delta_{P_0}$ is the modulus function for the adjoint action of $T$ on $U_0$.
On $G(\A_{\K})^1$  we define a measure via the exact sequence  
\[
1\longrightarrow G(\A_{\K})^1\longrightarrow G(\A_{\K})\xrightarrow{\;g\mapsto|\det g|\;}\R_{>0}\longrightarrow 1.
\]
Using Eisenstein series similarly as in~\cite{La66}, one can compute that with our choice of measures one has
\begin{equation}\label{volume}
\vol(G(\K)\backslash G(\A_{\K})^1)
=D_{\K}^{\frac{n(n-1)}{4}}\res_{s=1}\zeta_{\K}(s) \cdot \zeta_{\K}(2)\cdot\ldots\cdot\zeta_{\K}(n).
\end{equation}
(We actually will not make use of this explicit formula).   
This also fixes measures on $M(\A)$, $M(\A)^1$, and $M(\F_v)$ for every $M\in\LLL$ and every place $v$ of $\F$.

\begin{rem}\label{rem:volume:compact:grp}
 Although we chose the Haar measure on $\cpt_v$ (resp.\ $\cpt$) such that $\cpt_v$ (resp.\ $\cpt$) has volume $1$, the volume of $\cpt_v$ (resp.\ $\cpt$) with respect to the measure on $G(\F_v)$ (resp.\ $G(\A)$) is in general not $1$: Instead the volume of $\cpt_v$  with respect to the measure on $G(\F_v)$  equals $\N_{\F_v/\Q_p} (\DDD_{\F_v})^{-\frac{n^2+n}{2}}$ if $v$ is non-archimedean. If we write $\vol(\cpt_v)$ (resp.\ $\vol(\cpt)$) we will always mean the volume with respect to the measure on $G(\F_v)$ (resp.\ $G(\A)$).
\end{rem}

\subsection{Measures in the quasi-split case}\label{subsec:meas:twisted}
We also need to define measures on certain quasi-split subgroups $H\subseteq \GL_n$, namely on subgroups of the following form: Suppose there exists integers $m_1, \ldots, m_t$ and fields $K_1, \ldots, K_t/E$ ($E$ either a local or global field, and accordingly $K_i$ local or global) such that 
\begin{equation}\label{eq:restr:scalars:group}
 H_{/E}=\Res_{K_1/E} \GL_{m_1}\times\ldots\times \Res_{K_t/E} \GL_{m_t}
\end{equation}
where $\Res_{K_i/E}\GL_{m_i}$ denotes the Weil restriction of scalars of $\GL_{m_i}$ (as a group over $K_i$) to $E$ (cf.~\cite[\S 1.3]{We82}). We write $H_{/E}$ to indicate that we consider $H$ as a group over the base field $E$.
The choice of $T_0$, $P_0$, and maximal compact subgroups in $\GL_{m_i}(K_i)$ if $E$ is local (resp.\ $\GL_{m_i}(\A_{K_i})$ if $E$ is global) correspond to respective objects in $H(E)$ (resp.\ $H(\A_E)$). The pullback of measures under the Weil restriction then defines measures on the semistandard parabolic and Levi subgroups in $H(E)$.

For example, if $\sigma\in \GL_n(\F)$ is a regular $\F$-elliptic element, and $H(\F)=G_{\sigma}(\F)$ is the centraliser of $\sigma$ in $G$, then there exists a number field $\K/\F$ of degree $[\K:\F]=n$ such that $H_{/\F}\simeq \Res_{\K/\F}\GL_{1/\K}$, and 
\[
 \vol(G_{\sigma}(\F)\backslash G_{\sigma}(\A)^1) = \res_{s=1} \zeta_{\K}(s).
\]

\subsection{Comparison of measures}\label{subsec:comparison:meas}
We will later need to compare our choice of $p$-adic measures to the motivic measures defined in~\cite{Gr97}.
Let $E$ be a non-archimedean field, and let $E$, $K_i$, and $H$ be related as in~\eqref{eq:restr:scalars:group}. 
 Gross~\cite{Gr97} defines canonical Haar measures on $H(E)$ and $\GL_{m_1}(K_1)\times\ldots\times \GL_{m_t}(K_t)$ which we denote for the moment by $\mu_0$ and $\mu_1\times\ldots\times\mu_t$, respectively. 
By~\cite[Proposition 6.6]{GaGr99} the pullback of $\mu_1\times\ldots\times\mu_t$ under the Weil restriction equals $\mu_0$. By the definition of our measure on $H(E)$ above, we only need to compare our measure on a group $\GL_{m_i}(K_i)$ with $\mu_i$.
By~\cite{Gr97}, the volume  of $\GL_{m_i}(\OOO_{K_i})$ with respect to $\mu_i$ is $1$, whereas $\GL_{m_i}(\OOO_{K_i})$ has measure $\N_{K_i/\Q_p}(\DDD_{K_i/\Q_p})^{-(m_i^2+m_i)/2}$ with respect to our measure.
Hence $\mu_0$ differs from our previously defined measure on $H(E)$ by $\prod_{i=1}^t \N_{K_i/\Q_p}(\DDD_{K_i/\Q_p})^{-(m_i^2+m_i)/2}$.

\section{Hecke algebras}\label{sec:hecke:alg}
\subsection{Unramified Hecke algebra, local case}
We fix some notation and recall some well-known facts about the unramified local Hecke algebra, see~\cite{Gr98} and~\cite[\S 2]{ShTe12} for details.
Let $p$ be a rational prime and  $F/\Q_p$ a finite extension.
Let $A^+_F$ denote the image of $X_*^+(T_0)$ in $T_0(F)$ under $\xi\mapsto\varpi_F^{\xi}:=\diag(\varpi_F^{\xi_1}, \ldots, \varpi_F^{\xi_n})$ for $\xi=(\xi_1, \ldots, \xi_n)\in X_*^+(T_0)$.
Since $\cpt_F=G(\OOO_F)$ is an admissible compact subgroup relative to $T_0(F)$, the Cartan decomposition holds, i.e.\ $G(F)=\cpt_F A^+_F \cpt_F$. 
Write $\HHH(\cpt_F):=C_c^{\infty}(\cpt_F\backslash G(F)/\cpt_F)$ for the unramified Hecke algebra of $G(F)$, i.e., the space of all locally constant, bi-$\cpt_F$-inariant, compactly supported functions $\tau: G(F)\longrightarrow\C$ with multiplication given by convolution. If $\xi\in X_*(T_0)$ let $\tau_{\xi}\in \HHH(\cpt_F)$ be the characteristic function of the double coset 
\begin{equation}\label{eq:def:double:coset}
\omega_{F, \xi}:=\cpt_F\varpi_F^{\xi}\cpt_F \subseteq G(F).
\end{equation}
The Satake transformation (or Cartan decomposition) gives an isomorphism 
\[
\HHH(\cpt_F)\xrightarrow{\;\;\; \simeq \;\;\;} \C[X_*(T_0)]^W
\]
of $\C$-algebras, where $\C[X_*(T_0)]^W$ denotes the elements fixed by the Weyl group $W$. In particular, the functions $\tau_{\xi+\lambda}$ for $\xi\in X_*^+(T_0)$ and $\lambda=(\lambda_0,\ldots, \lambda_0)$, $\lambda_0\in\Z$, generate $\HHH(\cpt_F)$ as a $\C$-algebra.

For $\tau\in \HHH(\cpt_F)$ let 
\[
 \deg\tau=\vol(\cpt_F)^{-1}\int_{G(F)} |\tau(x)|\,dx
\]
be the normalised $L^1$-norm (recall that $\vol(\cpt_F)$ denotes the volume of $\cpt_F$ with respect to the measure on $G(F)$). Note that $\deg\tau_{\xi+\lambda}=\deg\tau_{\xi}$ for all $\lambda=(\lambda_0,\ldots,\lambda_0)$, $\lambda_0\in\Z$, and $\deg\tau_\xi=\deg\tau_{w\xi}$ for all $w\in W$.
The multiplication of two functions $\tau_{\xi}$, $\tau_{\mu}$ for $\xi, \mu\in X_*^+(T_0)$ can also be expressed by (cf.~\cite[(2.9)]{Gr98})
\[
\tau_{\xi}*\tau_{\zeta}
=\tau_{\zeta}*\tau_{\xi}
=\sum_{\substack{\nu\in X_*^+(T_0): \\ \nu\leq \zeta+\xi}} n_{\nu}(\xi,\zeta) \tau_{\nu}
\]
with $n_{\nu}(\xi, \zeta)=n_{\nu}(\zeta, \xi)$ certain non-negative constants given by
\[
 n_{\nu}(\xi, \zeta)
 =\tau_{\xi}*\tau_{\zeta}(\varpi_F^\nu)
\]
which are bounded  by 
\[
n_{\nu}(\xi, \zeta)
\leq  \vol(\cpt_F) \min\{\deg \tau_{\xi}, \deg\tau_{\zeta}\}.
\]

For an integer $\kappa\geq0$ let $\HHH^{\leq\kappa}(\cpt_F)$  be the vector subspace of $\HHH(\cpt_F)$ generated (as a $\C$-vector space) by the functions $\tau_{\xi}$ with $\xi\in X_*(T_0)$, $\|\xi\|_W\leq \kappa$. (Recall that $\|\cdot\|_W$ is Weyl group invariant so that with $\tau_{\xi}$ also $\tau_{w\xi}$ is an element of $\HHH^{\leq\kappa}(\cpt_F)$ for every $w\in W$.) 
Moreover, we write $\HHH^{\leq\kappa,\geq0}(\cpt_F)$ for the $\C$-vector subspace of $\HHH^{\leq \kappa}(\cpt_F)$ generated by all $\tau_{\xi}$ with $\|\xi\|_W\leq \kappa$ and $\xi_i\geq0$ for all $i=1, \ldots, n$.

\begin{lemma}\label{bounded_by_degree}
\begin{enumerate}[label=(\roman{*})]
  \item\label{integral_supp} If $\xi\in X_*^+(T_0)$, then $\supp\tau_\xi\subseteq \varpi_F^{\xi_n}\Mat_{n\times n} (\OOO_F)$, where $\Mat_{n\times n} (\OOO_F)$ is the set of $n\times n$-matrices with integral entries. 

\item There exist $\eps, \eta>0$ depending only on $n$ such that for all $\xi\in X_*^+(T_0)$ we have
\[
\deg\tau_\xi\le q_F^{A(\xi_1-\xi_n)}\le (\deg\tau_\xi)^B.
\]

\end{enumerate}
\end{lemma}
\begin{proof}
\begin{enumerate}[label=(\roman{*})]
 \item This is clear from the definition.
\item It suffices to consider $\xi\in X_*^+(T_0)$ with $\deg\tau_\xi>1$ (the case $\deg\tau_\xi=1$ is trivial). 
By~\cite[Proposition 7.4]{Gr98},
$\deg\tau_{\xi}= q_F^{2\langle\rho, \xi\rangle}\frac{Q_1(q_F)}{Q_2(q_F)}$
for $Q_1, Q_2$ suitable polynomials depending only on $n$ but not on $F$ both having the same degree.
Hence
\[
a^{-1}\deg\tau_{\xi}\le  q_F^{2\langle\rho, \xi\rangle}\leq a\deg\tau_{\xi}
\]
for some constant  $a>0$ depending only on $n$. The degree $\deg\tau_{\xi}$ equals an integer $\geq2$ so that we can choose $b>0$ with $a\leq 2^b$. 
Hence 
\[
(\deg\tau_{\xi})^{1-b}
\le q_F^{2\langle\rho, \xi\rangle}\leq (\deg\tau_{\xi})^{b+1}.
\]
Now 
\[
 2\langle\rho, \xi\rangle =\sum_{1\leq i< j\le n} (\xi_i-\xi_j)
 \ge \xi_1-\xi_n
\]
because of $\xi\in X_*^+(T_0)$. For the same reason, $ 2\langle\rho, \xi\rangle\le (n(n-1)/2) (\xi_1-\xi_n)$. This finishes the assertion.
\end{enumerate}
\end{proof}

\subsection{Global Hecke algebra}\label{subsec:general:hecke:alg}
Let $\HHH(\cpt_f):=C_c^{\infty}(\cpt_f\backslash G(\A_f)/\cpt_f)$ be the unramified Hecke algebra of  $G(\A_f)$. 
Suppose we are given $\xi_v\in X_*^+(T_0)$ for every non-archimedean place $v$ of $\F$ such that $\xi_v=0$ for almost all $v$. Let $\bxi=(\xi_v)_{v<\infty}$ and write
\[
 \tau_{\bxi}:=\prod_{v<\infty} \tau_{\xi_v}\in \bigotimes_{v<\infty}\HHH(\cpt_{\F_v})\subseteq \HHH(\cpt_f).
\]
The functions $\tau_{\bxi}$ generate $\HHH(\cpt_f)$ as a $\C$-algebra.  Similarly as in the local case we define the degree of $\tau\in \HHH(\cpt_f)$ as
\[
 \deg\tau=\vol(\cpt_f)^{-1}\int_{G(\A_f)} |\tau(x)|\,dx.
\]
If $\tau$ factorises into a product $\prod_{v<\infty}\tau_v$ with $\tau_v\in\HHH(\cpt_{\F_v})$ a function on $G(\F_v)$, then $\deg\tau=\prod_{v<\infty}\deg\tau_v$.

Suppose $K_f\subseteq \cpt_f$ is a subgroup of finite index.
Consider the Hecke algebra $\HHH(K_f):=C_c^{\infty}(K_f\backslash G(\A_f)/K_f)$ of $G(\A_f)$ with respect to $K_f$.
For  $g\in G(\A_f)$ let $\chi_{K_f, g}\in \HHH(K_f)$ be the characteristic function of the set $K_fgK_f$ normalised by $\vol(K_f)^{-1}$ (with respect to the measure on $G(\A_f)$). 
The functions $\chi_{K_f,g}$, $g\in G(\A_f)$, clearly generate $\HHH(K_f)$ as a $\C$-algebra. 

Suppose $g\in G(\A_f)$ and $\bxi$ are such that $g_v\in\cpt_v \varpi_v^{\xi_v} \cpt_v$ for all non-archimedean places $v$. Then clearly for every $x\in G(\A_f)$
\[
 \chi_{K_f, g} (x)
\leq \vol(K_f)^{-1} \tau_{\bxi}(x)
=\vol(\cpt_f)^{-1} [\cpt_f:K_f] \tau_{\bxi}(x).
\]
Moreover,
\[
\|\chi_{K_f, g}\|_{L^1(G(\A_f))}
\leq [\cpt_f:K_f] \deg\tau_{\bxi}
\leq [\cpt_f:K_f]^2 \|\chi_{K_f, g}\|_{L^1(G(\A_f))},
\]
and
\[
\frac{ \vol(K_fgK_f) }{\vol(\cpt_f)}
= \frac{\vol(K_f)}{\vol(\cpt_f)} \|\chi_{K_f, g}\|_{L^1(G(\A_f))}
= \frac{\|\chi_{K_f, g}\|_{L^1(G(\A_f))}}{[\cpt_f:K_f]}.
\]

\section{Spherical functions on $\GL_n(\C)^1$}\label{sec:spherical:fcts}
\subsection{Spherical Plancherel measure}
Let $\bfc:\aaa_{\C}^*\longrightarrow\C$ denote Harish-Chandra's $\bfc$-function.
By~\cite[Chapter IV, Theorem 5.7]{HelgGrGeom} it can be written as
\[
 \bfc(\lambda)
=\frac{\pi(\rho)}{\pi(\lambda)},
\]
where for $\lambda\in\aaa_{\C}^*$ we set
\[
 \pi(\lambda)=\prod_{\alpha\in\Phi^+} \langle\alpha, \lambda\rangle,
\]
and $\rho=\frac{1}{2}\sum_{\alpha\in\Phi^+}\alpha\in\aaa^*$ is the half-sum of all positive roots.
The spherical Plancherel measure for $\GL_n(\C)^1$ then equals $\bfc(\lambda)^{-2}d\lambda$.
We need to collect a few auxiliary estimates for $\bfc(\lambda)$ similar to those given in~\cite[pp. 128-129]{LaMu09} for the spherical Plancherel measure for $\GL_n(\R)^1$. 
First of all it is clear that for all $\lambda\in i\aaa^*$ we have
\[
 \bfc(\lambda)^{-2}
\ll\hat\beta(\lambda)
:=\prod_{\alpha\in\Phi^+}(1+|\langle \alpha, \lambda\rangle|)^2.
\]
In particular, writing $d:=\dim_{\R}\SL_n(\C)/\SU(n)= n^2-1$ and $r:=\dim\aaa=n-1$, we have
\begin{align}
\bfc(\lambda)^{-2}			&	\ll(1+\|\lambda\|)^{d-r}=(1+\|\lambda\|)^{2|\Phi^+|},  \;\;\;\;\text{ and }\label{eq:est:plancherel} \\
 D_{\xi}\bfc(\lambda)^{-2} 		&	\ll_{\xi} (1+\|\lambda\|)^{d-r-1}=(1+\|\lambda\|)^{2|\Phi^+|-1}
\end{align}
for $D_{\xi}\bfc^{-2}$ the directional derivative of $\bfc^{-2}$ along any $\xi\in\aaa_{\C}^*$. 
As in~\cite[p. 128]{LaMu09} we define another auxiliary function depending on a real parameter $t\geq 1$ and $\lambda\in i\aaa^*$ by
\begin{equation}\label{eq:def:of:tildebeta}
 \hat\beta(t, \lambda)= \prod_{\alpha\in\Phi^+} (t+|\langle\alpha, \lambda\rangle|)^2
\end{equation}
so that $\hat\beta(1, \lambda)=\hat\beta(\lambda)$, and
\[
 \hat\beta(t, \lambda)  \ll (t+\|\lambda\|)^{d-r}\;\text{ and }\;\;
\hat\beta(t, \lambda_1+\lambda_2)  \ll\hat\beta(t+\|\lambda_1\|, \lambda_2)
\]
for all $\lambda, \lambda_1, \lambda_2\in i\aaa^*$, $t\geq1$.

Suppose $M\subseteq G$ is a semistandard Levi subgroup, that is, $M\simeq\GL_{n_1}\times\ldots\times\GL_{n_t}$ for suitable $n_1,\ldots, n_t\geq1$ with $n_1+\ldots+n_t=n$ for some $1\le t\le n$. 
We denote by $\bfc^M$ the Harish-Chandra $\bfc$-function for $M(\C)$ on $\aaa_{\C}^{M,*}$ where $\aaa^M:=\aaa_0^M$.
Then the previous estimates and definitions hold mutatis mutandis for $M$.
\begin{lemma}
 If $M\subsetneq G$ is a semistandard Levi subgroup, then
\[
 \hat\beta^M(\lambda^M)(1+\|\lambda\|) 
\ll \hat\beta(\lambda)
\]
for all $\lambda\in i\aaa^*$ where we write $\lambda=\lambda^M+\lambda_M\in i\aaa^{M,*}\oplus i\aaa_M^*$. (Here $\aaa_M^*$ is the orthogonal complement of $\aaa^{M, *}$ in $\aaa^*$ with respect to $\langle\cdot, \cdot\rangle$.)
\end{lemma}
\begin{proof}
 The estimate is the analogue of~\cite[Lemma 3.1]{LaMu09} for $\GL_n(\C)^1$ instead of $\GL_n(\R)^1$, and the proof of~\cite[Lemma 3.1]{LaMu09} holds mutatis mutandis in our complex situation.
\end{proof}

\subsection{Elementary spherical functions}\label{subsec:sphericalfcts}
The elementary spherical function $\phi_{\lambda}$ on $G(\C)^1$ with parameter $\lambda\in\aaa_{\C}^*$  is by Harish-Chandra's formula defined by
\[
 \phi_{\lambda}(x)=\int_{\cpt_{\C}} e^{(\lambda+\rho)(H_0(kx))} dk.
\]
This formula (with the necessary changes) is valid for any semisimple Lie group. Note that $G(\C)^1$ is not semisimple, but $G(\C)^1/(G(\C)^1\cap Z(\C))$ is semisimple and $G(\C)^1\cap Z(\C)$ is contained in $\cpt_\C=\Unit(n)$ with measure $1$ so that the integral formula does indeed also hold for spherical functions on $G(\C)^1$. We extend $\phi_\lambda$ trivially to $G(\C)$. An analogous formula therefore also holds for the spherical functions on $M(\C)^1$ for any semistandard Levi subgroup $M\subseteq G$, and we again extend these functions trivially to all of $M(\C)$.
The spherical function satisfies
\[
 \phi_{\lambda}(1)=1,\;\text{ and }\;
 \big|\phi_{\lambda}(x)\big|\leq 1
\]
for all $\lambda\in i\aaa^*$ and $x\in G(\C)$, cf.~\cite[\S 2-\S 3]{DKV79}.
The function $\phi_{\lambda}$ can also be written as (cf.~\cite[Chapter IV, Theorem 5.7]{HelgGrGeom})
\begin{equation}\label{eq:second_exp_sphfct}
 \phi_{\lambda}(e^X)
= \frac{\pi(\rho)}{\pi(\lambda)}\frac{\sum_{\sigma\in W} \det(\sigma) e^{\sigma\lambda( X)}}{\sum_{\sigma\in W} \det(\sigma) e^{\sigma\rho( X)} }
= \bfc(\lambda)\frac{\sum_{\sigma\in W} \det(\sigma) e^{\sigma\lambda( X)}}{\sum_{\sigma\in W} \det(\sigma) e^{\sigma\rho( X)} }
\end{equation}
for $X\in \aaa$.
Note that by the proof of that theorem in~\cite{HelgGrGeom} we have
\begin{equation}\label{eq:sinh:det}
 \sum_{\sigma\in W} \det(\sigma) e^{\sigma\rho(X)} 
=\prod_{\alpha\in\Phi^+} 2\sinh\alpha(X)
\end{equation}
 so that
\[
 \phi_{\lambda}(e^X)
= \bfc(\lambda)2^{|\Phi^+|}\frac{\sum_{\sigma\in W} \det(\sigma) e^{\sigma\lambda( X)}}{ \prod_{\alpha\in\Phi^+} \sinh\alpha(X)} 
\]
if $\alpha(X)\neq0$ for all $\alpha\in\Phi^+$.

We need a descent formula for $\phi_{\lambda}(e^X)$ similar to the one given in~\cite[Proposition 2.3]{CoNe01} (cf.\ also~\cite[Theorem 2.2.8]{AnJi99}). 
We loosely follow the notation of~\cite{CoNe01}: Let $\tilde\Delta\subseteq \Delta_0$ be a subset. It generates a subset $\tilde\Phi^+\subseteq\Phi^+$ of positive roots with corresponding subgroup $\tilde G\subseteq G$ containing $T_0$ such that $\tilde\Phi^+$ is a system of positive roots for $(\tilde G, T_0)$.
 Then $\tilde G$ is a semi-standard Levi subgroup of $G$, and we denote by $\tilde\aaa\subseteq \aaa$, $\tilde\aaa^*\subseteq \aaa^*$ the subspaces built analogously as before but with $G$ replaced by $\tilde G$. In particular, $\tilde\aaa^*$ is the $\R$-span of $\tilde\Delta$. Further define
\[
\Phi_1^+:=\Phi^+\minus\tilde\Phi^+,\;\;
\Delta_1=\Delta_0\minus\tilde\Delta,\;\;
\tilde\rho:=\frac{1}{2}\sum_{\alpha\in\tilde\Phi^+} \alpha,\;\text{and}\;\;
\rho_1:=\frac{1}{2}\sum_{\alpha\in \Phi^+_1} \alpha=\rho-\tilde\rho.
\]
 Let $\aaa_1$ be the orthogonal complement of $\tilde\aaa$ in $\aaa$, and $\aaa_1^*$ the orthogonal complement of $\tilde\aaa^*$ in $\aaa^*$ (both with respect to $\langle\cdot, \cdot\rangle$). If $X\in\aaa_{\C}$ and $\lambda\in \aaa^*_{\C}$, we write $X=\tilde X + X_1$ and $\lambda=\tilde\lambda+\lambda_1$ for the respective decompositions. Let $\tilde W\subseteq W$ be the Weyl group of $\tilde G$ with respect to the torus $T_0\subseteq \tilde G$. Then $\tilde W$ leaves $\lambda_1$ invariant for every $\lambda\in\aaa^*_{\C}$ so that $\tilde W$ acts trivially on $\aaa_{1,\C}^*$.
For $\tilde\lambda\in\tilde\aaa^*_{\C}$ let $\phi^{\tilde G}_{\tilde\lambda}$  denote the spherical function for the parameter $\tilde\lambda$ with respect to $\tilde G$.

For $\lambda\in\aaa_{\C}^*$ we set 
\[
 \tilde\pi(\lambda)=\prod_{\alpha\in\tilde\Phi^+} \langle \alpha, \lambda\rangle
\]
and
\[
 \tilde{\bfc}(\lambda)
= \frac{\tilde\pi(\rho)}{\tilde\pi(\lambda)}.
\]
Note that similarly as for $\bfc$, we have
\begin{equation}\label{eq:growth:tildec-fct}
 \tilde{\bfc}(\lambda)^{-1} 
\ll (1+\|\lambda\|)^{|\tilde\Phi^+|} .
\end{equation}

\begin{lemma}\label{lemma:descent:spherical:fct}
\begin{enumerate}[label=(\roman{*})]
 \item  Suppose $X\in \aaa_{\C}$ is such that $\alpha(X)\neq0$ for all $\alpha\in\Phi^+_1$. Then for every $\lambda\in\aaa^*_{\C}$
 we have
\begin{equation}\label{eq:descent:formula:spherical}
\bfc(\lambda)^{-1} \phi_{\lambda}(e^X)
= \frac{2^{|\Phi^+_1|}}{|\tilde W|} 
\Big(\prod_{\alpha\in\Phi^+_1}\frac{1}{\sinh\alpha(X)}\Big)
\sum_{\sigma\in W} \det(\sigma) e^{(\sigma\lambda)_1(X_1)} ~ \phi_{\widetilde{\sigma\lambda}}^{\tilde G} (e^{\tilde X})
\tilde{\bfc}(\sigma\lambda)^{-1}.
\end{equation}

\item
 Suppose $X\in \aaa_{\C}$ is as in the first part, but additionally assume that $\alpha(X)=0$ for all $\alpha\in\tilde\Phi^+$. Then for every $\lambda\in\aaa^*_{\C}$ we have
\[
\bfc(\lambda)^{-1} \phi_{\lambda}(e^X)
= \frac{2^{|\Phi^+_1|}}{|\tilde W|} 
\Big(\prod_{\alpha\in\Phi^+_1}\frac{1}{\sinh\alpha(X)}\Big)
\sum_{\sigma\in W} \det(\sigma) e^{(\sigma\lambda)_1(X_1)} \tilde{\bfc}(\sigma\lambda)^{-1}.
\]
\end{enumerate}
\end{lemma}
\begin{rem}
 The formula~\eqref{eq:descent:formula:spherical} is similar to~\cite[Proposition 2.3]{CoNe01} (cf.\ also~\cite[Theorem 2.2.8]{AnJi99}). 
However our emphasis lies on the fact that we can control the behaviour of the function in dependence of the degeneracy of $X$ for \emph{all $\lambda$} at the same time, whereas~\cite[Proposition 2.3]{CoNe01} controls the behaviour of $\phi_{\lambda}(e^X)$ in dependence of the degeneracy of $\lambda$.
\end{rem}

\begin{proof}[Proof of Lemma~\ref{lemma:descent:spherical:fct}]
 Suppose first that $\lambda\in\aaa_{\C}^*$ and $X\in\aaa_{\C}$ are regular, i.e., $\langle\alpha, \lambda\rangle\neq0$ and $\alpha(X)\neq0$ for all $\alpha\in\Phi^+$. 
Note that $\sum_{\sigma\in W}\sum_{\tau\in\tilde W} Q(\sigma\tau)=|\tilde W| \sum_{\sigma\in W} Q(\sigma)$ for every function $Q$ on $W$ (cf.\ the proof of~\cite[Proposition 2.3]{CoNe01}). Using this together with~\eqref{eq:second_exp_sphfct} and~\eqref{eq:sinh:det} we can compute
\begin{align*}
 \phi_{\lambda}(e^X)
& = \frac{2^{|\Phi^+|}}{|\tilde W|} \sum_{\sigma\in W} \frac{\sum_{\tau\in\tilde W} \det(\tau\sigma) e^{\tau\sigma\lambda (X)}}{\prod_{\alpha\in\Phi^+} \sinh\alpha(H)}\prod_{\alpha\in\Phi^+} \frac{\langle \alpha, \rho\rangle}{\langle \alpha, \lambda\rangle} \\
& =\frac{2^{|\Phi^+|}}{|\tilde W|}
\sum_{\sigma\in W}
\frac{\sum_{\tau\in\tilde W} \det(\tau) e^{\tau(\sigma\lambda)(X)}}{\prod_{\alpha\in\Phi^+}\sinh\alpha(X)} 
\Big(\prod_{\alpha\in\tilde\Phi^+} \frac{\langle\alpha, \rho\rangle}{\langle \alpha,\sigma \lambda\rangle} \Big)
\Big(\prod_{\alpha\in\Phi_1^+}\frac{\langle \alpha, \rho\rangle}{\langle\alpha, \sigma\lambda\rangle}\Big),
\end{align*}
where we used that $\bfc(\sigma\lambda)=\det(\sigma)\bfc(\lambda)$, i.e.\ $\pi(\lambda)=\det(\sigma)\pi(\sigma\lambda)$ (cf.~\cite[p. 433]{HelgGrGeom}).
Now (cf.~\cite[p. 914]{CoNe01})
\[
 e^{\tau(\sigma\lambda)(X)} = e^{\tau\widetilde{(\sigma\lambda)} (\tilde X)} e^{(\sigma\lambda)_1(X_1)}  ,
\]
and $\langle \alpha, \sigma\lambda\rangle=\langle \alpha, \widetilde{(\sigma\lambda)}\rangle$ for all $\alpha\in \tilde\Phi^+$ 
so that $\phi_{\lambda}(e^X)$ equals
\begin{align*}
& \frac{2^{|\Phi^+_1|}}{|\tilde W|} \frac{1}{\prod_{\alpha\in\Phi^+_1}\sinh\alpha(X)}
\sum_{\sigma\in W} e^{(\sigma\lambda)_1(X_1)} 
\frac{2^{|\tilde\Phi^+|}\sum_{\tau\in\tilde W} \det(\tau) e^{\tau\widetilde{(\sigma\lambda)}(\tilde X)}}{\prod_{\alpha\in\tilde\Phi^+}\sinh\alpha(X)}
\Big( \prod_{\alpha\in\tilde\Phi^+} \frac{\langle\alpha, \widetilde{\rho}\rangle}{\langle \alpha, \widetilde{(\sigma\lambda)}\rangle}\Big)
\Big(\prod_{\alpha\in\Phi_1^+}\frac{\langle \alpha, \rho\rangle}{\langle\alpha, \sigma\lambda\rangle}\Big)\\
= & \frac{2^{|\Phi^+_1|}}{|\tilde W|} \frac{1}{\prod_{\alpha\in\Phi^+_1}\sinh\alpha(X)}
\sum_{\sigma\in W} e^{(\sigma\lambda)_1(X_1)} ~ \phi_{\widetilde{\sigma\lambda}}^{\tilde G} (e^{\tilde X})
\Big(\prod_{\alpha\in\Phi_1^+}\frac{\langle \alpha, \rho\rangle}{\langle\alpha, \sigma\lambda\rangle}\Big).
\end{align*}
Multiplying with  $\bfc(\lambda)^{-1}$, the function $ \bfc(\lambda)^{-1} \phi_{\lambda}(e^X)$ equals
\begin{align*}
& \frac{2^{|\Phi^+_1|}}{|\tilde W| \prod_{\alpha\in\tilde\Phi^+}\langle\alpha, \rho\rangle} 
\Big(\prod_{\alpha\in\Phi^+_1}\frac{1}{\sinh\alpha(X)}\Big)
\sum_{\sigma\in W} e^{(\sigma\lambda)_1(X_1)} ~ \phi_{\widetilde{\sigma\lambda}}^{\tilde G} (e^{\tilde X})\frac{\pi(\lambda)}{\prod_{\alpha\in\Phi_1^+} \langle\alpha, \sigma\lambda\rangle}\\
= & \frac{2^{|\Phi^+_1|}}{|\tilde W|} 
\Big(\prod_{\alpha\in\Phi^+_1}\frac{1}{\sinh\alpha(X)}\Big)
\sum_{\sigma\in W} \det(\sigma) e^{(\sigma\lambda)_1(X_1)} ~ \phi_{\widetilde{\sigma\lambda}}^{\tilde G} (e^{\tilde X})\tilde{\bfc}(\sigma\lambda)^{-1},
\end{align*}
where we used $\pi(\lambda)=\det(\sigma)\pi(\sigma\lambda)$ again.
This last expression is of course also valid for every (possibly singular) $\lambda\in\aaa^*_{\C}$ and every $X\in\aaa_{\C}$ with $\alpha(X)\neq0$ for all $\alpha\in\Phi_1^+$ so that the first part of the lemma is proven.

For the second part we have $\tilde X=0$ by assumption on $X$,  and therefore  $\phi^{\tilde G}_{\widetilde{(\sigma\lambda)}}(1)=1$ for every $\widetilde{(\sigma\lambda)}$ so that the second part also follows.
\end{proof}

\begin{cor}\label{cor:est:spherical_times_plancherel}
 Suppose $X\in\aaa_{\C}$ is as in the first part of Lemma~\ref{lemma:descent:spherical:fct}, i.e., $\prod_{\alpha\in\Phi^+_1}\sinh\alpha(X)\neq0$. Then
\[
 \big| \bfc(\lambda)^{-2} \phi_{\lambda}(e^X)\big|
\ll \big(1+\|\lambda\|\big)^{d-r-|\Phi_1^+|} \prod_{\alpha\in\Phi^+_1}\big|\sinh\alpha(X)\big|^{-1}
\]
for all $\lambda\in i\aaa^*$.
\end{cor}
\begin{proof}
We use the descent formula from Lemma~\ref{lemma:descent:spherical:fct}: Since $\lambda\in i\aaa^*$, we have $|e^{(\sigma\lambda)_1(X_1)}|=1$ and $|\phi_{\widetilde{\sigma\lambda}}^{\tilde G} (e^{\tilde X})|\leq~1$ for all $\sigma\in W$ and $X\in\aaa_{\C}$. Together with~\eqref{eq:est:plancherel} and~\eqref{eq:growth:tildec-fct} the assertion follows. 
\end{proof}

\subsection{Spherical functions and the Paley-Wiener Theorem}\label{subsec:paley:wiener}
 We need to introduce some further notation: Let $R>0$.
\begin{itemize}
 \item $C_R^{\infty}(\aaa)=$ space of smooth functions on $\aaa$ supported in
\[
V(R):=\{X\in\aaa\mid \|X\|_W\leq R\}\subseteq \aaa.
\]

\item  $C_R^{\infty}(G(\C)^1\sslash\cpt_{\C})=$  space of all spherical functions on $A_G\backslash G(\C)\simeq G(\C)^1$, i.e., all bi-$\cpt_\C$-invariant functions, supported in $\cpt_{\C} e^{V(R)} \cpt_{\C}$. 

\item $\PPP^R(\aaa_{\C}^*)=$  space of holomorphic functions $F:\aaa_{\C}^*\longrightarrow\C$ such that for every $N\in\N$ there exists $C=C_N>0$ with
\[
 \forall~\lambda\in\aaa_{\C}^*:\;\;\; |F(\lambda)|\leq C_N (1+\|\lambda\|)^{-N} e^{R\|\Re\lambda\|}.
\]
\end{itemize}
We denote by $C_R^{\infty}(\aaa)^W$, $C_c^{\infty}(\aaa)^W$, $\PPP^R(\aaa_{\C}^*)^W$ the respective subspaces of $W$-invariant elements.

The Paley-Wiener theorem relates these three spaces: It gives a commutative diagram of algebra isomorphisms (cf.~\cite[p. 41]{DKV79},~\cite[Theorem 1]{LaMu09}):
\[
 \xymatrixcolsep{8pc}\xymatrixrowsep{4pc}\xymatrix{
 C_R^{\infty}( G(\C)^1\sslash\cpt_{\C})		\ar@<0.5ex>[r]^-{\AAA}				\ar[d]^{\HHH}
&	C_R^{\infty}(\aaa)^{W}				 	\ar@<0.5ex>[l]^-{\BBB} 				\ar[dl]^{\wedge}		\\
\PPP^R(\aaa_{\C}^*)^W							
}
\]
where the homomorphisms are defined as follows:

\begin{itemize}
 \item $\AAA$ is the Abel transform
\[
\AAA(f)(X)
=\delta_{0}(\exp X)^{\frac{1}{2}}\int_{U_0(\C)}f(\exp X u)\,du,
\;\;\;\;\;
X\in\aaa.
\]

\item $\BBB=\AAA^{-1}$ is the inverse of $\AAA$,
\[
\BBB(h)(x)
=\frac{1}{|W|}\int_{i\aaa^*}\hat{h}(\lambda) \beta(\lambda) \phi_{-\lambda}( x) \,d\lambda,
\;\;\;\;
x\in G(\C).
\]

\item $h\mapsto \hat{h}$ denotes the usual Fourier transform.

\item $\HHH$ is the spherical Fourier transform
\[
 \HHH(f)(\lambda)=\int_{G(\C)^1}f(x) \phi_{\lambda}(x) dx, \;\;\;\;\; \lambda\in\aaa^*_{\C}.
\]
\end{itemize}

\section{Test functions}\label{sec:test:fcts}
We are now ready to define our space of test functions. Fix a positive real number $R$ for the rest of the paper.
If $\bkappa:=(\kappa_v)_{v<\infty}$ is a sequence of non-negative integers with $\kappa_v=0$ for almost all $v$, let $C_{R, \bkappa}^{\infty}(G(\A)^1)$ denote the space of compactly supported, smooth, bi-$\cpt$-invariant functions $f :G(\A)^1\longrightarrow\C$ satisfying 
\[
f\big|_{G(\F_v)}\in\HHH^{\leq \kappa_v, \ge 0}(\cpt_{\F_v})
\]
 for all non-archimedean valuations $v$, and 
\[
\supp f\cap G(\F_v)^1\subseteq \cpt_{\F_v} e^{V(R)} \cpt_{\F_v}=:\Xi_R\subseteq G(\F_v)^1= G(\C)^1
\]
if $v$ is archimedean.
In particular, if $f\in C_{R,\bkappa}^{\infty}(G(\A)^1)$ and $v$ is non-archimedean, then $\supp f\cap G(\F_v)$ is contained in the finite union 
\[
\bigcup_{\mu\in X_*^+(T_0):0\le\xi_n\le\ldots\le\xi_n\leq \kappa_v}\cpt_{\F_v} \varpi_v^{\mu}\cpt_{\F_v}.
\]

We are usually not interested in a general test function in $C_{R,\bkappa}^{\infty}(G(\A)^1)$, but will consider more specific elements:
Let 
\begin{equation}\label{eq:places:dividing:kappa}
 S_{\bkappa}=\{v\mid \kappa_v\neq0\}.
\end{equation}
For the archimedean place, we take  a non-trivial function $h\in C_R^{\infty}(\aaa)^W$ (usually with $h(0)=1$) and define the family of functions depending on the parameter $\mu\in\aaa_{\C}^*$ by
\[
f^{\mu}_{\C}(g):=\BBB(h_{\mu})(g)
\]
for $g\in G(\C)$. Here $h_{\mu}(X):=h(X) e^{\langle \mu, X\rangle}$ for $X\in\aaa$. We further define for $t\geq1$ the family of functions $h_{\mu, t}\in C_c^{\infty}(\aaa)$ by $h_{\mu, t}(X)= t^{n-1}h(tX)e^{\langle \mu, X\rangle} $ and put $f^{t, \mu}_{\C}(g):=\BBB(h_{t, \mu})(g)$ so that $h_{0,\mu}=h_{\mu}$ and $f_{\C}^{0,\mu}=f_{\C}^{\mu}$.

For the non-archimedean part suppose that $\bxi=(\bxi_v)_{v<\infty}$ is a sequence with $\xi_v\in X_*^+(T_0)$ and $\|\xi_v\|_W\leq \kappa_v$ for all $v$. 
We then define global test functions by
\begin{equation}\label{eq:def:global:test:fct}
 f_{\bxi}^{t, \mu}=\big(f_{\C}^{t, \mu}\cdot\tau_{\bxi}\big)\bigg|_{G(\A)^1} \in C_{R, \bkappa}^{\infty}(G(\A)^1)
\end{equation}
(which implicitly depend on $h$ of course), and set $f_{\bxi}^{\mu}=f_{\bxi}^{1, \mu}$.
It will sometimes also be useful to identify this with the functions 
\[
 f_{\bxi}^{t,\mu}
=\Big(f_{\C}^{t,\mu}\cdot \prod_{v\in S\minus S_{\infty}} \tau_{\xi_v}\Big)\bigg|_{G(\F_S)^1}\in C_{R, \bkappa}^{\infty}(G(\F_S)^1),
\]
if $S$ is a sufficiently large finite set of places containing $S_{\bkappa}$ and the archimedean places.
(Note that $\prod_{v\not\in S} \tau_{\xi_v}=\One_{\cpt^{S}}$ is the characteristic function of $\cpt^{S}\subseteq G(\A^S)$.)

\begin{rem}
 For a general number field the test functions can be chosen in a similar manner, cf.~\cite{proceedings}.
\end{rem}

\section{The distributions $J_{\ooo}(f)$ and weighted orbital integrals}\label{sec:fine:geom:exp}
As explained in the introduction we want to study the integral $\int_{t\Omega} J_{\text{geom}}(f_{\C}^{\mu}\cdot\tau)\,d\mu$ by using the fine geometric expansion for $J_{\text{geom}}$.
The goal of this section is to reduce the analysis of the geometric side of the trace formula to problems~\ref{arch} and~\ref{non-arch} from the introduction.

\subsection{The coarse and fine geometric expansion}\label{subsec:fine:geom:exp}
Suppose $\gamma_1, \gamma_2\in G(\F)$ and let $\gamma_i=\sigma_i\nu_i=\nu_i\sigma_i$ be their Jordan decomposition with $\sigma_i$ semisimple and $\nu_i$ unipotent. We call $\gamma_1$ and $\gamma_2$ equivalent if $\sigma_1$ and $\sigma_2$ are conjugate in $G(\F)$, cf.~\cite[\S 10]{Ar05}.
Then 
\[
 \OOO\longrightarrow \{\text{conj.-classes in }G(\F)_{ss}\},\;\;\;
\ooo\mapsto \text{conj.-cl. of }\sigma\text{ s.t. } \sigma\in\ooo
\]
is a bijection, where $G(\F)_{ss}$ denotes the set of semisimple elements in $G(\F)$.
 Equivalently, $\OOO$ is in bijection with the set of monic polynomials of degree $n$ with $\F$-rational coefficients by mapping $\ooo$ to the characteristic polynomial of $\sigma\in\ooo$.

The geometric side of the trace formula has an expansion parametrised by elements in $\OOO$, namely the so-called coarse geometric expansion
\begin{equation}\label{eq:coarse:geom:exp}
 J_{\text{geom}}(f)=\sum_{\ooo\in\OOO} J_{\ooo}(f),
\end{equation}
where $J_{\ooo}: C_c^{\infty}(G(\A)^1)\longrightarrow\C$ are suitable distributions (cf.~\cite{Ar05}).

\begin{definition}
Let $\OOO_{R, \bkappa}\subseteq \OOO$ denote the set of all equivalence classes $\ooo\in\OOO$ such that there exists $f\in C_{R, \bkappa}^{\infty}(G(\A)^1)$ with $J_{\ooo}(f)\neq0$.
\end{definition}
The set $\OOO_{R, \bkappa}$ is in fact finite (cf.~\cite{Ar86}) so that in particular the coarse geometric expansion is a finite sum.
 We give more details about the properties of  classes in $\OOO_{R, \bkappa}$ in Section~\ref{section_contr_eq_classes}.

 For $\ooo\in\OOO$ and $f\in C_c^{\infty}(G(\A)^1)$ we say that a finite set of places $S\supseteq S_{\infty}$ is sufficiently large with respect to $\ooo$ and $f$ if it satisfies the properties given in~\cite[p. 203]{Ar86}.
Let $S_{\text{bad},1}$ be the set of non-archimedean places of $\F$ of residue characteristic at most $n!$. 
Recall the definition of $S_{\bkappa}$ from~\eqref{eq:places:dividing:kappa}. If $\ooo\in\OOO$, and $\sigma\in\ooo$ is a semisimple element representing this equivalence class, set
\[
 S_{\ooo}:=\{v<\infty\mid |D^G(\sigma)|_v\neq1\}.
\]
Here 
\begin{equation}\label{eq:def:weyl:discr}
 D^G(\sigma):=\det\left(1-\Ad(\sigma); \ggG/\ggG_{\sigma}\right)\in \F
\end{equation}
is the Weyl discriminant of $\sigma$. Note that $D^G(\sigma)$ only depends on $\ooo$ but not on the choice of $\sigma$ so that $S_{\ooo}$ is indeed well-defined.

\begin{lemma}\label{lem:suff:large:set}
There exists a set $S_{\text{bad}}$ containing $S_{\text{bad},1}$ and depending only on $n$ such that the following holds.
Let $\ooo\in\OOO_{R,\bkappa}$, and put 
\[ 
S^{\bkappa, \ooo}
:=S_{\infty} \cup S_{\text{bad}} \cup S_{\bkappa}\cup S_{\ooo}.
\]
Then for every  $f\in C_{R, \bkappa}^{\infty}(G(\A)^1)$ the set $S^{\bkappa,\ooo}$ is sufficiently large with respect to $\ooo$ and $f$ in the sense of~\cite[p. 203]{Ar86}.
\end{lemma}

\begin{rem}
For $\ooo\in\OOO_{R,\bkappa}$ the elements in the set $S_{\ooo}$ can be bounded in terms of $\bkappa$ and $R$ by Lemma~\ref{prop_of_contr_classes}.
\end{rem}

\begin{proof}[Proof of Lemma \ref{lem:suff:large:set}]
This follows from~\cite[Appendix]{Ar86} together with Lemma~\ref{prop_of_contr_classes} and Corollary~\ref{prop_distance_to_torus} below.
\end{proof}

We included all the places with residue characteristic $\le n!$ in the set $S^{\bkappa,\ooo}$ to make sure that every Galois extension of $\F_v$ of degree at most $n!$ is at worst tamely ramified for every $v\not\in S^{\bkappa,\ooo}$. This will be useful later.  From now on $S$ will always denote a finite set of places of $\F$ containing $S^{\bkappa,\ooo}$ (provided that $\bkappa$ and $\ooo$ are clear from the context).

Arthur's fine geometric expansion is actually a refinement of the coarse expansion~\eqref{eq:coarse:geom:exp}: 
By~\cite[Theorem 8.1]{Ar86} specialised to $\GL_n$ there exist coefficients $a^M(\gamma, S)\in\C$ for  $M\in\LLL$ and $\gamma\in M(\F)$ such that
\begin{equation}\label{globdistr_wgt_int}
J_{\ooo}(f)
=\sum_{M\in\LLL}\frac{|W^M|}{|W^G|} \sum_{\gamma}a^M(\gamma, S) J_{M}^G(\gamma, f)
\end{equation}
for all $f\in C_{R, \bkappa}^{\infty}(G(\A)^1)$, $\ooo\in\OOO_{R, \bkappa}$, and any finite set of places $S$ containing $S^{\bkappa,\ooo}$.
Here $\gamma\in M(\F)\cap \ooo$ runs over a set of representatives for the $M(\F)$-conjugacy classes in $M(\F)\cap\ooo$, and the  $J_{M}^G(\gamma, f)$ are certain ($S$-adic) weighted orbital integrals which we will describe in more detail later. In particular, $J_M^G(\gamma, f)$ depends on the set $S$ although this is not visible from the notation, and it depends only on the $G(\F_S)$-conjugacy class of $\gamma$, but not on the specific representative.
  Note that the number of $M(\F)$-conjugacy classes on $M(\F)\cap \ooo$ is bounded  independently of $\ooo$, namely by the number of unipotent conjugacy classes in $M_{\sigma}(\F)$ which in turn is bounded in terms of $n$. The absolute value of the coefficients $a^M(\gamma, S)$ depends on the normalisation of measures chosen on the various groups involved in the definition of the weighted orbital integrals. With respect to our fixed measures we can give an upper bound for these coefficients.
If $\sigma\in M(\F)$ is elliptic, fix a diagonal matrix $\diag(\zeta_1, \ldots, \zeta_n)\in T_0(\bar\F)$ which is conjugate to $\sigma$ in $M(\bar\F)$ for $\bar\F$ an algebraic closure of $\F$. 
We set 
\[
 \Delta^M(\sigma):= \N_{\F/\Q} \Big(\prod_{i< j: \zeta_i\neq\zeta_j} (\zeta_i-\zeta_j)^2\Big),
\]
where the product runs over all indices $i<j$ for which $\alpha_i+\ldots+\alpha_{j-1}\in\Phi^{M,+}$. Here $\Phi^M$ denotes the roots system of $M$ with respect to $T_0$ and we choose the positive roots $\Phi^{M,+}$ such that they are compatible with our choice $\Phi^+$, i.e., $\Phi^{M,+}=\Phi^M\cap\Phi^+$. Hence if $\sigma$ is regular elliptic in $M(\F)$, then $\Delta^M(\sigma)$ is the norm of the discriminant of $\sigma$ as an element of $M(\F)$. Note that $\Delta^M(\sigma)$ is well-defined, i.e., not depending on the choice of $\diag(\zeta_1, \ldots, \zeta_n)$ since every two such diagonal matrices are conjugate by some element of the Weyl group for the pair $(M, T_0)$.

Further, for a non-archimedean place $v$ let $\zeta_{\F_v}$ denote the local Dedekind zeta function, i.e., $\zeta_{\F_v}(s)=(1-q_v^{-s})^{-1}$ for $\Re s>0$.
The main result of~\cite{coeff_est} is the following:

\begin{proposition}[{{\cite[Corollary 1.4]{coeff_est}}}]\label{prop:coeff:est}
There exist $a, b\geq 0$ depending only on $n$ and $\F$ such that the following holds: 
Let $S'$ be a finite set of places of $\F$ with $S_{\infty}\subseteq S'$,  and put $S'_f=S'\minus S_{\infty}$.
 Then for every $M\in\LLL$ and $\gamma\in M(\F)$ whose eigenvalues (in $\bar\F$) are all algebraic integers, we have
\begin{equation*}
\left|a^M(\gamma, S')\right|
\begin{cases}
	      \leq a \Delta^M(\sigma)^{b} 
	      \sum_{(s_v)_{v\in S_f'}}
	      \prod_{v\in S_{f}'}\Big|\frac{\zeta_{\F_v}^{(s_v)}(1)}{\zeta_{\F_v}(1)}\Big|	
													&\text{if } \sigma\text{ is elliptic in }M(\F), \\
	      =0											&\text{otherwise}
\end{cases}
\end{equation*}
with respect to the measures defined in Section~\ref{subsection_measures}. Here the sum runs over all tuples $s_v\in\Z_{\geq0}$ indexed by $v\in S_f'$ such that $\sum_{v\in S'_f} s_v\leq\dim\aaa_{M_{\sigma}}^M$.
\end{proposition}

\begin{rem}
 The stipulation in the proposition that all eigenvalues of $\gamma$ are algebraic integers is no real restriction:
The coefficients are invariant under scaling, i.e., $a^M(\gamma, S)=a^M(\alpha\gamma,S)$ for all $\alpha\in\F^{\times}$ so that the proposition in fact gives an upper bound for $|a^M(\gamma, f)|$ for any $\gamma$. 
\end{rem}

Since $\big|\zeta_{\F_v}^{(k)}(1)\zeta_{\F_v}(1)^{-1}\big|\leq c_k (\log q_v)^k$ for some constant $c_k>0$ depending only on the integer $k$, we in particular get the (slightly worse) upper bound
\begin{equation}\label{eq:coeff:est}
\left|a^M(\gamma, S')\right|
\ll |S'|^n \Delta^M(\sigma)^{\kappa} \Big(\prod_{v\in S'_f} \log q_v\Big)^{n}  
\end{equation}
if $\sigma$ is elliptic in $M(\F)$.

\subsection{Restrictions on contributing equivalence classes}\label{section_contr_eq_classes}
The assumption that $J_{\ooo}(f)\neq0$ for some $f\in C_{R, \bkappa}^{\infty}(G(\A)^1)$ gives constraints on $\ooo$ which we want to study in this section.
  
Let $v$ be a non-archimedean place of $\F$ and write $F=\F_v$. Let $\bar{F}$ denote an algebraic closure of $F$, and $\bar\F$ an algebraic closure of $\F$. Let $\val_F:F\longrightarrow \Z$ denote as usual the valuation on $F$ normalised by $\val_F(\varpi_F)=1$. The extension of the valuation $\val_{F}$ of $F$ to $\bar F$ is again denoted by $\val_F$.
The following gives a bound on the size of the eigenvalues of an element in $\omega_{F, \xi}=\cpt_F\varpi_F^{\xi}\cpt_F$ if $\xi\in X_*^+(T_0)$:
\begin{lemma}\label{lemma:bound:ev}
Let $x=(x_{ij})_{i,j=1,\ldots,n}\in G(F)$ with eigenvalues $\zeta_1, \ldots, \zeta_n\in \bar{F}$, and let $\xi\in X_0^+(T_0)$ such that $x\in\cpt_F\varpi_F^{\xi}\cpt_F$. Then:
\begin{enumerate}[label=(\roman{*})]
\item For all $i=1,\ldots, n$ we have
\[
 -\|\xi\|_W\leq \val_{F}(\zeta_i)\leq\|\xi\|_W. 
\]

\item\label{lem:bound:matrix:entry}  For all $i,j=1,\ldots, n$ we have
\[
 |x_{ij}|_F\le \max_{1\le k\le n} q_F^{|\xi_k|}.
\]

\end{enumerate}
\end{lemma}
\begin{proof}
 \begin{enumerate}[label=(\roman{*})]
  \item This is~\cite[Lemma 2.15]{ShTe12}.
  \item By assumption $x\in \cpt_F\varpi_F^{\xi}\cpt_F$ with $\xi_1\ge\ldots\ge\xi_n$ so that $\varpi_F^{-\xi_n}x\in\Mat_{n\times n}(\OOO_F)$. Hence $|\varpi_F^{-\xi_n} x_{ij}|_F\le 1$ for all $i,j$ so that 
  \[
    |x_{ij}|_F
    \le q_F^{-\xi_n}
    \le \max_{1\le k\le n} q_F^{|\xi_k|}
  \]
as asserted.\qedhere
 \end{enumerate}
\end{proof}

\begin{cor}\label{cor:bound:rootvalues}
 Suppose $\sigma\in G(F)$ is semisimple and in $G(\bar F)$ conjugate to $\tilde{\sigma}=\diag(\zeta_1, \ldots, \zeta_n)\in T_0(\bar F)$. Then, if $\sigma\in \cpt_F\varpi_F^{\xi} \cpt_F$ for a suitable $\xi\in X_*^+(T_0)$, we have
\[
-2\|\xi\|_W\leq \val_F(\alpha(\tilde{\sigma}))\leq 2\|\xi\|_W,
\;\text{ and }\;\; 
\val_F(1-\alpha(\tilde{\sigma}))
 \geq - 2\|\xi\|_W
\]
for all $\alpha\in \Phi^+$. \qed
\end{cor}

If $\sigma, \tilde\sigma$ are as in the corollary, we set
\[
 \Delta_F^-(\sigma):=\prod_{\substack{\alpha\in\Phi^+:\\\alpha(\tilde\sigma)\neq1}} \max\{1, |1-\alpha(\tilde{\sigma})|_F^{-1}\}.
\]
We make the same definition if $F$ is an archimedean field 
Clearly, $\Delta_F^-(\sigma)\geq1$ in any case.

\begin{cor}\label{cor:bound:pmdiscr}
 With the same notation as in Corollary~\ref{cor:bound:rootvalues} we have
 \[
\Delta_F^-(\sigma)
 \leq q_F^{2|\Phi^+|\|\xi\|_W} |D^G(\sigma)|_F^{-1},
\]
where $D^G(\sigma)\in F$ denotes the Weyl discriminant of $\sigma$ as in~\eqref{eq:def:weyl:discr}.
\end{cor}

\begin{proof}
 Note that $|D^G(\sigma)|_F= \prod_{\alpha\in \Phi^+:~ \alpha(\tilde{\sigma})\neq1} |\val_F(1-\alpha(\tilde\sigma))|_F$. Hence the assertion follows from the results above.
\end{proof}

\begin{cor}\label{cor:bound:ev}
Let $\sigma\in G(\F)$, and for each non-archimedean $v$ let $\xi_v\in X_*^+(T_0)$ be such that $\sigma$ is in $G(\F_v)$ conjugate to some element in $\cpt_v \varpi_v^{\xi_v}\cpt_v$. 
Let $\zeta_1, \ldots, \zeta_n\in\bar\F $ be the eigenvalues of $\sigma$ in $\bar\F$ and set $\tilde{\sigma}=\diag(\zeta_1, \ldots, \zeta_n)\in T_0(\bar\F)$. Then all but finitely many $\xi_i$ vanish, and
\begin{align*}
 \prod_{v<\infty} q_v^{-\|\xi_v\|_W}  \leq |\zeta_i|_{\C}
& \leq \prod_{v<\infty} q_v^{\|\xi_v\|_W} ,\\
|1-\alpha(\tilde\sigma)|_{\C}^{-1}
& \leq \prod_{v<\infty} q_v^{2\|\xi_v\|_W}
\end{align*}
for all $i=1, \ldots, n$ and all $\alpha\in\Phi^+$. In particular, 
\[
 \Delta_{\F_{\infty}}^{-}(\sigma)\leq \prod_{v<\infty} q_v^{2|\Phi^+| \|\xi_v\|_W}.
\]
\end{cor}

After these preliminary observations we return to our equivalence classes $\ooo\in\OOO_{R, \bkappa}$. 
If $\sigma\in G(\F)$ is semisimple and representing the equivalence class $\ooo$, we define the (absolute) Weyl discriminant $D^G(\ooo)$ of $\ooo$ by 
\[
D^G(\ooo)
=|D^G(\sigma)|_{\C}
=|\det\left(1-\Ad(\sigma); \ggG/\ggG_{\sigma}\right)|_{\C}\in\Q,
\]
which is independent of the choice of the semisimple representative $\sigma\in G(\F)$. Clearly, $D^G(\ooo)$ is just the $\F$-norm of $D^G(\sigma)$.

In the following we shall often write $S_f=S\minus S_{\infty}$, and define
\begin{equation}\label{eq:def:prod:over:places}
 \Pi_{\bkappa}:=\prod_{v\in S_{\bkappa}} q_v^{\kappa_v}=\prod_{v<\infty} q_v^{\kappa_v}, \text{ and }\;\;
\Pi_{S,\bkappa,\ooo}= \prod_{v\in S\minus S^{\bkappa,\ooo}} q_v.
\end{equation}
Then $\Pi_{\bkappa}$ and $\Pi_{S,\bkappa,\ooo}$ are both non-negative integers.

\begin{lemma}\label{prop_of_contr_classes}
Suppose $\ooo\in\OOO_{R, \bkappa} $ corresponds to the conjugacy class of $\sigma\in G(\F)_{ss}$.
Then:
\begin{enumerate}[label=(\roman{*})]
\item\label{prop_of_contr_classes1}
The characteristic polynomial $\chi_{\ooo}(x)=x^n+a_{n-1}x^{n-1}+\ldots+a_0$ of $\sigma$ depends only on $\ooo$ but not on the representative $\sigma$, and $a_0, \ldots, a_{n-1}\in\OOO_{\F}$.

\item\label{prop_of_contr_classes2}
For every non-archimedean place $v$ we have $|a_0|_v=|\det\sigma|_v=q_v^{-k}$
for some $k\leq \kappa_v$.

\item\label{prop_of_contr_classes3}
$\N_{\F/\Q}(a_i)=|a_i|_{\C}\ll_R \Pi_{\bkappa}$
for every  $i=0, \ldots, n-1$.

\item\label{bound_on_eigenvalues}
If $\zeta_1, \ldots, \zeta_n\in\bar\F$ are the roots of $\chi_{\ooo}$, then 
$|\zeta_i|_{\C}\ll_R\Pi_{\bkappa} $ for all $i$.

\item\label{bound_on_discr}
$D^G (\ooo)
\ll_R \Pi_{\bkappa}^{n(n-1)}$.
\end{enumerate}
Moreover,
\[
\big|\OOO_{R, \bkappa} \big|
=\big|\left\{\ooo\in\OOO\mid~\exists f\in C_{R, \bkappa}^{\infty}(G(\A)^1):~ J_{\ooo}(f)\neq0\right\}\big|
\ll_R \Pi_{\bkappa}^{c_2}
\]
for some constant $c_2>0$ depending only on $n$ and $\F$.
\end{lemma}

The number of elements in $\OOO_{R, \bkappa}$ can also be bounded by using~\cite[Proposition 8.7]{ShTe12}. This proposition applies also to more general groups, but since we are only interested in $\GL_n$, the proof in our case is more elementary.

\begin{proof}
Let $\ooo\in\OOO_{R,\bkappa} $. It is clear that the characteristic polynomial of any representative $\sigma$ depends only on $\ooo$.
Recall the definition of the double coset $\omega_{v,\xi}$ from~\eqref{eq:def:double:coset} and set $\omega_{v,R}=A_G\cpt_{\C}e^{V(R)}\cpt_{\C} $ if $v$ is archimedean.
Then for every non-archimedean place $v$ there exist $\nu_v\in\UUU_{G_{\sigma}}(\F_v)$ and $\xi_v$ such that the $G(\F_v)$-orbit of $\sigma\nu_v$ intersects $\omega_{v,\xi_v}$ and $\|\xi\|_W\leq \kappa_v$. Note that the definition of $\OOO_{R,\bkappa}$ in Section~\ref{sec:test:fcts} implies that $\xi_{vi}\ge0$ for all $i=1,\ldots, n$. 
Similarly, there exists $\nu_v\in\UUU_{G_{\sigma}}(\C)$ such that the $G(\C)$-orbit of $\sigma\nu_{v}$ intersects $\omega_{v,R}$ if $v$ is archimedean. 
In any case, let $\tilde{\gamma}_v$ be an element of this intersection.

Now the characteristic polynomial of $\tilde{\gamma}:=\prod_{v}\tilde{\gamma}_v$ (over the ring $\A$) coincides with $\chi_{\ooo}$.
If $v$ is a non-archimedean place, $\tilde{\gamma}_v\in\omega_{v,\xi_v}\subseteq \Mat_{n\times n}(\OOO_{\F_v})$ so that the characteristic polynomial of $\tilde{\gamma}_v$ has coefficients in $\OOO_{\F_v}$. 
Therefore $\chi_{\ooo}$ has coefficients in $\bigcap_{v\not\in S_{\infty}}\OOO_{\F_v}=\OOO_{\F}$ proving~\ref{prop_of_contr_classes1}.
For~\ref{prop_of_contr_classes2}, note that  $\tilde{\gamma}_v\in\omega_{v,\xi_v}$ implies
$|\det\sigma|_v=|\det \tilde{\gamma}_v|_v\geq q_v^{-\kappa_v}$.
Since $\tilde{\gamma}_v\in\omega_{v,R}$ if $v$ is archimedean, the absolute value of the coefficients $a_i'$, $i=0, \ldots, n-1$, of the characteristic\- polynomial of $\tilde{\gamma}':=\tilde{\gamma}_v |\det \tilde{\gamma}_v|_v^{-\frac{1}{n}}$ are bounded from above by some constant depending only on $R$. Hence
\[
\N_{\F/\Q}(a_i)=|a_i|_{\C}
\ll_R\max\{1, |\det \tilde{\gamma}_{\C}|_{\C}\}
= |\det \sigma|_{\C}
\leq \prod_{w\in S_{\bkappa}} q_w^{\kappa_w}=\Pi_{\bkappa}
\]
for all $i$.
Applying Rouch\'{e}'s Theorem to the characteristic polynomial yields~\ref{bound_on_eigenvalues} (we could also apply Corollary~\ref{cor:bound:ev}), and therefore also~\ref{prop_of_contr_classes3} and~\ref{bound_on_discr} by noting that the discriminant of a regular element is a homogeneous polynomial of degree at most $n(n-1)$ in the roots of $\chi_{\ooo}$ (if $\sigma$ is not regular, it is a homogeneous polynomial of degree less than $n(n-1)$).

The assertion about the number of elements in $\OOO_{R, \bkappa}$ now follows from the fact that $\ooo$ is uniquely determined by its characteristic polynomial, and that this characteristic polynomial must have coefficients in $\OOO_{\F}$ of norm bounded by  $C\Pi_{\bkappa}$ for some $C\ll_R 1$. In other words, we need to count the number of points of the lattice $\OOO_F\hookrightarrow \F_{\infty}=\C$ contained in the disk of radius $(C \Pi_{\bkappa})^{1/2}$. The number of such points is bounded by $c \Pi_{\bkappa}^{1/2}$ for some constant $c>0$ (depending only on $\F$) which follows from the fact that the Dedekind zeta function of $\F$ has a simple pole at $1$. Hence the last claim of the lemma follows.
\end{proof}

\subsection{Reduction to local distributions}\label{subsec:reduction:loc:distr}
To study the $S$-adic weighted orbital integrals $J_M^G(\gamma, f)$ appearing in the fine geometric expansion we break it down into $v$-adic integrals for $v\in S$. (Recall that $S$ is an arbitrary finite set of places of $\F$ containing $S^{\bkappa,\ooo}$.)
 For this we use a variant of Arthur's splitting formula for $(G, M)$-families applicable to weighted orbital integrals. 
The splitting formula not only holds for compactly supported functions, but also functions of \emph{almost compact support}. We do not define this space $C^{\infty}_{ac}(G(\F_S))$ of almost compactly supported functions here (see~\cite[\S 1]{Ar88b} for a definition), but only not that our test functions $f^{t,\mu}_{\bxi}$ are contained in this space. 
 
 Suppose first that we partition $S$ into two disjoint sets $S_1$, $S_2$. Further suppose that we have a test function $f=f_1\cdot f_2\in C_{ac}^{\infty}(G(\F_{S_1}))\cdot C_{ac}^{\infty}(G(\F_{S_2}))$ and an element $\gamma_1\cdot\gamma_2\in M(\F_{S_1})\cdot M(\F_{S_2})$ in the $M(\F_S)$-conjugacy class of $\gamma$. Recall the definition of the constant term map from~\eqref{eq:def:constant:term}.
Then by~\cite{Ar88b} we have
\[
 J_M^G(\gamma, f)=\sum_{L_1, L_2\in \LLL(M)} d_M^G(L_1, L_2) J_M^{L_1}(\gamma_1, f_1^{(Q_1)}) J_M^{L_2}(\gamma_2, f_2^{(Q_2)})
\]
 for certain coefficients $d_M^G(L_1, L_2)\in\R$ depending only on $M$, $L_1$, $L_2$ but not on the sets $S$, $S_1$, or $S_2$, and such that $d_M^G(L_1, L_2)=0$ unless the natural map
\[
 \aaa_M^{L_1}\oplus\aaa_M^{L_2}\longrightarrow\aaa_M^G
\]
is an isomorphism. Further, $Q_i$ is a certain parabolic subgroup $Q_i=L_iV_i\in \PPP(L_i)$ associated with $L_i$ as in~\cite[\S 7]{Ar88b} which does not depend on the sets $S$, $S_1$, or $S_2$ or on $L_j$ for $j\neq i$,
and $ f_i^{(Q_i)}\in C_{ac}^{\infty}(L_i(\F_{S_i}))$ is defined by
\[
 f_i^{(Q_i)}(m)= \delta_{Q_i}(m)^{1/2}\int_{\cpt_{S_i}}\int_{V_i(\F_{S_i})} f(k^{-1}mvk)\,dv\,dk.
\]

We can repeat this procedure with $S_i$ in place of $S$, $L_i$ in place of $G$ and so on, and split the integral $J_M^{L_i}(\gamma_i, f_i^{(Q_i)})$ further. Note that if $M\subseteq L'\subseteq L\subseteq G$ are Levi subgroups, and $Q'\in\PPP^L(L')$ and $Q\in\PPP^G(L)$ with unipotent radical $U_Q$, then $Q'U_Q\in\PPP^G(L')$ and $(f^{(Q)})^{(Q')}=f^{(Q'U_Q)}$. We can repeat this procedure until we have split $S$ into the sets $\{v\}$, $v\in S$. 

Let $\LLL_{S}(M)$ denote the set of tuples $\underline{L}=(L_v)_{v\in S}$ of Levi subgroups $L_v\in\LLL(M)$ indexed by $v\in S$. 
If $\underline{L}\in\LLL_{S}(M)$ and  $\gamma\in M(\F)$, choose for every $v\in S$ an element $\gamma_v\in M(\F_v)$ which is conjugate to $\gamma$ in $G(\F_v)$.

The procedure just described immediately yields the following:
\begin{lemma}\label{reduction_to_local_case}
 There are globally defined constants $d_M^G({\underline{L}})\in \R$ for $\underline{L}\in \LLL_{S}(M)$ such that for all $f=\prod_{v\in S}f_v\in C_{ac}^{\infty}(G(\F_{S}))$ and all $\gamma=\prod_{v\in S}\gamma_v\in M(\F_{S})$, 
\begin{equation}\label{red_to_loc_case}
J_M^G(\gamma, f)
=\sum_{\underline{L}\in\LLL_{S}(M)} d_M^G(\underline{L}) \prod_{v\in S} J_M^{L_v}(\gamma_v, f_v^{(Q_v)})
\end{equation}
where for $v\in S$ the parabolic subgroups $Q_v\in\PPP(L_v)$ are attached to $L_v$ in a way that does not depend on $S$.

Moreover, for any $\underline{L}\in \LLL_S(M)$ with $d_M^G(\underline{L})\neq0$ the following is true:
\begin{enumerate}[label=(\roman{*})]
\item $d_M^G(\underline{L})$ can attain only a finite number of values. The collection of these values depends only on $n$.

\item The natural map
\[
 \bigoplus_{v\in S} \aaa_M^{L_v}\longrightarrow \aaa_M^G
\]
is an isomorphism.

\item 
$\left|\left\{v\in S\mid L_v\neq M\right\}\right|\leq \dim \aaa_M^G$.

\item
$\sum_{v\in S} \dim\aaa_M^{L_v} \leq \dim \aaa_M^G$.

\item
The number of contributing tuples $\underline{L}\in\LLL_{S}(M)$ is bounded by
\[
\left|\left\{\underline{L}\in \LLL_{S}(M)\mid d_M^G(\underline{L})\neq0\right\}\right|
\leq \Big(|\LLL(M)| |S|\Big)^{\dim\aaa_M^G}.
\]\qed
\end{enumerate}
\end{lemma}

\begin{rem}
 $J_M^{L_v}(\gamma_v, f_v^{(Q_v)})$ only depends on the $L_v(\F_v)$-conjugacy class of $\gamma_v$ so that we may later choose $\gamma_v$ in this conjugacy class such that its centraliser has the most convenient form for us. 
\end{rem}

It is clear now from the fine geometric expansion, the splitting formula, and the upper bounds for $|a^M(\gamma, S)|$ and $|\OOO_{R, \bkappa}| $ that we only need to study the problems~\ref{arch} and~\ref{non-arch} from the introduction. We shall do so in the next sections.

\section{Norms and distances}\label{section_norms}
We need to bound weighted orbital integrals at the non-archimedean places $v$ (which will be done in the next two sections) and a common way to do this, is to work on the Bruhat-Tits building for $G(\F_v)$. We recall some necessary properties of this building in this section.
\subsection{Buildings and distance functions}\label{subsec:buildings}
Let $p$ be a rational prime and $F/\Q_p$ a finite extension with valuation $v$.
The principal apartment $\AAA_F=\AAA(T_0,F)$ of the (extended) Bruhat-Tits building of $G$ over $F$ is as a set equal to $\aaa_0$. The Weyl group acts by reflections on $\AAA_F$, and $T_0(F)$ on $\AAA_F$ by $T_0(F)\times \AAA_F\ni (t,X)\mapsto X-H_F(t)\in\AAA_F$. Here $H_F: T(F)\longrightarrow\aaa_0$ denotes the map $t=\diag(t_1,\ldots, t_n)\mapsto (\log_{q_F}|t_1|_F, \ldots, \log_{q_F} |t_n|_F)$. Hence we get an action of the normaliser $N_F\simeq T_0(F)\rtimes W$ of $T_0(F)$ in $G(F)$ on $\AAA_F$.
In particular, the image of $T_0(F)\times \{0\}\subseteq T_0(F)\times\AAA_F$ in $\AAA_F$ equals the lattice of cocharacters $X_*(T_0)\subseteq\aaa_0$.
The apartment $\AAA_F$ in fact has a simplicial structure which is preserved by the action of $N_F$ (cf.~\cite{AGP13}). 
Our Weyl group invariant norm $\|\cdot\|_W$ then defines a metric $d_{\AAA_F}$ on $\AAA_F$.

The (affine) Bruhat-Tits building $\BBB(G,F)$ of $G$ over $F$ can be constructed as in~\cite{AGP13}: The points of $\BBB(G,F)$ are the equivalence classes of $(g,X)\in G(F) \times \AAA_F$ under a certain equivalence relation for which we denote the equivalence classes by $[(g,X)]$. More precisely, $(g,X)$ and $(h, Y)$ are equivalent if and only if there exists an element $n\in N_F$ such that $n\cdot X=Y$ and $g^{-1} h n$ is contained in the parahoric subgroup associated with the point $X$.
The group $G(F)$ acts on $\BBB(G,F)$ by $h\cdot~\!\![(g,X)]=[(hg,X)]$ and we can embed the homogeneous space $G(F)/\cpt_F$ into $\BBB(G,F)$ by $g\cpt_F\mapsto [(g,0)]$.
If $g\in G(F)$, then $g\cdot \AAA_F\subseteq \BBB(G,F)$ is the apartment $\AAA(gT_0g^{-1},F)$ associated with the split torus $gT_0(F)g^{-1}$, and we define a metric on $\AAA(gT_0g^{-1}, F)$ by ``transporting'' the metric $d_{\AAA_F}$ to $\AAA(gT_0g^{-1},F)$ via $g$.
For any two points in $\BBB(G,F)$ we can find an element in $G(F)$ mapping both points into a common apartment, and this gives rise to a well-defined metric $d_{\BBB(G,F)}$ on the whole building.

\subsection{Behaviour under field extensions}\label{subsec:buildings:field:ext}
Let $E/F$ be a finite Galois extension of degree $d=[E:F]$ and ramification degree $\fff=\fff(E/F)$. Let $\OOO_F\subseteq F$ (resp.\ $\OOO_E\subseteq E$) be the ring of integers and let  $\varpi_F\in\OOO_F$ (resp.\ $\varpi_E\in\OOO_E$) be a uniformising elements. 
Let $\eee(E/F)=d/\fff$ be the ramification index so that $\varpi_E^{\eee(E/F)}$ equals the product of $\varpi_F$ with some suitable unit in $\OOO_E^{\times}$.
Note that $\log_{q_E}|a|_E=\eee(E/F)\log_{q_F}|a|_F$ for every $a\in F^{\times}$.
As sets the principal apartments $\AAA_F\subseteq \BBB(G,F)$ and $\AAA_E\subseteq \BBB(G,E)$ coincide and we identify them via multiplication with $\eee(E/F)$,
\[
 \AAA_F\xrightarrow{\;\;\;\eee(E/F)\cdot\;\;\;}\AAA_E,
\]
since this respects the simplicial structures of $\AAA_F$ and $\AAA_E$ and identifies the images of $T_0(F)$ in $\AAA_F$ with the image of $T_0(E)$ in $\AAA_E$.
This gives an embedding of $G(F)\times\AAA_F\hookrightarrow G(E)\times \AAA_E$. One can check that the equivalence relation is preserved by this embedding so that we get an embedding $\BBB(G,F)\hookrightarrow\BBB(G,E)$ for which the metric satisfies $d_{\BBB(G,E)} (x,y)=\eee(E/F) d_{\BBB(G,F)} (x,y)$ for all $x,y\in \BBB(G,F)$, cf.~\cite[(3.3)]{AGP13}.
The Galois group $\Gal(E/F)$ acts on $\BBB(G,E)$ and we denote by $\BBB(G,E)^{\Gal(E/F)}$ the fixed points of $\Gal(E/F)$ on $\BBB(G,E)$.
In general, one only has $\BBB(G,F)\subseteq \BBB(G, E)^{\Gal(E/F)}$, but if $E/F$ is tamely ramified (e.g., if $v\not\in S_{\text{bad},1}$), then $\BBB(G,F)=\BBB(G, E)^{\Gal(E/F)}$ (cf.~\cite{Pr01}).
In any case, the points of $\BBB(G, E)^{\Gal(E/F)}$ can not be ``too far away'' from the points in $\BBB(G, F)$ as explained in~\cite{Rou77}. We will use this fact in the proof of Corollary~\ref{prop_distance_to_torus} below.

\subsection{Norms on groups}\label{subsec:norms}
We keep the notation from the last section but also allow $F=\R$ or $F=\C$ here. 
If $F$ is non-archimedean, we define a norm $\|\cdot\|_{G(F)}$ on $G(F)$ (in the sense of~\cite{Ko05}) by setting
\[
\|g\|_{G(F)}= e^{\|\xi\|_W}
\]
where $\xi\in X_*^+(T_0)$ is uniquely determined by $g\in\cpt_F\varpi_F^{\xi}\cpt_F$. If $F$ is archimedean, we define a norm $\|\cdot\|_{G(F)^1}$ on $G(F)^1$ analogously by
\[
 \|g\|_{G(F)^1}= e^{\|\xi\|_W}
\]
where $\xi\in\overline{\aaa^+}$ is such that $g\in\cpt_Fe^{\xi} \cpt_F$. We extend $\|\cdot\|_{G(F)^1}$ trivially to all of $G(F)$.

To unify notation we will sometimes write $q_F=e=\varpi_F$ if $F$ is archimedean.
Note that if $F$ is non-archimedean, then $\log\|g\|_{G(F)}=d_{\BBB(G,F)}(g \cdot x_0, x_0)$ for $x_0=[(\id, 0)]\in \BBB(G,F)$ the base point of the homogeneous space $G(F)/\cpt_F\subseteq \BBB(G,F)$.
If $F$ is an archimedean field, we set $\eee(E/F)=[E:F]$ if $E=\C$. 
Then the norms on $G(F)$ and $G(E)$ are related by
\begin{equation}\label{lemma_behav_of_norm_under_field_extension}
\|g\|_{G(E)} =\|g\|_{G(F)}^{\eee(E/F)}
\end{equation}
if $F$ is non-archimedean, and by $\|g\|_{G(E)^1} =\|g\|_{G(F)^1}^{\eee(E/F)}$ if $F$ is archimedean.
We collect a few facts about this norm:

\begin{lemma}\label{est_on_norms}
Let $\G$ denote the group $G(F)$ if $F$ is non-archimedean, and the group $G(F)^1$ if $F$ is archimedean.
\begin{enumerate}[label=(\roman{*})]
\item If $F$ is non-archimedean and if $\xi_1, \xi_2, \xi_3\in X_*^+(T_0)$, then 
\[\big(\omega_{F,\xi_1}\cdot\omega_{F,\xi_2}\big)\cap\omega_{F,\xi_3}\neq\emptyset
\;\;\;\Rightarrow \;\;\;
\|\xi_3\|_W\leq\|\xi_1\|_W+\|\xi_2\|_W.
\]
In particular, if $g_1, g_2\in\G$, then 
\[
\|g_1g_2\|_{\G}\leq \|g_1\|_{\G} \| g_2\|_{\G}.
\]

\item\label{lemma_est_on_norms_inverse}
If $g\in G(F)$, then $\|g^{-1}\|_{\G}=\|g\|_{\G}$.

\item\label{lemma_est_on_norms_iwasawa}
If $g=muk\in G(F)$ with 
$m\in M(F)$, $u=(u_{ij})_{i,j=1,\ldots,n}\in U(F)$, $k\in\cpt_F$ for $P=MU$ a standard parabolic subgroup, then 
\[
 \|m\|_{\G}\leq \|g\|_{\G}\;\text{ and }\;\;
\|u\|_{\G} \leq \|g\|_{\G}^{2}
\]
if $F$ is non-archimedean, and if $F$ is archimedean, then 
\[
 \|m\|_{\G}\leq n^{\eee(F/\R)}\|g\|_{\G}^{\eee(F/\R)}\;\text{ and }\;\;
|u_{ij}|_F \leq n^{2\eee(F/\R)}\|g\|_{\G}^{2\eee(F/\R)}
\]
for all $1\le i<j\le n$ and further 
\[
\|u\|_{\G}\le n^{4(n-1)} \|g\|_{\G}^{3(n-1)}.
\]
\end{enumerate}
\end{lemma}
\begin{proof}
Part~\ref{lemma_est_on_norms_inverse} is clear, since $g\in \cpt_F\varpi_F^{\xi} \cpt_F$ implies $g^{-1}\in\cpt_F\varpi_F^{-\xi} \cpt_F$, and $\|-\xi\|_W=\|\xi\|_W$.
The first part follows  from~\cite[(4.4.4)]{BrTi72}. 

For the last part we can clearly assume that $g=mu$. Using Cartan decomposition for $M(F)$ and the fact that $\cpt_F^M$ normalises $U(F)$ and leaves the norm invariant, we moreover can assume that $m=\diag(t_1, \ldots, t_n)\in T_0(F)$. 
Suppose first that $F$ is non-archimedean.
Let $\xi=(\xi_1, \ldots, \xi_n)\in X_*^+(T_0)$ be such that $mu\in\cpt_F\varpi_F^{\xi} \cpt_F$.  
Then $\|m\|_{\G}=\max_{1\le k\le n} e^{\log_{q_F}|t_k^{\pm1}|_F}$ and each $t_i^{\pm1}$ is a matrix entry of either $mu$ or $(mu)^{-1}$. By part~\ref{lemma_est_on_norms_inverse} and Lemma~\ref{lemma:bound:ev}~\ref{lem:bound:matrix:entry} we therefore get
\[
 \|m\|_{\G}\leq \max_{1\le k\le n} e^{|\xi_k|} =\|g\|_{\G}. 
\]
The upper bound on $\|u\|_{\G}$ then follows from $\|u\|_{\G}\leq \|m^{-1}\|_{\G}\|mu\|_{\G}$ because of the first part of the lemma.

Now suppose that $F$ is archimedean and $\xi\in\overline{\aaa^+}$ is such that $g=mu=s k_1 e^\xi k_2$ with $s\in\R_{>0}$ and $k_1, k_2\in\cpt_F$. Without loss of generality we can assume that $s=1$. Each matrix entry of $k_i$, $i=1,2$, has $F$-norm bounded by $1$.  Hence for all $i,j$ we get
\[
 |g_{ij}|_F\le (ne^{\|\xi\|_W})^{\eee(F/\R)}\;\;
 \text{ as well as }\;\;
 |(g^{-1})_{ij}|_F\le (ne^{\|\xi\|_W})^{\eee(F/\R)}
\]
so that 
\[
 |t_i^{\pm1}|_F\le (ne^{\|\xi\|_W})^{\eee(F/\R)}=(n\|g\|_{\G})^{\eee(F/\R)}
 \;\; \text{ and } \;\;
 |t_iu_{ij}|_F\le (ne^{\|\xi\|_W})^{\eee(F/\R)}=(n\|g\|_{\G})^{\eee(F/\R)}
\]
for all $i=1,\ldots, n$, $j>i$. Since $\|m\|_{G(F)}=\max_{1\le k\le n}|t_k^\pm1|_F$, we get 
\[
 \|m\|_{\G}
 \le (n\|g\|_{\G})^{\eee(F/\R)},
 \]
and 
\[
 |u_{ij}|_F
 =|t_i^{-1}|_F |t_iu_{ij}|_F
\le (n\|g\|_{\G})^{2\eee(F/\R)} 
\]
for all $i<j$. 
To prove the last assertion write $u=k_1 e^{\xi'} k_2$ with $k_1, k_2\in\cpt_F$ and $\xi'\in\overline{\aaa^+}$ with $\sum_{i=1}^n \xi'_i=0$. Then
\[
 \tr e^{2\xi'} =\tr (\overline{u}^t u) = n+\sum_{i<j} |u_{ij}|_F^{2/\eee(F/\R)}
 \le n+ \frac{n^3(n-1)}{2}\|g\|_{\G}^3 .
\]
Hence $\max_{1\le k\le n} e^{2\xi'_k} \le n^4\|g\|_{\G}^3$. Since $\sum_i\xi'_i=0$, we also get $\max_{1\le k\le n} e^{-2\xi'_k}\le (n^4\|g\|_{\G}^3)^{n-1}$ so that
\[
 \|u\|_{\G}\le n^{4(n-1)} \|g\|_{\G}^{3(n-1)}.
\]
\end{proof}

\subsection{An equivalent norm}
In the case that $F$ is non-archimedean we define a second norm $\|\cdot\|_{E,2}$ on $U_0(E)\subseteq G(E)$ as follows: For $u=(u_{ij})_{i,j=1, \ldots, n}\in U_0(E)$, we set
\[
 \|u\|_{E,2}:= \max_{i,j}|u_{ij}|_E
\]
which is just the maximum of the $E$-adic matrix norm of $u$.
This is certainly always $\geq1$  and hence defines an abstract norm in the sense of~\cite{Ko05}. 
\begin{lemma}\label{second_norm_on_non_arch}
For all $u\in U_0(E)$ we have
\[
\log_{q_E} \|u\|_{E,2}\leq \log\|u\|_{G(E)}\leq (n-1)\log_{q_E} \|u\|_{E,2}.
\]
\end{lemma}
\begin{proof}
 Let $u\in U_0(E)$ and write its Cartan decomposition as $u=k_1\varpi_E^{\xi} k_2$, where $\xi=(\xi_1, \ldots, \xi_n)\in X_*^+(T_0)$ is uniquely determined. Note that $\det u=1$ so that $\xi_1+\ldots+ \xi_n=0$, $\xi_1\ge\ldots\ge \xi_n$, and therefore $\xi_n\leq 0\leq \xi_1$. Hence 
\[
\|u\|_{E,2}=\max_{i,j}|u_{ij}|_E=
\max_{i}q_E^{-\xi_i}=q_E^{-\xi_n}=q_E^{|\xi_n|}.
\]
On the other hand, $\log \|u\|_{G(E)}=\|\xi\|_W=\max\{|\xi_1|,|\xi_n|\}\leq (n-1)|\xi_n|$ since
\[
 0\le \xi_1 
 =-\xi_2-\ldots-\xi_n
 \le \sum_{i\ge2: \xi_i<0}\xi_i 
 \le -(n-1)\xi_n
 = (n-1)|\xi_n|. 
\]
Hence the asserted inequalities follow.
\end{proof}

\subsection{The constant term map}
For the moment let $S$ be an arbitrary finite set of places of $\F$ (not necessarily containing the archimedean place). Let $M\in\LLL$ and $P=MU\in\PPP(M)$. We define the constant term for $f\in C_c^{\infty}(G(\F_S))$ along $P$ by
\begin{equation}\label{eq:def:constant:term}
f^{(P)}(m)=\delta_P(m)^{1/2} \int_{\cpt_S}\int_{U(\F_S)} f(k^{-1} muk) \,du\,dk\;\;\;\;\text{for }
m\in M(\F_S).
\end{equation}
It is clear that  $f^{(P)}\in C_c^{\infty}(M(\F_S))$, and if $f$ is bi-$K_{S}$-invariant for some finite index subgroup $K_S\subseteq \cpt_S$, $f^{(P)}$ is bi-$K_S^M$-invariant for $K_S^M=K_S\cap M(\F_S)$. 

\begin{lemma}\label{lemma:behaviour:constant:term:non-arch}
\begin{enumerate}[label=(\roman{*})] 
\item Suppose $S=\{v\}$ consists of a single non-archimedean place and $f=\tau_{\xi}\in\HHH^{\leq\kappa_v}(\cpt_v)$. Then
\[
 \tau_{\xi}^{(P)}= \sum_{\substack{\zeta\in X_*^+(T_0) : \\ \|\zeta\|_W\leq \kappa_v}} c_P(\xi, \zeta) \tau_{\zeta}^M
\]
for suitable constants $c_P(\xi, \zeta)\geq0$ satisfying 
\[
 c_P(\xi, \zeta)\leq\begin{cases}
			   q_v^{a+b\kappa_v}					&\text{if } \F\text{ is ramified at }v,\\
			   q_v^{b\kappa_v}					&\text{if } \F\text{ is unramified at } v,
                    \end{cases}
\]
for $a, b>0$ some constants depending only on $n$ and $\F$. Here $\tau_{\zeta}^M:M(\F_v)\longrightarrow\C $ is the characteristic function of the double coset $\cpt_v^M\varpi_v^{\zeta} \cpt_v^M$. Moreover, 
\[
c_P(0,0)=\vol(\cpt_v\cap U(\F_v)). 
\]

\item Suppose $S$ consists only of the archimedean place. There exists a constant $c>0$ depending on $n$ and $R$ such that if $f\in C_c^{\infty}(G(\C)^1\sslash\cpt_\C)$ is supported in $\cpt_{\C} e^{V(R)} \cpt_{\C}$, then $f^{(P)}$ is supported in $\cpt_{\C}^M e^{V(c)} \cpt_{\C}^M\subseteq A_G\backslash M(\F_v)$, and
\[
 |f^{(P)}(m)|\ll_{R} \sup_{u\in U(\C)}|f(mu)| .
\]

\end{enumerate}
\end{lemma}
The lemma in particular implies that for almost all non-archimedean places $v$ (namely those for which $\vol(\cpt_v\cap U(\F_v))=1$, i.e., in particular for all unramified $v$), the constant term of the characteristic function of the maximal compact subgroup in $G(\F_v)$  along any parabolic in $\FFF$ equals the characteristic function of the maximal compact subgroup in the Levi component of that parabolic subgroup containing $T_0$.

\begin{proof}
\begin{enumerate}[label=(\roman{*})]
\item
Since the functions $\tau_{\zeta+\lambda}^M$ with $\zeta\in X_*^+(T_0)$ and $\lambda=(\lambda_1,\ldots,\lambda_1)$, $\lambda_1\in\Z$, generate the unramified Hecke algebra of $M(\F_v)$, we have an expansion
\[
 \tau_{\xi}^{(P)}=\sum_{\zeta\in  X_*^+(T_0):~\sum\zeta_i=\sum\xi_i} c_P(\xi, \zeta) \tau_{\zeta}^M
\]
for suitable coefficients $c_P(\xi, \zeta)\in \C$. (Note that if $\tau_{\zeta}^M$ appears with non-trivial coefficient, then necessarily $\sum_{i=1}^n\zeta_i=\sum_{i=1}^{n}\xi_i$.)
 Let $m\in M(\F_v)$ and $u\in U(\F_v)$. Then clearly $mu\in \supp \tau_{\xi}=\omega_{\xi, v}$ implies that $\|mu\|_{G(\F_v)}= e^{\|\xi\|_W}\leq e^{\kappa_v}$
 By Lemma~\ref{est_on_norms} this implies that $\|m\|_{G(\F_v)}\leq e^{\kappa_v}$ and $\|u\|_{G(\F_v)}\leq e^{2\kappa_v}$. Hence $m\in \omega_{\zeta, v}$ for some $\zeta\in X_*^+(T_0)$ with $\|\zeta\|_{G(\F_v)}\leq q_v^{\kappa_v}$ so that $c_P(\xi, \zeta)=0$ if $\|\zeta\|_W>\kappa_v$.
Moreover, 
\[
 u\in U(\F_v)\cap \bigcup_{\substack{\zeta\in X_*^+(T_0): \\ \|\zeta\|_W\leq 2\kappa_v}} \omega_{\zeta, v}
\]
so that
\begin{align*}
 \tau_{\xi}^{(P)}(m)
&\leq  \delta_P(m)^{1/2} \vol\Big(U(\F_v)\cap \bigcup_{\substack{\zeta\in X_*^+(T_0): \\ \|\zeta\|_W\leq 2\kappa_v}} \omega_{\zeta, v}\Big) 
\leq q_v^{\langle \rho, \xi\rangle}  \sum_{\zeta:  ~\|\zeta\|_W\leq 2\kappa_v} \vol(\omega_{\zeta, v})\\
&= q_v^{\langle \rho, \xi\rangle}  \vol(\cpt_v)\sum_{\zeta:  ~\|\zeta\|_W\leq 2\kappa_v} \deg \tau_{\zeta}.
\end{align*}
Now $\vol(\cpt_v)$ (with respect to the measure on $G(\F_v)$) equals $1$ if $\F$ is unramified at $v$ and is bounded by a constant $a$ depending only on $v$ and $\F$ if $\F$ is ramified at $v$. The remaining part is bounded by $q_v^{b\kappa_v}$ for some $b>0$ depending only on $n$ by Lemma~\ref{bounded_by_degree}. Hence the asserted bound for $c_P(\xi, \zeta)$ follows. The formula $c_P(0,0)$ follows immediately from the above calculation.

\item This follows similarly as  the first part from Lemma~\ref{est_on_norms}.
\end{enumerate}
\end{proof}

\section{Auxiliary estimates for the non-archimedean case}\label{sec:aux:est:non-arch}
In this section $F=\F_v$ for some non-archimedean place $v$ of $\F$.
To solve problem~\ref{non-arch} from the introduction (which shall be done in the next section), we first need to estimate the size of the norm of certain group elements in $G(F)$ in dependence of conjugacy classes of semisimple elements. We first treat the easier case that the conjugacy class splits in $G(F)$ and then deduce the quasi-split case from this.
\subsection{The split case}
Let $F$ and $E$ be as in Section~\ref{subsec:buildings} and Section~\ref{subsec:buildings:field:ext}.
If $\xi_F\in X_*^+(T_0)$  write $\omega_{F, \xi_F}=\cpt_F\varpi_F^{\xi_F}\cpt_F$, and  let accordingly $\omega_{E, \xi_E}=\cpt_E \varpi_E^{\xi_E} \cpt_E$ with  $\xi_E=\eee(E/F) \xi_F$. 

Let $\gamma_F=\sigma_F\nu_F\in G(F)$ be such that $\sigma_F$ is semisimple in $G(F)$ and $\nu_F\in G_{\sigma_F}(F)\cap U_0(F)$ is unipotent. 
We assume that $\gamma_F$ is of such a form that there exists a standard Levi subgroup $M_1\in\LLL$ such that $\sigma_F\in M_1(F)$ is regular elliptic. (There is always a $G(F)$-conjugate of $\gamma_F$ with this property.)
Let $G_{\sigma_F}(F)$ be the centraliser of $\sigma_F$ in $G(F)$ and $T_{\sigma_F}(F)\subseteq G_{\sigma_F}(F)$ the maximal torus with $\sigma_F\in T_{\sigma_F}(F)$.
There exists a Galois extension $E/F$ of degree $\leq n!$ such that $T_{\sigma_F}$ splits over $E$ (cf.~\cite{JKZ13}). If the residue characteristic of $F$ is greater than $n!$ (i.e.\ if $v\not\in S_{\text{bad},1}$), then $E/F$ is at worst tamely ramified.
In any case, there exists $y\in G(E)$ such that $yT_{\sigma_F}(E) y^{-1}= :A_{\sigma_F}(E)$ is a split diagonal torus and $M_{\sigma_F}(E):=yG_{\sigma_F}(E) y^{-1}$ is the Levi component of a standard parabolic subgroup in $G(E)$. We can assume that $y\in U_0(E)\cpt_E$ by Iwasawa decomposition with respect to the standard minimal parabolic subgroup.
Then $\delta:=y\sigma_F y^{-1}\in A_{\sigma_F}(E)$ and $M_{\sigma_F}(E)$ is the centraliser of $\delta$ in $G(E)$.
We abbreviate $M_{\sigma_F}=:M_{\delta}$, $A_{ \sigma_F}=:A_{\delta}$, and let $P_{\delta}=M_{\delta} U_{\delta}\in \FFF_{\text{std}}$ be the standard parabolic subgroup with Levi component $M_{\delta}$.
Without loss of generality we can assume that 
\[
y=v_0k_0
\]
with $v_0\in U_{\delta}(E)$ and $k_0\in\cpt_E$.

\begin{proposition}\label{distance_to_centr_non-arch_split}
Let the notation be as above.  
There exists a constant $a>0$ depending only on $n$, but not on $F$,  $\gamma$ or $\delta$ such that the following holds:
 Suppose that $\mu\in X_*^{ +}(T_0)$, $x\in G(E)$, and $u\in\UUU_{M_{\delta}(E)}$ are such that $x^{-1} \delta ux\in \omega_{E, \mu}$. Let $x=mvk\in M_{\delta}(F)U_{\delta}(F)\cpt_F$ be an Iwasawa decomposition with respect to $P_{\delta}$. Then
\begin{align*}
 \log\|m^{-1}x\|_{G(E)}			& \leq a\Big( \|\mu\|_W+\sum_{\alpha\in \Phi^+:~\alpha(\delta)\neq1}|\val_E(1-\alpha^{-1}(\delta))|\Big),\\
 \log\|m^{-1} u m\|_{G(E)}		& \leq a\Big( \|\mu\|_W+\sum_{\alpha\in \Phi^+:~\alpha(\delta)\neq1}|\val_E(1-\alpha^{-1}(\delta))|\Big).
\end{align*}
Here $\val_E:E\longrightarrow \Z\cup\{\infty\}$ denotes the discrete valuation on $E$ normalised by $v_E(\varpi_E)=1$.
\end{proposition}

\begin{rem}\label{rem:positive:discriminants}
 We have
\[
 \sum_{\alpha\in \Phi^+:~\alpha(\delta)\neq1}|\val_E(1-\alpha^{-1}(\delta))|
=2\log_{q_E}\Delta_E^-(\delta^{-1}) + \log_{q_E}|D^G(\delta^{-1})|_E
\]
in the notation of Section~\ref{sec:fine:geom:exp}.
\end{rem}

\begin{proof}[Proof of Proposition~\ref{distance_to_centr_non-arch_split}]
Let $u\in \UUU_{M_{\delta}(E)}$ be such that $x^{-1}\delta ux\in \omega_{E,\mu}$. Conjugating with an element of $\cpt^{M_{\delta}}_E$ (this does not change norms) if necessary, we can assume that $u\in U_0^{M_{\delta}}(E)$.
 If $u=1$, the assertion is proved in~\cite[Lemma 7.9]{ShTe12}. To prove our slightly more general assertion, we basically follow their proof of that lemma and modify it as necessary.  
Write $x= t\widetilde{n}k$ for the Iwasawa decomposition with $t\in T_0(E)$, $\widetilde{n}\in U_0(E)$, $k\in\cpt_E$. We can further decompose $\widetilde{n}=n_1v$ with uniquely determined $n_1\in U_0^{M_{\delta}}(E)$ and $v\in U_{\delta}(E)$ so that $m=tn_1$ and $m^{-1}x= vk$, $m^{-1}um=n_1^{-1}t^{-1}utn_1$. 
By  assumption, 
\[
\delta (\delta^{-1}v^{-1} \delta u' v)
=(\delta u') ((\delta u')^{-1}v^{-1} \delta u' v)
=v^{-1} \delta u' v
=\widetilde{n}^{-1} t^{-1} \delta u t\widetilde{n}
\in \omega_{E,\mu},
\]
 where we write $u':=n_1^{-1}t^{-1} u tn_1\in U_0^{M_{\delta}}(E)$. 
Then the first expression on the left hand side actually gives the Iwasawa decomposition of $x^{-1}\delta u x$ with respect to the minimal standard parabolic subgroup, and the second expression gives the Iwasawa decomposition with respect to the parabolic $P_{\delta}$.
Hence Lemma~\ref{est_on_norms}~\ref{lemma_est_on_norms_iwasawa} gives
\[
 \|\delta\|_{G(E)}, ~  \|\delta u'\|_{G(E)}\leq e^{\|\mu\|_W}
\]
and
\[
 \|(\delta^{-1}v^{-1} \delta u' v)\|_{G(E)}, ~ \|((\delta u')^{-1}v^{-1} \delta u' v)\|_{G(E)} \leq e^{2\|\mu\|_W}.
\]
The first two inequalities immediately imply that 
\[
\|m^{-1}u m\|_{G(E)}
=\|u'\|_{G(E)}
\leq e^{2\|\mu\|_W}.
\]
Let $\lambda_1\in X_*^+(T_0)$ with $\|\lambda_1\|_W\leq 2\|\mu\|_W$ be such that $\delta^{-1} v^{-1}\delta u' v\in \cpt_E \varpi_E^{\lambda_1 }\cpt_E=\omega_{E, \lambda_1}$

We now use the Chevalley decomposition of $U_0$:
For every positive root $\alpha\in\Phi^+$ fix an isomorphism $u_{\alpha}: \G_a\longrightarrow U_{\alpha}\subseteq U_0$ for $U_{\alpha}\subseteq U_0$ denoting the $1$-dimensional subgroup of $U_0$ corresponding to the root $\alpha$ and $\G_a$ denotes the one-dimensional affine group scheme defined over $\Z$. Fix an order on the set of positive roots, $\Phi^+=\{\alpha_1, \ldots, \alpha_{|\Phi^+|}\}$ so that $\alpha_1, \ldots, \alpha_{n-1}$ are the simple roots in the usual ordering. With this we obtain isomorphisms (of schemes)
\[
 \G_a^{|\Phi^+|} \xrightarrow{\;\;\; u_{\alpha_1}\times\ldots\times u_{|\Phi^+|}\;\;\;} 
U_{\alpha_1}\times\ldots\times U_{\alpha_{|\Phi^+|}}
\xrightarrow{\;\;\;(v_{1}, \ldots, v_{|\Phi^+|})\mapsto v_1\cdot\ldots\cdot v_{|\Phi^+|}\;\;\;} U_0.
\]
Hence we can write $v=u_{\alpha_1}(X_1)\cdot\ldots\cdot u_{\alpha_{|\Phi^+|}}(X_{|\Phi^+|})$ and $u'=u_{\alpha_1}(Y_1)\cdot\ldots\cdot u_{\alpha_{|\Phi^+|}}(Y_{|\Phi^+|})$ for uniquely determined $X_1, \ldots, X_{|\Phi^+|},Y_1, \ldots, Y_{|\Phi^+|}\in E$. Note that $X_i=0$ if $U_{\alpha_i}\cap U_{\delta}=\{1\}$, and similarly $Y_i=0$ if $U_{\alpha_i}\cap U_0^{M_1}=\{1\}$. Moreover, $\alpha_i(\delta)=1$ if $U_{\alpha_i}\cap U_\delta=\{1\}$. In particular, $X_i\neq0$ implies $Y_i=0$, and conversely, $Y_i\neq0$ implies $X_i=0$. 
Then $n^{-1} = u_{\alpha_{|\Phi^+|}}(-X_{|\Phi^+|})\cdot\ldots\cdot u_{\alpha_1}(-X_1)$ and
\begin{align*}
 \delta^{-1} v^{-1}\delta u' v
&= \bigg[\delta^{-1}\bigg(\prod_{i=|\Phi^+|}^{1} u_{\alpha_i}(-X_i)\bigg) \delta\bigg] \bigg[\prod_{i=1}^{|\Phi^+|} u_{\alpha_i}(Y_i)\bigg] \bigg[\prod_{i=1}^{|\Phi^+|} u_{\alpha_i}(X_i)\bigg]\\
&=\bigg[\prod_{i=|\Phi^+|}^{1} u_{\alpha_i}(-\alpha_i^{-1}(\delta) X_i)\bigg] \bigg[\prod_{i=1}^{|\Phi^+|} u_{\alpha_i}(Y_i)\bigg] \bigg[\prod_{i=1}^{|\Phi^+|} u_{\alpha_i}(X_i)\bigg]\\
&=\prod_{i=1}^{|\Phi^+|} u_{\alpha_i}\big((1-\alpha_i^{-1}(\delta))X_i +Y_i +P_i(X_1, \ldots, X_{|\Phi^+|} , Y_1, \ldots, Y_{|\Phi^+|})\big),
\end{align*}
where $P_i$ is a polynomial in the variables $X_1, \ldots, X_{|\Phi^+|}, Y_1, \ldots, Y_{|\Phi^+|}$ such that only those  $X_j, Y_k$ occur non-trivially in $P_i$ for which $\alpha_j<\alpha_i$ and $\alpha_k<\alpha_i$. Further, $P_i$ has no constant term.
Note that $(1-\alpha_i^{-1}(\delta))X_i +Y_i$ either equals $(1-\alpha_i^{-1}(\delta))X_i$ or $Y_i$ depending on whether $U_{\alpha_i}\cap U_0^{M_{\delta}}=\{1\}$ or $U_ {\alpha_i}\cap U_{\delta}=\{1\}$. 
One can now argue inductively as in the proof of~\cite[Lemma 7.9]{ShTe12} to conclude from the above expression of $ \delta^{-1} v^{-1}\delta u' v$ and the fact that $ \delta^{-1} v^{-1}\delta u' v\in\omega_{E, \lambda_1}$  that there exists a constant $a\geq0$ depending only on $n$ such that
\[
\val_E(X_i)\geq - \sum_{\substack{\alpha\in\Phi^+:\\\alpha(\delta)\neq1}} a^{|\Phi^+|} \Big(|\val_E(1-\alpha^{-1}(\delta))|+a \|\lambda_1\|_W\Big)=:-A(\delta, \lambda_1),
 ~~~~~
 \text{ and }
 ~~~~
 \val_E(Y_i)\geq -A(\delta, \lambda_1)
\]
for all $i=1, \ldots, |\Phi^+|$. 
It follows that
\[
\log_{q_E} \|v\|_{E,2}
=\log_{q_E}\max_{i,j} |v_{i,j}|_E
\leq \max \big\{\log_{q_E} |X_i|_E^{k}:~ i=1, \ldots, |\Phi^+|,~k=0,\ldots,n\big\}
\leq n A(\delta, \lambda_1)
\]
so that 
\[
\log \|v\|_{G(E)} 
\leq n^{2(n-1)} A(\delta, \lambda_1).
\]
Using the estimate $\|\lambda_1\|_W\leq 2\|\mu\|_W$ and the fact that $u'=n_1^{-1}t^{-1} u tn_1\in U_0^{M_{\delta}}(E)$ and $\|m^{-1}x\|_{G(E)}=\|v\|_{G(E)}=\|n_1^{-1}t^{-1} x\|_{G(E)}$ the assertion of the lemma follows since $x=(tn_1)vk$ is an Iwasawa decomposition of $x$ with respect to $P_{\delta}$.
\end{proof}

\begin{cor}[to the proof of Proposition~\ref{distance_to_centr_non-arch_split}]\label{bound_on_y}
 Let $\sigma_F$, $\delta$, and $y=v_0k_0$ be as at the beginning of this section, and suppose that $\sigma_Fu$ is in $G(F)$ conjugate to some element in $\omega_{F,\xi_F}$ for some $\xi_F=(\xi_1,\ldots,\xi_n)\in X_*^+(T_0)$ and some unipotent $u\in G_{\sigma_F}(F)$. Further suppose that the matrix entries of $\sigma_F$ are $F$-integral. Then
\[
\log\|y\|_{G(E)}\leq c_1\|\xi_F\|_W +c_2|\log_{q_F} |D^G(\sigma_F)|_F| 
\]
with constants $c_1, c_2\geq0$ depending only on $n$.
\end{cor}

\begin{proof}
We first estimate $\|\sigma\|_{G(E)}$. For that let $\zeta=(\zeta_1,\ldots,\zeta_n)\in X_*^+(T_0)$ be such that $\sigma\in\omega_{F,\zeta}$. Note that by assumption $\zeta_1\ge\ldots\ge\zeta_n\ge0$ so that $\|\zeta\|_W=\zeta_1$. Hence 
\[
\log \|\sigma\|_{G(E)}=\zeta_1
\le \zeta_1+\ldots+\zeta_n
= \log_{q_E} |\det\sigma|_F^{-1}
= \xi_1+\ldots+\xi_n
\le n\|\xi_F\|_W
\] 
We have $\|y\|_{G(E)}=\|v_0\|_{G(E)}$ and 
\[
\| \sigma\|_{G(E)}
=\| v_0^{-1}\delta v_0\|_{G(E)}
=\|\delta (\delta^{-1}v_0^{-1}\delta v_0)\|_{G(E)}.
\]
Lemma~\ref{est_on_norms} implies that $\|\delta^{-1}v_0^{-1}\delta v_0\|_{G(E)}\le \|\sigma\|_{G(E)}^2$. The proof of the previous proposition (with $v=v_0$ and $u'=1$) together with Remark~\ref{rem:positive:discriminants} then implies that for some $a_1>0$
\[
 \log\|y\|_{G(E)}
 \le a_1\left(\|\xi_F\|_W + 2\log_{q_E}\Delta_E^-(\delta^{-1}) + \log_{q_E}|D^G(\delta^{-1})|_E\right)
\]
Now $|D^G(\delta^{-1})|_E=|\delta_(\delta)^{-1}|_E|D^G(\delta)|_E =|\delta_(\delta)^{-1}|_E|D^G(\sigma_F)|_F^{\eee(E/F)}$ and $|\delta_0(\delta^{-1}|_E$ can be bounded by Lemma~\ref{lemma:bound:ev}.
By Corollary~\ref{cor:bound:pmdiscr} $\Delta_E^-(\delta)$ is bounded from above by
\[
 q_E^{a_2\|\xi_F\|_W} |D^G(\delta)|_E^{-1}
\]
for some constant $a_2>0$ depending only on $n$.
Since $\Delta_E^-(\delta^{-1})$ can be written as a product of $\Delta_E^-(\delta)$ and a suitable product of eigenvalues of $\delta^{\pm1}$ (which can again be bounded by Lemma~\ref{lemma:bound:ev}), the asserted estimate for $\|y\|_{G(E)}$ follows.  
\end{proof}

\subsection{Non-split non-archimedean case}
We keep the notation from the last section. The purpose of this section is to prove the non-split analogue of Proposition~\ref{distance_to_centr_non-arch_split}:
\begin{cor}\label{prop_distance_to_torus}
There exists $a, b\geq0$ depending only $n$, but not on  $F$ or $\gamma_F$ such that the following holds. 
Suppose $x\in G(F)$  and $u\in \UUU_{G_{\sigma_F}}(F)$ are such that  $x^{-1}\sigma_F u x\in \omega_{F, \xi_F}$. 
Then there exists $g\in G_{\sigma_F}(F)$ which is independent of $u$ such that
\begin{align}
\log \|gx\|_{G(F)} 			&	 \leq a\left|\log_{q_F}|D^G(\sigma_F)|_F\right| + b\|\xi_F\|_W+\delta \text{ and }\\
\log \|g u g^{-1}\|_{G(F)} 		&	 \leq a\left|\log_{q_F}|D^G(\sigma_F)|_F\right| +  b\|\xi_F\|_W+\delta,
\end{align}
where $\delta=0$ unless the residue characteristic of $F$ is less than or equal to $n!$ in which case $\delta>0$ can be chosen independentl of $F$.
\end{cor}

\begin{proof}
Let $x_1:= yx\in G(E)$ and $u_1:=yuy^{-1}\in\UUU_{M_{\delta}}(E)\subseteq M_{\delta}(E)$ so that by assumption $x_1^{-1}\delta u_1x_1\in\omega_{E, \xi_E}$. 
There exists $k\in\cpt_E\cap M_{\delta}(E)$ such that $u_2:=ku_1k^{-1}\in U_0^{M_{\delta}}$. Put $x_2:=kx_1\in G(E)$ so that by assumption $x_2^{-1}\delta u_2 x_2\in\omega_{E, \xi_E}$. Hence we may apply Proposition~\ref{distance_to_centr_non-arch_split} so that there exists a constant $a_1>0$ depending only on $n$, and some element $g_1\in M_{\delta}(E)$ such that 
 \begin{align*}
  \log  \|g_1x_2\|_{G(E)}   	& \leq a_1\bigg( \|\xi_E\|_W+\sum_{\alpha\in \Phi^+:~\alpha(\delta)\neq1}|\val_E(1-\alpha^{-1}(\delta))|\bigg),\;\;\text{ and }\\
\log\|g_1 u_2 g_1^{-1}\|_{G(E)} & \leq a_1\bigg( \|\xi_E\|_W+\sum_{\alpha\in \Phi^+:~\alpha(\delta)\neq1}|\val_E(1-\alpha^{-1}(\delta))|\bigg).
 \end{align*}
 Since $[E:F]\leq n!$, we get $\|\xi_E\|_W\leq n!\|\xi_F\|_W$. It further follows from Remark~\ref{rem:positive:discriminants} and the proof of Corollary~\ref{bound_on_y} that
\[
\sum_{\alpha\in \Phi^+:~\alpha(\delta)\neq1}|\val_E(1-\alpha^{-1}(\delta))|
\leq c\big(\log_{q_F} |D^G(\sigma_F)|_F^{-1} + \|\xi_F\|_W \big)
\] 
for $c>0$ some constant depending only on $n$.
In terms of the distance on the building, these inequalities yield
\begin{align*} 
d_{\BBB(G,E)}(x_2, g_1^{-1})			& 	  \leq a_2\Big(\log_{q_F}|D^G(\sigma_F)|_F^{-1} +\|\xi_F\|_W\Big),\;\;\text{ and } \\
 d_{\BBB(G,E)} (u_2 g_1^{-1}, g_1^{-1})		& 	  \leq a_2\Big(\log_{q_F}|D^G(\sigma_F)|_F^{-1} +\|\xi_F\|_W\Big)
\end{align*}
for some constant $a_2>0$ depending only on $n$.
Plugging in the definition of $x_2$ and $u_2$, we obtain
\begin{align*}
  d_{\BBB(G, E)}(x, (y g_2)^{-1})			&	  \leq a_2\Big(\log_{q_F}|D^G(\sigma_F)|_F^{-1} +\|\xi_F\|_W\Big),\;\;\text{ and }\\
 d_{\BBB(G,E)} (u (y g_2)^{-1}, (yg_2)^{-1})		&	  \leq a_2\Big(\log_{q_F}|D^G(\sigma_F)|_F^{-1} +\|\xi_F\|_W\Big)
\end{align*}
for $g_2:=y^{-1} g_1 k_1 y\in G_{\sigma}(E)$. 
Since the norm $\|\cdot\|_{G(E)}$ is submultiplicative, Corollary~\ref{bound_on_y} implies that there exists $a_3>0$ depending only on $n$ such that
\begin{align*}
  d_{\BBB(G, E)}(x, g_2^{-1})			&	  \leq C(\xi_F,\sigma_F):=a_3\Big(\log_{q_F}|D^G(\sigma_F)|_F^{-1} +\|\xi_F\|_W\Big),\;\;\text{ and }\\
 d_{\BBB(G,E)} (u g_2^{-1}, g_2^{-1})		&	  \leq C(\xi_F,\sigma_F).
\end{align*}
Note that if $C(\xi_F, \sigma_F)\neq0$, it is bounded away from $0$ by a constant independent of $F$.
Consider the isometric action of the Galois group $\Gal(E/F)$ on $\BBB(G,E)$ and $\BBB(G_{\sigma_F},E)$.
Suppose first that $E$ is tamely ramified over $F$. Then $\BBB(G,F)=\BBB(G,E)^{\Gal(E/F)}$ and $\BBB(G_{\sigma_F},F)=\BBB(G_{\sigma_F},E)^{\Gal(E/F)}$ (cf.~\cite{Pr01}).
Since $x\in G(F)$, we have for every $\tau\in\Gal(E/F)$,
\[
 d_{\BBB(G,E)}(x, g_2^{-1})
 = d_{\BBB(G,E)}(x, \tau(g_2^{-1})),
\]
and $\tau(g_2)^{-1}$ again commutes with $\sigma_F$.
Let $ \CCC_0:=\{\tau(g_2^{-1})\mid \tau\in\Gal(E/F)\}\subseteq G_{\sigma_F}(E)$, and let $\CCC$ be the closed convex hull of $\CCC_0$ inside the building $\BBB(G_{\sigma_F},E)$. By the triangle inequality, $d_{\BBB(G,E)}(\tau_1(g_2^{-1}),\tau_2(g_2^{-1}))\leq 2C(\xi_F,\sigma_F)$ for all $\tau_1, \tau_2\in\Gal(E/F)$.
As explained in the proof of~\cite[Lemma 7.11]{ShTe12}, the set of successively constructed midpoints is dense in $\CCC$ so that for every $z\in\CCC$ we have $d_{\BBB(G,E)}(z,\tau(g_2^{-1}))\leq 2C(\xi_F,\sigma_F)$ for every $\tau\in\Gal(E/F)$.
Moreover, $\CCC$ is invariant under the action of $\Gal(E/F)$, and hence has a fixed point $z_0\in\CCC^{\Gal(E/F)}\subseteq \BBB(G_{\sigma_F},F)$ satisfying $d_{\BBB(G,E)}(z_0,g_2^{-1})\leq 2C(\xi_F,\sigma_F)$.

Now by definition of the building, there exists $g\in G_{\sigma_F}(F)$ with $d_{\BBB(G_{\sigma_F},F)}(z_0,g^{-1})\leq 1$. Hence
\begin{align*}
\log \|gx\|_{G(F)} & = d_{\BBB(G,F)}(x,g^{-1})
=\eee(E/F) d_{\BBB(G,E)}(x,g^{-1})\\
& \leq n! \big(d_{\BBB(G,E)}(x,g_2^{-1})+d_{\BBB(G,E)}(g_2^{-1},z_0)+ d_{\BBB(G,E)}(z_0,g^{-1})\big)\\
& \leq 3 n! C(\xi_F,\sigma_F)+1
\end{align*}
which proves the asserted upper bound for $x$ if $C(\xi_F, \sigma_F)\neq0$ as then $C(\xi_F, \sigma_F)$ is bounded away from $0$. If on the other hand $C(\xi_F, \sigma_F)=0$, then the points $g_2$ and $z_0$ actually satisfy $d_{\BBB(G,E)}(x,g_2^{-1})=d_{\BBB(G,E)}(z_0,g_2^{-1})=0$. This means that $z_0\in\cpt_E^{\Gal(E/F)}=\cpt_F$, and therefore 
\[
 \log \|z_0^{-1} x\|_{G(F)}\leq \eee(E/F)( d_{\BBB(G,E)}(x,g_2^{-1})+d_{\BBB(G,E)}(z_0,g_2^{-1}))=0
\]
so that the assertion is also true for $C(\xi_F, \sigma_F)=0$.

For the unipotent element we similarly get
\begin{align*}
\|g u g^{-1}\|_{G(F)} & = \eee(E/F) d_{\BBB(G,E)} (ug^{-1}, g^{-1})\\
& \leq n! \big( d_{\BBB(G,E)} (ug^{-1}, g_2^{-1})+ d_{\BBB(G,E)} (g_2^{-1}, g^{-1})\big)\\
& = n! \big( d_{\BBB(G,E)} (g^{-1}, u^{-1} g_2^{-1})+ d_{\BBB(G,E)} (g_2^{-1}, g^{-1})\big)\\
& \leq n! \big(d_{\BBB(G,E)}(g^{-1},g_2^{-1})+ d_{\BBB(G,E)} (g_2^{-1}, u^{-1} g_2^{-1})+ d_{\BBB(G,E)} (g_2^{-1}, g^{-1})\big)\\
& = 2 n! d_{\BBB(G,E)} (g^{-1},g_2^{-1}) + n! d_{\BBB(G,E)}(ug_2^{-1},g_2^{-1})\\
& \leq 3 n! C(\xi_F,\sigma_F)+1
\end{align*}
proving the asserted upper bound for $u$ if $C(\xi_F, \sigma_F)\neq0$. The case $C(\xi_F, \sigma_F)=0$ follows as above, hence finishing the corollary in the unramified or tamely ramified case.

If, on the other hand, $E/F$ is not tamely ramified, this construction gives the existence of a point $z_0\in \CCC^{\Gal(E/F)}\subseteq \BBB(G_{\sigma_F}, E)^{\Gal(E/F)}$ with $d_{\BBB(G,E)}(z_0,g_2^{-1})\leq 2C(\xi_F,\sigma_F)$. By~\cite[\S 5.2]{Rou77} there exists a constant $c$ (which can be bounded in terms of $n$) such that 
$d_{\BBB(G,E)}(z_0, \BBB(G_{\sigma},F))\leq c\fff(E/F)\leq c\cdot n!$. Hence we can find a point $z_0'\in \BBB(G_{\sigma_F},F)$ and a constant $\delta'$ depending only on $n$ such that
$d_{\BBB(G,E)}(z_0',g_2^{-1})\leq 2 C(\xi_F,\sigma_F)+\delta'$. Hence replacing in this case $z_0$ by $z_0'$ and  $C(\xi_F,\sigma_F)$ by $C(\xi_F,\sigma_F)+\delta'$ we can proceed as in the tamely ramified case above.
\end{proof}

\section{Based root data and semisimple centralisers}\label{subsec:equiv:classes:root:data}
We later need to bound unipotent orbital integrals for unipotent conjugacy classes contained in centralisers of semisimple elements in $G(\F_v)$ uniformly in $v$. The purpose of this section is to describe the centralisers of semisimple elements in $G(\F)$ and $G(\F_v)$ in a form suitable to us.

\subsection{Based root data}\label{sec:based:root:data}
The notation in this section is inspired by~\cite{CGP_PRGrps,SprLAG}. 
For the rest of this section let $F$ be an arbitrary local or global field of characteristic $0$.
Let $\OOO_F$ be the ring of integers of $F$ if $F$ is non-archimedean or global. We also set $\OOO_F=\Z$ if $F=\R$, and $\OOO_F=\Z+\Z i$ if $F=\C$ so that we do not need to distinguish between archimedean and non-archimedean fields. 
Let $E$ be a finite reduced $F$-algebra, that is, $E=\prod_{i\in I} E_i$ for $I$ a finite index set and $E_i$ finite field extensions over $F$. Let $d_i=[E_i:F]$.

Suppose we are given a collection of based root data (cf.~\cite[\S 16.2]{SprLAG}) $\Phi_{m_i}=(X_i, \Delta_i, X_i^{\vee}, \Delta_i^{\vee})$ of type $A_{m_i-1}$ for some integers $m_i\geq1$. Let $_{E_i}\Phi_{m_i}$ denote $\Phi_{m_i}$ considered as a based root datum relative to $E_i$. 
We call
\[
 _E\Phi=\prod_{i\in I} ~_{E_i}\Phi_{m_i}
\]
a product of based root data of type $A$ relative to $E$, and define $n:=\sum_{i\in I}m_id_i$ to be the dimension of $_E\Phi$. 
We also set $\Phi=\prod_{i\in I} \Phi_{m_i}$ for the product of based root data without reference to the base field.

Let $H'_i=\GL_{m_i}$, and let $S'_i\subseteq H'_i$ be an $E_i$-split maximal torus, and $B'\supseteq S'_i$ a Borel subgroup such that $(S'_i(E_i), B'_i(E_i))$ corresponds to $_{E_i}\Phi_i$. 
Put
\[
 H'=\prod_{i\in I} H'_i,\;\; 
S'=\prod_{i\in I} S'_i,\;\; 
B'=\prod_{i\in I} B'_i.
\]
We shall say that a subgroup of $H'$ is an $E$-torus, $E$-parabolic subgroup, etc., if it is the product over $i\in I$ of $E_i$-tori, $E_i$-parabolic subgroups, etc.
 Similarly, if we talk about any properties of an $E$-subgroup of $H'$ (such as being $E$-split, for example), we always mean that any factor has this property as an $E_i$-subgroup of $H'_i$.
We put
\[
 H=\Res_{E/F}H',\;\;
S=\Res_{E/F}S',\;\;
B=\Res_{E/F}B'
\]
where $\Res_{E/F}$ denotes the Weil restriction of scalars from $E$ to $F$.
Then $S(F)$ is a maximal torus in $H(F)$ and equals the centraliser in $H$ of its maximal $F$-split component $S_1\subseteq S$. The based root datum determined by $(S_1, B)$ relative to $F$ equals $\prod_{i\in I}\Phi_{m_i}$.
Note that $\Res_{E/F}$ gives a bijection between the set of $E$-parabolic subgroups in $H'$ containing $S'$ and set of $F$-parabolic subgroups in $H$ containing $S$.
Let $\RRR^n_F$ denote the set of all products of based root data of type $A$ and dimension $n$ over $F$. 

\begin{definition}
We call two elements $_{E^j}\Phi_{m_i}\in \RRR^n_F$, $j=1,2$, \emph{equivalent} if the following holds:
\begin{enumerate}
 \item $|I^1|=|I^2|$,
\item\label{permut} $\exists$ permutation $\pi: I^1\longrightarrow I^2$ such that 
\begin{itemize}
 \item $\forall i\in I^1:$ the fields $E_i^1$ and $E_{\pi(i)}^2$ are Galois conjugate, and
\item $\forall i\in I^1:~ m_i^1=m_{\pi(i)}^2$ 
\end{itemize}
\end{enumerate}
\end{definition}
Here we say that two fields $K_1, K_2$ over $F$ are \emph{Galois conjugate} if there exists a field automorphism $K_1\longrightarrow K_2$ fixing $F$.

For $_E\Phi\in\RRR^n_F$ we denote by $[_E\Phi]\subseteq \RRR^n_F$ the equivalence class of $_E\Phi$. Let $\RRR^n_F/\!\sim$ be the set of all equivalence classes in $\RRR^n_F$.
If $\chi(T)\in F[T]$ is a monic polynomial of degree $n$ we can attach an equivalence class $[_E\Phi]$ of degree $n$ to $\chi$ as follows: Let $\chi(T)=\prod_{i\in I} \chi_i(T)^{m_i}$ be the factorisation in $F[T]$ of $\chi$ into irreducible polynomials $\chi_i$ with $\chi_i\neq\chi_j$ for $i\neq j$ and $m_i\ge 1$ suitable integers. Let $E_i$ be field extension of degree $d_i=\deg \chi_i$ over $F$ such that $\chi_i$ has a root in $E_i$, and let $E=\prod_{i\in I} E_i$. We then attach the  equivalence class of the based root datum $_E\Phi=\prod_{i\in I} ~_{E_i}\Phi_{m_i}\in\RRR^n_F$ to $\chi$ and write $\Omega(\chi):=[_E\Phi]\subseteq \RRR^n_F$. Note that $\Omega(\chi)$ is indeed well-defined.

To make our choices more definite, we pick specific representatives for the classes in $\RRR^n_F/\sim$.
Write $_E\Psi=\prod_{i=1}^t ~_{E_i}\Psi_{m_i}\in \RRR^n_F$, for the product of based root data corresponding to $S'_i=T_0^{H'_i}\subseteq ~H'_i=\GL_{m_i}$ the maximal torus of diagonal elements and $B'_i\supseteq ~S'_i$ the Borel subgroup consisting of upper triangular matrices. It is clear that any equivalence class in $\RRR^n_F/\!\sim$ contains a representative of this form, and we will usually use this representative for computations. 
We further fix an integral basis of $E$ over $F$ (if $F$ is archimedean and $F=E_i$, then we take $1$ as an integral basis; if $F=\R$ and $E_i=\C$ we take $\{1,i\}$).
This basis gives a realisation of the restriction of scalars $\Res_{E/F}$. In particular it fixes an embedding of $E_i^{\times}$ into $\GL_{d_i}(F)$, and more generally of $\GL_{m_i}(E_i)$ into $\GL_{m_id_i}(F)$. Hence we can view $H(F)=(\Res_{E/F} H')(F)$ as embedded into the Levi subgroup $M_H:=\GL_{m_1d_1}(F)\times\ldots\times\GL_{m_td_t}(F)$ which in turn we embed diagonally in $\GL_n(F)$. Note that $M_H(F)$ is the smallest Levi subgroup of $\GL_n$ containing $H$. 
Our choice of basis yields $H(\OOO_F)=H'(\OOO_E)$, $S(\OOO_F)=S'(\OOO_E)$, and $B(\OOO_F)=B'(\OOO_E)$. 

The semisimple conjugacy classes in $\GL_n$ are in bijection to monic polynomials of degree $n$ by mapping a conjugacy class to the characteristic polynomial of one of its members. If $[\sigma]\subseteq G(F)$ is such a conjugacy class, we write $\Omega([\sigma])=\Omega(\chi)$ with $\chi$ the characteristic polynomial of $\sigma$. Then the following is immediate.

\begin{proposition}
The map $\Omega$ from  the set of semisimple conjugacy classes in $G(F)$ to equivalence classes in $\RRR^n_F/\!\sim$ is surjective. 
Moreover, if $_E\Psi\in\Omega([\sigma])$ is the representative defined above, we can find $\sigma\in [\sigma]$ such that $G_{\sigma}(F)=H(F)$.
\end{proposition}

\begin{rem}
 \begin{enumerate}
\item  $F$-split regular conjugacy classes are mapped to $[_E\Phi]$ with $I=\{1,\ldots, n\}$, $E=F\oplus\ldots\oplus F=F^n$, and $m_i=1$ for all $i\in I$.

\item $F$-elliptic conjugacy classes are mapped to $[_E\Phi]$ with $I=\{1\}$, $E=E_1$ a field extension of degree $d_1$ over $F$ with $d_1$ dividing $n$, and $m_1=n/d_1$.

\item Regular $F$-elliptic conjugacy classes are mapped to $[_E\Phi]$ with $I=\{1\}$, $m_1=1$, and $E=E_1$ a field extension of degree $n$ over $F$.
 \end{enumerate}
\end{rem}

\subsection{Completion of based root data}
Let $F$ be a number field and $v$ be a place of $F$. Let $_E\Phi\in\RRR^n_{F}$. 
For each $i\in I$, let $W_v^i$ be the set of places of $E_i$ above $v$. The completion of $E$ at $v$ then equals
\[
 E_v=\prod_{i\in I,~ w\in W_v^i} E_{i,w},
\]
where $E_{i,w}$ denotes the completion of $E_i$ at $w$.
This is a finite reduced $F_v$-algebra. Let
\[
 _{E_v}\Phi=\prod_{i\in I, ~w\in W_v^i} ~_{E_{i,w}}\Phi_{m_i}\in \RRR^n_{F_v}
\]
be the \emph{completion of $_E\Phi$ at $v$}. This map preserves the equivalence relations on the respective sets, that is
\[
 \RRR^n_{F}\longrightarrow \RRR^n_{F_v}
\]
 descends to a map
\[
 \RRR^n_{F}/\!\sim\longrightarrow \RRR^n_{F_v}/\!\sim.
\]
Accordingly, let $H'_v=\prod_{i\in I,~w\in W_v^i} H'_i$, $S'_v=\prod_{i\in I,~w\in W_v^i} S'_i$, $H_v=\Res_{E_v/F_v} H'_v$, and so on.

\begin{rem}
\begin{enumerate}
\item The cardinality of $\RRR^n_{F_v}/\!\sim $ is finite for every $v$. In fact, its cardinality is bounded by a number depending only on $n$ but not on $v$.

\item If a semisimple conjugacy class $[\sigma]\subseteq G(F)$ is mapped to $[_E\Phi]$, then the image of $[_E\Phi]$ in $\RRR^n_{F_v}/\!\sim$ under completion at $v$ reflects the splitting behavior of the characteristic polynomial of $\sigma$ over $F_v$. Moreover,
\[
 \xymatrix{
\{\text{ss.\ conj.-cl.\ in }G(F)\} \ar[d]\ar[r] 	& \RRR^n_{F}/\!\sim \ar[d]\\
\{\text{ss.\ conj.-cl.\ in }G(F_v)\} \ar[r] 		& \RRR^n_{F_v}/\!\sim
}
\]
commutes.
\end{enumerate}
\end{rem}

\subsection{Unipotent conjugacy classes}
Let $F$ be either a local or global field of characteristic $0$, and suppose that $E$ is a finite field extension of $F$. 
Let $X$ be a reductive group defined over $E$, and put $Y=\Res_{E/F}X$. 
Let $\UUU_{X}$ and $\UUU_{Y}$ denote the varieties of unipotent elements in $X$ and $Y$ which are defined over $E$ and $F$, respectively. Note that $\UUU_{Y}(F)=(\Res_{E/F}\UUU_{X})(F)=\UUU_X(E)$. 
Further, there is a canonical bijection bijection 
\[
 \{X(E)\text{-conj.\ classes in }\UUU_{X}(E)\}
\leftrightarrow
 \{Y(F)\text{-conj.\ classes in }\UUU_{Y}(F)\} 
\]
mapping a conjugacy class $\bfU'(E)$ to $\bfU(F)=(\Res_{E/F}\bfU')(F)$

In $\GL_n$ the unipotent conjugacy classes are parametrised by partitions of $n$ (see~\cite{CoMc93}) so that returning to the notation of Section~\ref{sec:based:root:data},  the unipotent conjugacy classes in $H'$ over $E$ are parametrised by tuples of partitions $\tau=(\tau_i)_{i\in I}$ with $\tau_i$ being a partition of  $m_i$ so that the unipotent conjugacy classes in $H(F)$ are parametrised by the same tuples. 

Similarly, if $F$ is global, the unipotent conjugacy classes in $H'_v(E_v)$ and $H_v(F_v)$ are parametrised by tuples $\tau=(\tau_{i,w})_{i,w}$ with $\tau_{i,w}$ a partition of $m_i$.
 If $\tau$ is one such these partitions we denote by $\bfU_{\tau}'(E)$ (resp.\ $\bfU_{\tau}(F)$) the corresponding conjugacy classes in $H'(E)$ (resp.\ $H(F)$) in the global case, and in $H'_v(E)$ (resp.\ $H_v(F)$) in the local case.

\begin{rem}
 In the local case we are in fact only interested in classes attached to tuple $\tau=(\tau_{i,w})$ with $\tau_{i,w_1}=\tau_{i,w_2}$ for all $w_1, w_2\in W_v^i$. These are exactly the classes occurring when localising the global classes. 
\end{rem}

\subsection{Measures}\label{sec:measures:based:root}
Suppose now that $F$ is a local field, and we are given an equivalence class in $\RRR_F^n/\sim$. Let $_E\Psi$ be the representative for this class constructed above.
Recall our choice of measures on $H'(E)$, $H(F)$ and their subgroups from Section~\ref{subsec:meas:twisted} and their relation to the canonical measures $\mu_{H'(E)}$, $\mu_{H(F)}$ from Section~\ref{subsec:comparison:meas} (cf.~\cite{Gr97}).
The maximal compact subgroup $\cpt^{H'}_E=H'(\OOO_E)$ (resp.\ $\cpt^H_F=H(\OOO_F)$) is the stabiliser of a hyperspecial (resp.\ special) vertex in the principal apartment $\AAA(S', E)$ (resp.\ $\AAA(S,F)$) in the Bruhat-Tits building of $H'(E)$ (resp.\ $H(F)$).
If $E/F$ is unramified or tamely ramified, $\cpt^H_F$ is also hyperspecial. 
In any case, $\cpt^{H'}_E$ (resp.\ $\cpt^H_F$) is admissible with respect to $S'(E)$ (resp.\ $S(F)$) in the sense of~\cite[\S 1]{Ar81} so that in particular the Iwasawa decompositions $H'(E)=B'(E)\cpt_E^{H'}$ and $H(F)=B(F) \cpt_F^H$ hold.
By construction of the canonical measures in~\cite{Gr97} we have $\mu_{H'(E)}(\cpt^{H'}_E)=1=\mu_{H(F)} (\cpt^H_F)$.

Let $\tau=(\tau_i)_{i\in I}$ be a partition of $(m_i)_{i\in I}$. Since in $\GL_{m_i}$ all unipotent classes are of Richardson type (see~\cite[Chapter 7]{CoMc93}), there exists a standard parabolic subgroup $Q'=L'V'\in\FFF^{H'}(S')$ (already defined over $\Q$) such that $\bfU_{\tau}'(E)$ is induced from the trivial class in $L'(E)$. Then $Q'$ only depends on $\Psi$ and $\tau$, but not on $F$ or $E$. 
By~\cite{LaMu09} there exists a constant $c>0$ such that the invariant orbital integral over the class $\bfU_{\tau}'(E)$ equals
\begin{equation}\label{eq:inv:meas:unip}
 c \int_{\cpt_E^{H'}}\int_{V'(E)} f(k^{-1} uk) \,du\,dk
\end{equation}
for every $f\in C_c^{\infty}(H'(E))$, where $dv$ and $dk$ are our usual measures on $V'(E)$ and $\cpt_E^{H'}$, respectively. 
The computation on~\cite[p. 255]{Ar88a}, shows that we need to take $c=1$ for the invariant measure on $\bfU_{\tau}'(E)$ to be compatible with the constructions in~\cite{Ar88a}. 
The invariant measure on $\bfU_{\tau}(F)$ is then normalised by pullback as before. 
More precisely, the invariant orbital integral over the orbit $\bfU_{\tau}(F)$ is given by
\[
 \int_{\cpt_F^H}\int_{V(F)} f(k^{-1} uk) \,du\,dk
\]
for every $f\in C_c^{\infty}(G(F))$, where $Q=\Res_{E/F}Q'$ and $V=\Res_{E/F}V'$.

\section{Weighted orbital integrals at the non-archimedean places}\label{sec:weighted:orb:int:non-arch}
The goal of this section is to solve problem~\ref{non-arch} from the introduction. The final result is Corollary~\ref{cor:bound:weighted:orb:int:after:const:term:non.arch}.
\subsection{Integrals over logarithmic polynomials}\label{sec:int:log:pol}
We need effective versions of certain convergence results from~\cite[pp. 257-261]{Ar88a} which we shall prove in this section.
Let $F$ be a finite extension of $\F_v$ for some non-archimedean place $v$, and fix integers $d, l\geq1$. Let $\PPP(F^d, F^l)$ be the set of polynomials $p:F^d\longrightarrow F^l$ with coefficients in $F$. 
Let $\alpha=(\alpha_1, \ldots, \alpha_d)$ be a multiindex of integers $\alpha_1, \ldots, \alpha_d\geq0$, and if $x=(x_1, \ldots, x_d)\in F^d$, write $x^{\alpha}=\prod_{i=1}^d x_i^{\alpha_i}$.
Let $\PPP_{\alpha}(F^d, F^l)$ be the set of all $p(x)\in\PPP(F^d, F^l)$ for which every coefficient of $x^{\alpha'}$ vanishes if $\alpha_i'>\alpha_i$ for some $i\in\{1,\ldots, d\}$. Let $\|p\|_F$ denote the supremum of the $F$-norms of the coefficients of $p$. For $\delta>0$ put
\[
 \PPP_{\alpha}^{\delta}(F^d, F^l)=\{p\in\PPP_{\alpha}(F^d, F^l)\mid\|p\|_F>\delta\}.
\]
Here we denote by $\|p(x)\|_F$ the usual vector norm on $F^l$ of the element $p(x)\in F^l$.
If $\rho\in\R\cup\{\pm\infty\}$, we further define
\[
 B_F(\rho)=\{x\in F\mid |x|_F\leq q_F^{\rho}\} \subseteq F
\]
so that $B_F(\infty)=F$, $B_F(0)=\OOO_F$, and $B_F(-\infty)=\{0\}$.
\begin{rem}
 Setting $q_F:=e$ for an archimedean field $F$, the analogue of the results of this section and Section~\ref{subsec:meas:unip:orb} below stay valid in the archimedean situation.
\end{rem}

The following is a more effective variant of~\cite[Lemma 7.1]{Ar88a} and we will closely follow Arthur's proof.

\begin{lemma}
 Let $\rho\in\R_{\geq0}$ and set 
\[\Gamma_d(\rho,F)=\{x=(x_1, \ldots, x_d)\in F^d\mid \forall i:~ x_i\in B_F(\rho)\}=B_F(\rho)^d\subseteq F^d.\]
Fix $\alpha, d,l$ as before. Then there exist $C, t>0$ depending only on $\alpha$, $d$, and $l$, but not on $F$, $v$, or $\rho$ such that
\begin{equation}\label{volume:polynomial:ball}
 \vol\big(\{x\in \Gamma_d(\rho,F)\mid \|p(x)\|_F<\eps\}\big)
\leq C \vol\big(\Gamma_d(\rho, F)\big)q_F^{t\rho} (\delta^{-1}\eps)^t
\end{equation}
for all $\delta>0$, $\eps\in(0,1]$, and all $p\in\PPP_{\alpha}^{\delta}(F^d, F^l)$.

Moreover, if $\rho\in(0,1]$, then 
\[
  \vol\big(\{x\in \Gamma_d(\rho,F)\mid \|p(x)\|_F<\eps\}\big)
=  \vol\big(\{x\in \Gamma_d(0,F)\mid \|p(x)\|_F<\eps\}\big)
\leq C \vol\big(\OOO_F\big)^d (\delta^{-1}\eps)^t.
\]
\end{lemma}

\begin{proof}
The second assertion immediately follows from the first one so that it suffices to prove the first assertion.
 It also suffices to prove this assertion for $l=1$. Further note that if $p\in \PPP_{\alpha}^{\delta}(F^d, F)$, then $\delta^{-1}p\in\PPP_{\alpha}^1(F^d, F)$ so that by replacing $\eps$ with $\delta^{-1} \eps =:\eta $ we may assume that $\delta=1$. 
We proceed by induction on $d$ and $\alpha$. First assume that $d=1$ so that $\alpha=\alpha_1$ is a non-negative integer. If $\alpha=0$, $p\in \PPP_{\alpha}^1(F,F)$ is just a constant so that the left hand side of~\eqref{volume:polynomial:ball} is at most $\vol(\Gamma_d(\rho,F))$.

Now assume that $\alpha>0$ and that the assertion holds for all $\alpha'<\alpha$. Write $p(x)=(x-\xi)q(x)$ with $\xi\in E$ some root of $p$ in a suitable extension $E/F$ of degree at most $\alpha$, and $q:F\longrightarrow E$ a polynomial of degree $\leq \alpha-1$. Identifying $E$ with $F^{[E:F]}$ via some fixed basis, we can view $q$ as an element in $\PPP_{\alpha-1}(F, F^{[E:F]})$. We extend the norm $|\cdot|_F$ to $|\cdot|_E:E\longrightarrow\R$ so that $|x|_E=|x|_F$ for all $x\in F$.
Now by assumption, $1<\|p\|_F\leq 2\max\{1, |\xi|_E\} \|q\|_E$ so that $\|q\|_E> \frac{1}{2}\min\{1, |\xi|_E^{-1}\}$. Suppose first that $|\xi|_E\geq q_F^{\rho}+1$. Then $|q(x)|_E<\frac{\eta}{|\xi|_E-q_F^{\rho}} $ is implied by $|p(x)|_F<\eta$ so that
\[
 \{x\in \Gamma_1(\rho, F)\mid |p(x)|_F<\eta\}
\subseteq \Big\{x\in \Gamma_1(\rho, F)\mid |q(x)|_E<\frac{\eta}{|\xi|_E-q_F^{\rho}}\Big\}.
\]
By the inductive hypothesis there exist constants $C_1, t_1>0$ depending only on $\alpha$ such that the volume of the right-hand set is bounded by
\[
C_1 \vol\big(\Gamma_1(\rho, F)\big)\Big(\|q\|_E^{-1} \frac{\eta}{|\xi|_E-q_F^{\rho}} \Big)^{t_1} 
\leq C_1 2^{2t_1} \vol\big(\Gamma_1(\rho, F)\big) q_F^{2t_1\rho} \eta^{t_1} 
\]
which is the desired bound in the first case.
Now suppose that $|\xi|_E< q_F^{\rho}+1$. Then $x\in \Gamma_1(\rho, F)$ with $|p(x)|_F<\eta$ implies that $|q(x)|_E<\eta^{1/2}$ or $|x-\xi|_E<\eta^{1/2}$. Since now $\|q\|_E>\frac{1}{2}\frac{1}{1+q_F^{\rho}}$, the bound follows in both cases from the induction hypothesis. 
This finishes the proof of the lemma for $d=1$.

Now suppose that $d>1$ and that the assertion holds for every $d'<d$. First note that $\Gamma_d(\rho, F)=\Gamma_1(\rho, F)^d=B_F(\rho)^d$.
For $x=(x_1, \ldots, x_d)$ and $\alpha=(\alpha_1, \ldots, \alpha_d)$ write $\tilde{x}=(x_2, \ldots, x_d)$ and $\tilde{\alpha}=(\alpha_2, \ldots, \alpha_d)$. Then $p\in\PPP_{\alpha}^1(F^d, F)$ can be written as
\[
 p(x)=\sum_{i=1}^{\alpha_1} \tilde{p}_i(\tilde{x}) x_1^i
\]
for suitable polynomials $\tilde{p}_1, \ldots, \tilde{p}_{\alpha_1}\in\PPP_{\tilde{\alpha}}(F^{d-1}, F)$. Since $\|p\|_F>1$, there exists $i$ with $\|\tilde{p}_i\|_F>1$.  If $x\in\Gamma_d(\rho, F)$ with $|p(x)|_F<\eta$ and $|\tilde{p}_i(\tilde{x})|_F\geq\eta^{1/2}$, we view $p(x)$ as a polynomial in $x_1$ which is contained in $\PPP_{\alpha_1}^{\eta^{1/2}}(F,F)$. The case $d=1$ gives then an upper bound for the set of such $x$ of the desired form. 

On the other hand, if $x\in \Gamma_d(\rho, F)$ with $|p(x)|_F<\eta$ and $|\tilde{p}_i(\tilde{x})|_F< \eta^{1/2}$, we can apply the induction assumption for $d=1$, again obtaining an upper bound for the volume of such $x$ of the desired form. Adding both cases finishes the proof of the assertion for arbitrary $d$.
\end{proof}

\begin{proposition}\label{prop:bound:int:logpol}
 Let $\rho\in\R_{\geq0}$, $k, d\in\N$, and $\alpha=(\alpha_1, \ldots, \alpha_d)$ be given. 
Then there exist constants $C, t>0$, both independent of $v$, $F$, and $\rho$, such that the following holds: For every $\delta>0$ and all non-zero polynomials $p_1, \ldots, p_k\in\PPP_{\alpha}^{\delta}(F^d, F)$  we have
\[
 \int_{\Gamma_d(\rho, F)} \lambda(x)\,dx
\leq C \vol\big(\Gamma_d(\rho, F)\big) \bigg[ q_F^{\rho t} \delta^{-t} + \log ^k q_F^{\rho} + \log^k A_F(p_1, \ldots, p_k) +1\bigg],
\]
where
\[
 \lambda(x)=\bigg|\prod_{i=1}^k \log|p_i(x)|_F \bigg|
\]
and
\[
 A_F(p_1, \ldots, p_k)= \max_{i,\alpha'} |a_{\alpha'}^{(i)}|_F
\]
with $a_{\alpha'}^{(i)}$ denoting the coefficient of $p_i$ at the monomial $x^{\alpha'}$.
Moreover, if $\rho<1$, $\rho$ can be replaced by $0$ on the right-hand side of the above inequality.
\end{proposition}

\begin{proof}
Let $p:=\bigoplus_{i=1}^{k} p_i\in\PPP_{\alpha}^{\delta}(F^d, F^k)$, and write
 $\Gamma_d(\rho, F)_{p<1}:=\{x\in\Gamma_d(\rho, F)\mid \|p(x)\|_F<1\}$ and $\Gamma_d(\rho, F)_{p\geq1}:=\{x\in\Gamma_d(\rho, F)\mid \|p(x)\|_F\geq1\}$.
Then for every $x\in \Gamma_d(\rho, F)_{p\geq1}$ we have
\begin{align*}
 0\leq \lambda(x)
&\leq \prod_{i=1}^k \log\max\{1, \sup_{x\in\Gamma_d(\rho, F)} |p_i(x)|_F\}\\
&\leq \prod_{i=1}^k\log\max\{1, |\alpha| q_F^{|\alpha| \rho} \sup_{\alpha'} |a_{\alpha'}^{(i)}|_F\}
\leq \bigg(\log|\alpha| + \log q_F^{|\alpha|\rho} +\log  A_F(p_1, \ldots, p_k)\bigg)^k ,
\end{align*}
where we write $|\alpha|=\alpha_1+\ldots+\alpha_d$.
Hence
\begin{multline}\label{integral:log:polynomial1}
 \int_{\Gamma_d(\rho, F)}\lambda(x)\,dx
\leq \vol\big(\Gamma_d(\rho, F)\big) \bigg(\log|\alpha| + \log q_F^{|\alpha|\rho} +\log  A_F(p_1, \ldots, p_k)\bigg)^k\\
+\int_{\Gamma_d(\rho, F)_{p<1}}\lambda(x)\,dx.
\end{multline}
To estimate this last integral, we proceed as in~\cite[pp. 259-260]{Ar88a}. 
If $x\in\Gamma_d(\rho, F)_{p<1}$, there exists $i\in\{1, \ldots, k\}$ such that $|p_i(x)|_F\leq |p_j(x)|_F\leq \|p(x)\|_F<1$ for all $j\neq i$. In particular,  $x\in \Gamma_d(\rho, F)_{p_i<1}$ and $\lambda_p(x)\leq \big|\log|p_i(x)|_F\big|^k$ so that
\[
 \int_{\Gamma_d(\rho, F)_{p<1}}\lambda(x)\,dx
\leq \sum_{i=1}^k \int_{\Gamma_d(\rho, F)_{p_i<1}}\big|\log |p_i(x)|_F\big|^k \,dx.
\]
Hence it suffices to consider $l=1$.
Dividing the domain of integration into shells as in~\cite[p. 260]{Ar88a}, it follows from~\cite{Ar88a} and the last lemma that this can be bounded by
\begin{equation}\label{integral:log:polynomial2}
 C \vol\big(\Gamma_d(\rho, F)\big) q_F^{t\rho} \delta^{-t} \sum_{m\geq0} 2^{-mt} \big((m+1)\log 2\big)^k
\leq \tilde{C} \vol\big(\Gamma_d(\rho, F)\big) q_F^{t\rho} \delta^{-t} 
\end{equation}
for suitable constants $C, \tilde{C}, t>0$ depending only on $k, d$, and $\alpha$.
Combining~\eqref{integral:log:polynomial1} and~\eqref{integral:log:polynomial2} the proposition follows.
The last assertion then follows from the previous considerations and the second part of the last lemma.
\end{proof}

\subsection{Weighted orbital integrals on unipotent orbits}\label{subsec:meas:unip:orb}
We return to the notation of Section~\ref{sec:measures:based:root}. 
Recall that $V'$ is the unipotent radical of a standard parabolic subgroup.   
Hence we can identify $V'$ with its Lie algebra $\vvv'$ via the map $V'\longrightarrow \vvv', ~u\mapsto u-\One$. There is a canonical isomorphism of schemes (over $\Z$) $\G_a^{d'}\longrightarrow \vvv'$ given via the coordinate entries of the matrices $\vvv'$ where $d'$ is the dimension of $\vvv'$. In particular, we also get an isomorphism $\G_a^{d'}\longrightarrow V'$.
Via
\[
  \xymatrixcolsep{8pc}\xymatrixrowsep{4pc}\xymatrix{
 E^{d'}		\ar@<0.5ex>[r]^-{\Res_{E/F}}				\ar[d]                         & F^{d'[E:F]} \ar@{-->}[d]\\
V'(E) 	\ar@<0.5ex>[r]^-{\Res_{E/F}}									& V(F)
}
\]
we get isomorphisms $F^{d'[E:F]}\longrightarrow V(F)$ and $E^{d'}\longrightarrow V(F)$. We denote the latter by $\phi$.

Let $M'\in \LLL^{H'}(S')$, and $M=\Res_{E/F}M'$. Suppose that $\gamma_u\in\bfU(F)\cap M(F)$. Then for every $f\in C_c^{\infty}(H(F))$ the weighted orbital integral $J_{M}^{H}(\gamma_u,f)$ is given by
\[
J_{M}^{H}(\gamma_u, f)
= \int_{\cpt_F^H}\int_{V(F)} f(k^{-1} u k) w_{M, V}^{H(F)}(u) \,du\,dk
\]
for a suitable weight function $w_{M,V}^{H(F)}:V(F)\longrightarrow\C$ which is $\cpt_F^H$-conjugation invariant (cf.~\cite{Ar88a}). 
Here the measures are as in Section~\ref{sec:measures:based:root}. 
In particular, the volume of $\cpt_F^H$ is $1$ (with respect to $dk$), and the volume of $V(\OOO_F)$ is $\N(\DDD_{E/F})^{-d'/2}$.
The weight function essentially depends only on the root datum $\Psi$ and the partition $\tau$ defining the unipotent class, but in a certain sense not on $F$ or $E$. 
It follows from its construction in~\cite{Ar88a} that it can be written in the following form:

\begin{lemma}\label{lemma:uniform:form:weightfct}
Let $\Psi$, $\tau$, and $M$ be as before, and let $V'$ be associated with $\tau$ and $\Psi$ as before.
 There exist polynomials $p_1, \ldots, p_l:\G_a^{d'}\longrightarrow \G_a^k$ for some integer $k\geq0$ with $\Q$-rational coefficients and a polynomial $q:\C^l\longrightarrow\C$ such that
\[
 w_{M,V}^{H(F)} \big(\phi (X_1, \ldots, X_{d'})\big)
=q\big(\log \|p_1(X_1, \ldots, X_{d'})\|_F,\ldots, \log \|p_l(X_1, \ldots, X_{d'})\|_F\big)
\]
for all $F$ and $E$ such that $_E\Psi\in \RRR^n_{F}$, and all $X_1, \ldots, X_{d'}\in E$. Here we view the $p_i(X_1,\ldots,X_{d'})\in E$ as an element in $F^{k[E:F]}$. \qed
\end{lemma}
 
Recall the definition of the sets $\Gamma_{n^2}(\rho, F)=\Gamma_1(\rho, F)^{n^2}=B_{\rho}(F)^{n^2}\subseteq \Mat_{n\times n}( F)$ for $\rho\geq0$ from the last section.

\begin{lemma}\label{estimate_for_int_over_weight}
There exist constants $c_1, c_2\geq0$ depending only on $n$ and $\F$ such that the following holds:
For every non-archimedean place $v$ of $\F$, and every finite reduced $\F_v$-algebra $E$ such that $_E\Psi\in \RRR^n_{\F_v}$, and all $\tau$ and $M$ as before we have
\[
 \int_{V(\F_v)\cap \Gamma_{n^2}(\rho, \F_v)} \left|w_{M,V}^{H(\F_v)}(x)\right|\,dx
\leq c_1 q_v^{n^2\rho c_2}
\]
for all $\rho\geq0$. Here $V'$ is attached to $\tau$ and $_E\Psi$ as before, and $V=\Res_{E/\F_v}V'$.
\end{lemma}

\begin{proof}
Write $F=\F_v$.
Let $x\in V(F)\cap \Gamma_{n^2}(\rho, F)$  so that every matrix entry of $x$ has $F$-valuation $\geq-\rho$. Identifying $x$ with its preimage $x'\in V'(E)$ under $\Res_{E/F}$,  every matrix entry of $x'$ has $E$-valuation $\geq-[E:F]\rho$.  It therefore suffices to estimate
\[
 \int_{V'(E)\cap \Gamma_{H'}([E:F]\rho, E)} \left|w_{M,V_1}^{H(F)}(\phi(x'))\right|\,dx'
\]
where $\Gamma_{H'}([E:F]\rho, E)\subseteq H'(E)$ denotes the set of all elements of $H'(E)$ whose matrix entries all have $E$-valuation $\geq-[E:F]\rho$. 
It follows from Lemma~\ref{lemma:uniform:form:weightfct} and Proposition~\ref{prop:bound:int:logpol} that
\[
 \int_{V'(E)\cap \Gamma_{H'}([E:F]\rho, E)} \left|w_{M,V}^{H(F)}(\phi(x'))\right|\,dx'
\leq c_1' \vol(B_{[E:F]\rho}(E))^{\dim_E H'(E)} q_E^{[E:F]\rho c_2'}
\]
for constants $c_1', c_2'>0$ depending only on $n$ and $\F$. 
For this last claim (i.e.\ that the constants can be chosen to depend only on $n$ and $\F$) note that the $v$-adic norms of the coefficients of the polynomials $p_1, \ldots, p_l$ are all $1$ for almost all $v$, i.e., $p_1,\ldots,p_l\in \PPP_\alpha^{1/2}(F^{d'},F^k)$ for almost all $v$.

Now $q_E=q_F^{\fff(E/F)}$ and 
\[
\vol(B_{[E:F]\rho}(E))
\leq\vol(\OOO_E) q_E^{[E:F]\rho} 
= \N(\DDD_{E})^{-1/2} q_E^{[E:F] \rho}
\leq q_F^{[E:F]\rho}.
\]
The degree $[E:F]$ is bounded by $n$ so that the assertion follows.
\end{proof}

\begin{rem}
 The above estimate also holds if we replace  the weight function by $1$ in which case we obtain
\[
 \vol\big(V(\F_v)\cap \Gamma_{n^2}(\rho, \F_v)\big)
\leq  c_1 q_v^{n^2\rho c_2}.
\]
\end{rem}

\begin{rem}
 If $R$ is a parabolic subgroup in $H$ containing $S$, then its Levi component $M_R$ equals $M_R=\Res_{E/F} M_R'$ for some semi-standard Levi subgroup $M_R'$ in $H'$, so in particular it is isomorphic to a direct product of $\GL_{l_j}$, $j=1, \ldots, s$,  with $l_1+\ldots+l_s=n$. 
Hence the above discussion also defines the weighted orbital integrals $J_M^{M_R}(\gamma_u,f)$, and the estimates above stay valid with the obvious modifications.
\end{rem}

\begin{rem}\label{rem:arch:case}
 As already remarked in the last section, the analogue results stay true for $F=\F_v$ an archimedean field. In that case, the estimate of Lemma~\ref{estimate_for_int_over_weight} becomes
 \[
   \int_{V(\F_v)\cap \Gamma_{n^2}(\rho, \F_v)} |w_{M,V}^{H(\F_v)}(x)|\,dx
\leq c_1 e^{c_2\rho}
 \]
for $c_1, c_2>0$ constants depending only on $n$.
\end{rem}

\subsection{Weight functions on $G$}\label{subsec:weight:fcts:gln}
So far we have studied weight functions on unipotent classes in certain subgroups of $G(F)$. 
We need another type of weight functions to fully define the local weighted orbital integrals. 
Let $F$ be an arbitrary local field (we also allow archimedean fields in this section), and let $M\subseteq G$ be a semi-standard Levi subgroup of $G$.
For $Q\in\FFF(M)$ define the weight function
$v_Q':G(F)\longrightarrow \C$
 by (cf.~\cite[p. 200]{Ar86},~\cite[\S 2]{Ar81})
\[
 v_Q'(x)=\int_{\aaa_Q^G} \Gamma_Q^G(X, -H_Q(x))dX
\]
where 
\[
 \Gamma_Q^G: \aaa_0\times \aaa_Q\longrightarrow\C
\]
is given by
\[
 \Gamma_Q^G(X, Y)
=\sum_{\substack{Q_1\in\FFF: \\ Q\subseteq Q_1}} (-1)^{\dim \aaa_{Q_1}^G} \tau_{Q}^{Q_1}(X)\hat{\tau}_{Q_1}^G(X-Y_{Q_1})
\]
with $Y_{Q_1}\in \aaa_{Q_1}$ denoting the projection of $Y$ onto $\aaa_{Q_1}$. Here $\tau_{Q}^{Q_1}$ and $\hat{\tau}_{Q_1}$ are certain characteristic functions defined in~\cite[\S 4]{Ar86}.
Then as a function of $X\in\aaa_0^G=\aaa$ the support of $\Gamma_Q^G(\cdot, Y)$ is compact for $Y$ fixed (see~\cite{Ar81}) and volume of the support (in $\aaa$) is bounded by  a polynomial in $Y$ (with coefficients in $\Q$) of degree $\leq \dim \aaa$ which is independent of the local field $F$. Moreover, by definition, $|\Gamma_Q^G(X,Y)|$ is bounded by an absolute constant independent of $X$ and $Y$.
Note that $v_Q'$ is left $M_Q(F)$- and right $\cpt_F$-invariant.

The definition immediately implies that
\begin{equation}\label{polynomial_bound_wgt_fct}
 \left|v_Q'(x)\right|
\leq c(1+\|H_Q(x)\|_W)^{n-1},
\end{equation}
where $c>0$ is some constant independent of $F$.

\begin{cor}\label{cor:weight:bound:by:norm}
There exists a constant $c>0$ depending only on $n$ but not on $F$ such that for every $Q\in\FFF(M)$ and $x\in G(F)$ if $F$ is non-archimedean, and for every $x\in G(F)^1$ if $F$ is archimedean, we have 
\[
 \left|v_Q'(x)\right|
 \leq 
 \begin{cases}
 c (1+ \log\|x\|_{G(F)})^{n-1}			&\text{if }F\text{ is non-archimedean},\\
 c (1+ \log\|x\|_{G(F)^1})^{n-1}			&\text{if }F\text{ is archimedean}.
 \end{cases}
\] 
\end{cor}

\begin{rem}
 The estimate could be improved by replacing $\|x\|_{G(F)}$ on the right hand side by $\|x\|_{M_Q\backslash G}:=\inf_{m\in M_Q(F)}\|mx\|_{G(F)}$.
\end{rem}

\begin{proof}
 If $v$ is non-archimedean, $\|H_Q(x)\|_W\leq \log\|x\|_{G(F)}$ by~\cite[pp. 486-487]{Ko05} which is enough to show by the preceding lemma. 
If $v$ is archimedean, we can similarly prove that $\|H_Q(x)\|_W\leq \log\|x\|_{G(F)^1}$: Let $P_0\subseteq Q$ be the minimal parabolic subgroup contained in $Q$ with Levi component $T_0$. As explained in~\cite[\S 12.1]{Ko05}, $H_Q(x)$ equals the image of $H_{P_0}(x)$ under the orthogonal projection from $\aaa$ onto $\aaa_{Q}^G$ so that $\|H_Q(x)\|_W\leq \|H_{P_0}(x)\|_W$ (again, in~\cite[\S 12.1]{Ko05} only the non-archimedean case is considered, but the arguments stay obviously true in general). 
Suppose $\xi\in \aaa_0$ is such that $x=k_1 e^{\xi} k_2\in\cpt_{F} e^{\xi} \cpt_F$.  Then $H_{P_0}(x)=H_{P_0}(k_1 e^{\xi})$. By Kostant's convexity theorem~\cite{Ko73}, $H_{P_0}(k_1 e^{\xi})$ lies inside the convex hull of the Weyl group orbit of the point $H_{P_0}(e^{\xi})\in\aaa_0$. Since $\|H_{P_0}(e^{\xi})\|_W=\|\xi\|_W=\log\|x\|_{G(F)^1}$ the assertion now follows in the archimedean case.
\end{proof}

\subsection{Weighted orbital integrals on general classes}\label{subsec:weighted:orb:int:non-arch}
Fix a non-archimedean place $v$ of $\F$ and let $F=\F_v$. 
Suppose $\gamma=\sigma\nu\in G(F)$ as usual with $\sigma$  semisimple and $\nu$ unipotent and commuting with $\sigma$.
We choose $M_1$ as in~\cite[\S 4]{Ar86}, that is, $M_1(F)$ is the minimal standard Levi subgroup of $G(F)$ which contains $\sigma$. Replacing $\gamma$ (and therefore $\sigma$) by a $G(F)$-conjugate if necessary, we can assume that $\sigma$ is not contained in any parabolic subgroup of $G(F)$ (standard or not) which is properly contained in $M_1(F)$. Let $P_1=M_1U_1\in\FFF(T_0)$ be the corresponding standard parabolic subgroup. 
Then $S=M_{1,\sigma}$ is a maximal torus of $G_{\sigma}(F)=H(F)$ and $B(F)=P_{1,\sigma}(F)$ a minimal parabolic of $H(F)$. This is all consistent with our construction in Section~\ref{subsec:equiv:classes:root:data}. In particular, $\sigma$ is $F$-elliptic in $M_H(F)$ and regular $F$-elliptic in $M_1(F)$.

\begin{example}
 \begin{enumerate}[label=(\roman{*})]
  \item If $\sigma=\diag(\sigma_1, \ldots, \sigma_n)\in G(F)$ is a regular diagonal matrix (i.e.\ $\sigma_i\neq\sigma_j$ for $i\neq j$), then $M_1(F)=S(F)=H(F)=T_0(F)$.
  
  \item If $\sigma\in G(F)$ is $F$-elliptic (i.e.\ if the characteristic polynomial of $\sigma$ equals a power of the minimal polynomial of $\sigma$), then we choose $\sigma$ such that $M_1$ equals $\GL_d\times\ldots\times\GL_d\subseteq \GL_n$ ($n/d$-many factors, diagonally embedded in $\GL_n$) with $d$ the degree of the minimal polynomial of $\sigma$. The group $S(F)$ is then isomorphic to $\GL_1(E)\times\ldots\times\GL_1(E)\subseteq \GL_{n/d}(E)\simeq H(F)$ ($n/d$-many factors, diagonally embedded) with $E$ a suitable $d$-dimensional extension over $F$.
 \end{enumerate}
\end{example}

Note that if the characteristic polynomial of $\sigma$ has coefficients in $\OOO_{\F}\subseteq F$, then
\begin{equation}\label{eq:bound:for:field:disc}
 \N_{E_i/F}(\DDD_{E_i})\ll |D^G(\sigma)|_F^{-c}
\end{equation}
for some $c>0$ depending only on $n$ and $\F$.

Let $M\in \LLL(M_1)$ be a Levi subgroup containing $M_1$.
By~\cite[Corollary 8.7]{Ar88a} the distribution $J_M(\gamma, f)$ equals for $f\in C_c^{\infty}(G(F))$ the integral
\[
 |D^G(\sigma)|_F^{1/2}\int_{G_{\sigma}(F)\backslash G(F)}\bigg(\sum_{R_1\in \FFF^{G_{\sigma}}(M_{\sigma})} J_{M_{\sigma}}^{M_{R_1}}(\nu, \Phi_{R_1, y})\bigg) \,dy,
\]
where the compactly supported smooth function $\Phi_{R_1, y}: M_{R_1}(F)\longrightarrow \C$ is given by
\[
 \Phi_{R_1,y}(m)=\delta_{R_1}(m)^{1/2}\int_{\cpt_{\sigma, F}}\int_{N_{R_1}(F)} f(y^{-1}\sigma k^{-1} m n k y) v_{R_1}'(ky) \,dn\,dk,
\]
and
\[
 v_{R_1}'(z)=\sum_{\substack{Q\in\FFF(M): ~ Q_{\sigma}=R_1, \\ \aaa_Q=\aaa_{R_1}}}v_Q'(z)
\]
with $v_Q'$ as in Section~\ref{subsec:weight:fcts:gln}.
Here $\FFF^{G_{\sigma}}(M_{\sigma})$ denotes the set of Levi subgroups in $G_{\sigma}$ containing $M_{\sigma}$, and the quotient measure on $G_{\sigma}(F)\backslash G(F)$ is chosen in accordance with Section~\ref{subsec:meas:twisted}. This description is also valid over any archimedean field.
As explained before, for every $R_1$ there exists a standard parabolic $Q_1=L_1V_1\in\FFF^{M_{R_1}}(M_{1,\sigma})$ such that the unipotent class $\VVV_{R_1}$ generated by $\nu$ in $M_{R_1}$ is the Richardson class (in $M_{R_1}$) corresponding to $V_1(F)$. Note that the induced class of $\VVV_{R_1}$ to $G_{\sigma}$ (along $N_{R_1}$) then corresponds to $N_{R_1}(F)V_1(F)$ which is the unipotent radical of the parabolic subgroup $R_1V_1:=Q=LV\in \FFF^{G_{\sigma}}(M_{1,\sigma})$.
The unipotent orbital integral can therefore be written as 
\[
 J_{M_{\sigma}}^{M_{R_1}}(\nu, \Phi_{R_1, y})
=\int_{\cpt_{F}^{G_{\sigma}}}\int_{V(F)} f(y^{-1} \sigma k^{-1} vk y) w_{M_{\sigma},V_1}^{M_{R_1}(F)}(v) v_{R_1}'(ky) \,dv\,dk.
\]
where we continue $w_{M_{\sigma},V_1}^{M_{R_1}(F)}$ trivially to all of $V$.

\begin{lemma}\label{lemma:first_est_orb_non-arch}
 There exist constants $C, c_1, c_2, c_3, c_4, c_5, c_6\geq0$ depending only on $n$ and $\F$, but not on $v$ or $\sigma$ such that the following holds:
 Suppose that the $G(F)$-orbit (under conjugacy) of $\sigma \UUU_{G_{\sigma}}(F)$ intersects $\ooo$ non-trivially for some $\ooo\in \OOO_{R, \bkappa}$. Then for every $\gamma=\sigma \nu\in \sigma\UUU_{G_{\sigma}}(F)\cap M(F)$ and every $\xi\in X_0^+(T_0)$ with $0\le\xi_1\ldots\le\xi_n\le\kappa_v$ 
 we have
\[
 \bigg|J_M^G(\gamma, \tau_{\xi})\bigg|
\leq C q_F^{c_1+c_2\kappa_v} |D^G(\sigma)|_F^{-c_3} \sum_{\xi'\in X_*^+(T_0): \|\xi'\|_W\leq \kappa'} \int_{G_{\sigma}(F)\backslash G(F)} \tau_{\xi'}(y^{-1}\sigma y) \,dy,
\]
where
\[
 \kappa'=c_4+c_5\kappa_v-c_6\log_{q_F}|D^G(\sigma)|_F,
\]
and we can take $c_1=0=c_4$ if the residue characteristic of $F$ is $>n!$.
\end{lemma}
\begin{proof}
Let $y\in G(F)$, $k\in\cpt_{F}^{G_{\sigma}}$, and $u\in V(F)$, and suppose that $\tau_{\xi}(y^{-1} \sigma k^{-1} u k y)\neq0$. By Corollary~\ref{prop_distance_to_torus} there exist constants $a, b, \delta_v\geq0$ and  $g\in G_{\sigma}(F)$ (independent of $u$) such that
\[
 \|g y\|_{G(F)}, ~\|g ug^{-1}\|_{G(F)}\leq |D^G(\sigma)|_F^{-a} q_F^{b\kappa_v+\delta_v} ,
\]
where $\delta_v=0$ if the residue characteristic of $F$ is $>n!$.
Hence there are $a_1, a_2, b_1, b_2,  d_1, d_2, C_1, C_2\geq0$, all depending only on $n$ and $\F$, such that with $\rho:= \delta+a_1\kappa_v-a_2\log_{q_F} |D^G(\sigma)|_F$ and $\kappa'=\delta+d_1\kappa_v-d_2\log_{q_F} |D^G(\sigma)|_F$ we get
\begin{align*}
  \bigg|J_M(\gamma, \tau_{\xi})\bigg|
& \leq C_1 (1+ \log\left(|D^G(\sigma)|_F^{-a} q_F^{b\kappa_v+\delta_v}\right))^{n-1}
 \sum_{\xi'\in X_*^+(T_0): \|\xi'\|\leq \kappa'} \int_{G_{\sigma}(F)\backslash G(F)}\tau_{\xi'}(y^{-1} \sigma y)  \,dy\\ 
& \cdot \int_{V(F)\cap \Gamma_{n^2}(\rho, F)} |w_{M_{\sigma},V_1}^{G_{\sigma}(F)}(v)|\,dv\\
& \leq C_2 q_F^{b_1\delta+b_2\kappa_v}|D^G(\sigma)|_F^{-b_3} \sum_{\xi'\in X_*^+(T_0): \|\xi'\|\leq \kappa'} \int_{G_{\sigma}(F)\backslash G(F)} \tau_{\xi'}(y^{-1} \sigma y)\,dy .
\end{align*}
Here we used Corollary~\ref{cor:weight:bound:by:norm} in the first inequality to bound $|v_{R_1}'(kgy)|$, and Lemma~\ref{estimate_for_int_over_weight} in the second step to bound the integral over $|w_{M_{\sigma},V_1}^{G_{\sigma}(F)}|$. 
This finishes the assertion.
\end{proof}

\begin{cor}\label{cor:final:bound:weighted:orb:int:non-arch}
Let $\gamma$ and $v$ be as in Lemma~\ref{lemma:first_est_orb_non-arch}, and $F:=\F_v$.
Then
\[
  \bigg|J_M^G(\gamma, \tau_{\xi})\bigg|
\leq q_F^{a+b\kappa_v}|D^G(\sigma)|_F^{-c}
\]
for every $\xi$ with $\|\xi\|_W\leq \kappa_v$, where $a, b,c\geq0$ are suitable constants depending only on $n$ and $\F$, but not on $\gamma$, $v$,  or $\kappa_v$.
\end{cor}

\begin{proof}
In view of Lemma~\ref{lemma:first_est_orb_non-arch} we only need to show that there exist constants $c_1, c_2, c_3>0$ depending only on $n$ and $\F$ such that
\begin{equation}\label{eq:final:bound:non-arch1}
\int_{G_{\sigma}(F)\backslash G(F)}\tau_{\xi'}(y^{-1} \sigma y)\,dy 
\leq q_F^{c_1+c_2\kappa_v}  |D^G(\sigma)|_F^{-c_3}
\end{equation}
for every $\xi'\in X_*^+(T_0)$ with $\|\xi'\|_W\leq\kappa_v'$ with $\kappa_v'$ as in Lemma~\ref{lemma:first_est_orb_non-arch}. (Note that there are at most $(\kappa_v')^n$ such $\xi'$'s, and $(\kappa_{v}')^n\ll_{\eps} q_v^{\eps\kappa_v}$ for all $\eps>0$ if $\kappa_v>0$.)

Let $d_{\text{can}}y$ denote the quotient measure on $G_{\sigma}(F)\backslash G(F)$ obtained from the canonical measures of~\cite{Gr97} (cf.\ Section~\ref{subsec:comparison:meas}) on $G_{\sigma}(F)$ and $G(F)$.
Suppose first that $v\not\in S_{\text{bad}}$. Then  $G_{\sigma}$ splits over a Galois extension $E/F$ which is at worst tamely ramified.
Hence by~\cite[Theorem 7.3]{ShTe12} there are constants $c_1, c_2, c_3>0$ depending only on $n$  such that 
\begin{equation}\label{eq:final:bound:non-arch2}
\int_{G_{\sigma}(F)\backslash G(F)}\tau_{\xi'}(y^{-1} \sigma y)\,d_{\text{can}}y 
\leq q_F^{c_1+c_2\kappa_v'} |D^G(\sigma)|_F^{-c_3}
\end{equation}
for every $\xi'$ with $\|\xi'\|_W\leq\kappa_v'$. 
The canonical measure on $G(F)$ differs from our measure by a power of the norm of the different of $F=\F_v$ with exponent depending only on $n$. 
On the other hand, the canonical measure on $G_{\sigma}(F)$ differs from our choice of measures by $\big(\prod_{i=1}^r \N(\DDD_{E_i})\big)^c$ for some $c$ depending only on $n$. Since the eigenvalues of $\sigma$ are all algebraic integers, the constant by which $d_{\text{can}}y$ differs from our usual quotient measure can be bounded by $|D^G(\sigma)|_F^{-a}$ for some $a>0$ depending only on $n$ and $\F$ (cf.~\eqref{eq:bound:for:field:disc}).
Hence~\eqref{eq:final:bound:non-arch1} follows from~\eqref{eq:final:bound:non-arch2} which finishes the asserted estimate for $v\not\in S_{\text{bad}}$. 

Now suppose that $v\in S_{\text{bad}}$ and write $F=\F_v$ again. Note that the set $S_{\text{bad}}$ only depends on $n$ and $\F$. 
Recall the set $\RRR_F^n/\sim$ of equivalence classes of based root data over $F$ of dimension $n$, and the special representatives $_E\Psi$ we fixed in~\ref{sec:based:root:data}. The set $\RRR_F^n/\sim$ is finite and we enumerate the representatives $_E\Psi$ as $_{E^l}\Psi^l$, $l=1, \ldots, s$. Recall the corresponding groups $H^{l\prime}$, $H^l$ related by $H^l=\Res_{E^l/ F} H^{l\prime}$ and the embedding of $H^l$ into $\GL_n$.
Then for every semisimple $\sigma$ the centraliser $G_{\sigma}(F)$ of $\sigma$ is in $G(F)$ conjugate to one of the $H^l(F)$'s.
Let $d_{H^l\backslash G} x$ denote the measure on $H^l(F)\backslash G(F)$ obtained from the canonical measures on $G(F)$ and $H^l(F)$ by taking quotients.
The intersection $\cpt_F\cap H^l(F)$ is an open-compact subgroup in $H^l(F)$, hence has non-zero measure as a subgroup of $H^l(F)$. We denote its measure by $C_l>0$.
For $m\geq0$ let 
\[
\Omega_m
=\{g\in G(F)\mid\|g\|_{G(F)}\leq e^m\},
\]
and
\[
 \Omega_m^l:=\bigcup_{y\in \Omega_m} H^l(F) y.
\]
Then
\[
 H^l(F)\backslash \Omega_m^l
= \{x\in H^l(F)\backslash G(F)\mid \inf_{m\in H^l(F)} \|mx\|_{G(F)}\leq e^m\}.
\]
Corollary~\ref{prop_distance_to_torus} together with an estimate on $|D^G(\sigma)|_v^{-1}$ as in the proof of the last lemma implies that to finish the proof of the corollary for the bad places, we need to estimate the volume of $H^l(F)\backslash \Omega_{\kappa_v'}^l$ as a subset of $H^l(F)\backslash G(F)$.

Let $\chi_m:G(F)\longrightarrow\R$ be the characteristic function of $\Omega_m$, and $\chi_m^l: H^l(F)\backslash G(F)\longrightarrow\R$ the characteristic function of $H^l(F)\backslash \Omega_m^l$.
Then
\begin{align*}
\vol_{G(F)}(\Omega_m)
& = \int_{G(F)} \chi_m(g)\,dg
=\int_{H^l(F)\backslash G(F)} \int_{H^l(F)} \chi_m(xg) \,d_{H^l}x \,d_{H^l\backslash G} g \\
& \geq \vol_{H^l(F)}(H^l(F)\cap \cpt_F) \int_{H^l(F)\backslash G(F)} \chi_m^l(g) \,d_{H^l\backslash G} g \\
& = C_l \vol_{H^l(F)\backslash G(F)} (H^l(F)\backslash \Omega_m^l),
\end{align*}
where $\vol_{G(F)}$ etc.\ indicates that we compute the measure of the set with respect to the measure on $G(F)$ etc.

The volume of $\Omega_m$ can be estimated by using Iwasawa decomposition and the fact that $\chi_m$ is a bi-$\cpt_F$-invariant function:
\[
 \vol_{G(F)}(\Omega_m)
\leq \int_{T_0(\OOO_v)\backslash T_0(F)} \int_{U_0(\OOO_v)\backslash U_0(F)} \chi_m(tu) |\delta_0(t)^{-1}|_v \,dt\,du.
\]
But $\chi_m(x)\neq0$, $x\in G(F)$, implies that $|x_{ij}|_F\leq q_F^m$ for all $i,j=1, \ldots, n$. Hence
\[
 \vol_{G(F)}(\Omega_m)
\leq q_F^{Km}
\]
for some $K\geq0$ depending only on $n$ but not on $m$. Hence
\[
 \vol_{H^l(F)\backslash G(F)} \left(\{x\in H^l(F)\backslash G(F)\mid \inf_{m\in H^l(F)} \|mx\|_{G(F)}\leq e^{\kappa_v'}\}\right)
\leq C_l^{-1} q_F^{K\kappa_v'}
\leq q_F^{c_1+c_2\kappa_v'}
\] 
for suitable constants $c_1, c_2\geq0$ depending only on $n$ but not on $m$. This gives the estimate~\eqref{eq:final:bound:non-arch1} for the bad places but without the effective dependence on $\sigma$.
Since the number of bad places and the number of elements in $\RRR_F^n/\sim$ can be bounded from above by some number depending only $n$ and $\F$, this is sufficient for our purposes and finishes the proof of the corollary.

\end{proof}

\begin{rem}\label{rem:estimate:weighted:orb:non-arch:general}
In order to bound the weighted orbital integral $\big|J_M(\gamma, \tau_{\xi})\big|$, we estimated the absolute value of integrand at every point of the integration domain. In particular, if $\tau: G(F)\longrightarrow\C$ is a function with $\supp\tau\subseteq\supp\tau_{\xi}=\omega_{F, \xi}$ and $|\tau|\leq 1$, then for every $\gamma\in M(F)$,
\[
 \bigg|J_M^G(\gamma, \tau)\Big|\leq q_F^{a+b\kappa_v} |D^G(\sigma)|_F^{-c}
\]
with $a,b,c>0$ as in Corollary~\ref{cor:final:bound:weighted:orb:int:non-arch}.
\end{rem}

\begin{cor}\label{cor:bound:weighted:orb:int:after:const:term:non.arch}
 Let $\gamma$ be as in Lemma~\ref{lemma:first_est_orb_non-arch}, and $v\in S$, $F:=\F_v$.
Then for every $L\in\LLL(M)$ and $Q\in\PPP(L)$ we have
\[
  \bigg|J_M^L(\gamma, \tau_{\xi}^{(Q)})\bigg|
\leq q_F^{a+b\kappa_v} |D^L(\sigma)|_F^{-c}
\]
for all $\xi\in X_0^+(T_0)$ with $0\le\xi_1\le\ldots\le\xi_n\leq \kappa_v$, where $a, b,c\geq0$ are suitable constants depending only on $n$ and $\F$, but not on $\gamma$, $v$,  $\kappa_v$, $L$ or $Q$.

Moreover, for every $M$ and every $Q\in\PPP(M)$ we have $\big|J_M^M(\gamma, \tau_{0}^{(Q)})\big|=\tau_0^{(Q)}(\gamma)$.
\end{cor}
Note that at almost all places the function $\tau_0^{(Q)}$ equals the unit element in the spherical Hecke algebra of $M(F)$ so that at those places $\tau_0^{(Q)}(\gamma)=1$ if $\gamma\in\cpt^M_F$, and it equals $=0$ if $\gamma\not\in\cpt_F^M$. 
\begin{proof}
 The first part is a consequence of Corollary~\ref{cor:final:bound:weighted:orb:int:non-arch} and Lemma~\ref{lemma:behaviour:constant:term:non-arch}. The last part is just the definition of the distribution $J_M^M(\gamma,\cdot)$.
\end{proof}

\section{Auxiliary estimates in the archimedean case}\label{sec:aux:est:arch}
To solve problem~\ref{arch} from the introduction, we need to find estimates for the archimedean case analogues to those in Section~\ref{sec:aux:est:non-arch}. 
Since the only archimedean place of $F$ is complex, we only need to consider the complex situation in which case every semisimple conjugacy class is split.
Recall that we fixed a real number $R>0$ once and for all.

\begin{lemma}\label{lemma:bound:norm:arch}
There exist constants $c_1, c_2\geq0$ depending only on $n$ and $R$ such that the following holds: Let $\sigma\in T_0(\C)$, $M_2(\C)=G_{\sigma}(\C)$, and $P_2=M_2U_2$ a parabolic subgroup with Levi component $M_2$. Then for any  $y\in G(\C)$ and $u\in \UUU_{G_{\sigma}}(\C)$ with $\|y^{-1} \sigma u y\|_{G(\C)^1}\leq e^R$ we have
\begin{align*}
 \|m^{-1} y\|_{G(\C)^1} 
&\leq c_1 \Delta_{\C}^-(\sigma)^{c_2}\;\text{ and }\\
\|m^{-1} u m\|_{G(\C)^1} 
&\leq c_1 
\end{align*}
where $y=mvk\in M_2(\C) U_2(\C)\cpt_{\C}$ is the Iwasawa decomposition of $y$ with respect to $P$.
(Recall the definition of $\Delta_{\C}^-(\sigma)$ from Section~\ref{section_contr_eq_classes}.)
\end{lemma}

\begin{proof}
We keep the notation from the previous lemma and its proof.
Let $\zeta_1, \ldots, \zeta_t\in\C$ be the pairwise different eigenvalues of $\sigma$ and let $n_1, \ldots, n_t$ be their respective multiplicities. 
Conjugating everything with a representative of a Weyl group element in $\cpt_{\C}$ if necessary, we can assume that $\sigma=\diag(\zeta_1, \ldots, \zeta_1, \ldots, \zeta_t, \ldots, \zeta_t)$ (each $\zeta_i$ occurring $n_i$-times in a consecutive sequence) so that $G_{\sigma}(\C)=\GL_{n_1}(\C)\times\ldots\times\GL_{n_t}(\C)=M_2(\C)$ is diagonally embedded in $G(\C)$, and $P_2=M_2U_2$ is a standard parabolic. 

Suppose $y$ is as in the lemma. By Iwasawa decomposition and the $\cpt_{\C}$-invariance of the norm we can assume that $y=mv\in M_2(\C) U_2(\C)$. 
Let $\delta_1=m^{-1}$, $y_1=\delta_1 y=v$, and $u_1=\delta_1 u\delta_1^{-1}=m^{-1}um$. Then $\|y_1^{-1}\sigma u_1 y_1\|_{G(\C)^1}\leq e^R$. 
Moreover, there exists $k_1\in\cpt_{\C}\cap M_2(\C)$ such that $u_2:=k_1 u_1 k_1^{-1}\in U_0^{M_2}(\C)$. Then
\[
 e^R
\geq \|(y_1^{-1} k_1^{-1})\sigma (k_1 u_1 k_1^{-1}) (k_1y_1)\|_{G(\C)^1}
=\|y_2^{-1}\sigma u_2 y_2\|_{G(\C)^1}
=\|\tilde y_2^{-1}\sigma u_2 \tilde y_2\|_{G(\C)^1}
\]
for $y_2:=\delta_2 y=k_1y_1$, $\delta_2:= k_1\delta_1$, and $\tilde y_2:=y_2 k_1^{-1}\in U_2(\C)$.
Then 
\[
\tilde y_2^{-1}\sigma u_2 \tilde y_2
= \big(\sigma u_2\big) \big((\sigma u_2)^{-1}\tilde y_2^{-1} (\sigma u_2) \tilde y_2\big)
\]
 is the Iwasawa decomposition with respect to $P_2=M_2U_2$. 
 Hence applying Lemma~\ref{est_on_norms}~\ref{lemma_est_on_norms_iwasawa} twice, we can find constants $a_1, a_2, a_3>0$ depending only on $n$ and $R$ such that
 \[
  \|\sigma\|_{G(\C)^1}\le a_1,\;\;
  \|u_2\|_{G(\C)^1}\le a_2, \;\text{ and }\;\;
  \|(\sigma u_2)^{-1}\tilde y_2^{-1} (\sigma u_2) \tilde y_2\|_{G(\C)^1}\le a_3.
 \]
This implies the asserted estimate on $m^{-1}um=k_1^{-1}u_2 k_1$. 
On the other hand,  $\tilde y_2^{-1}\sigma u_2 \tilde y_2= \sigma (\sigma^{-1}\tilde y_2^{-1}\sigma u_2 \tilde y_2)$ is the Iwasawa decomposition of $\tilde y_2^{-1}\sigma u_2 \tilde y_2$ with respect to $P_0(\C)=T_0(\C)U_0(\C)$ so that by Lemma~\ref{est_on_norms}~\ref{lemma_est_on_norms_iwasawa} we can also find some $a_4>0$ depending only on $n$ and $R$ such that
\[
 \|\sigma^{-1}\tilde y_2^{-1}\sigma u_2 \tilde y_2\|_{G(\C)^1}\le a_4
\]
This together with  Lemma~\ref{lemma:norm:bound:unipotent} below yields the assertion for $v=m^{-1}y=y_1=k_1^{-1}y_2=k_1^{-1}\tilde y_2 k_1$. 
\end{proof}

\begin{lemma}\label{lemma:norm:bound:unipotent}
 Let $\sigma\in T_0(C)$ be such that $M_2(\C):=G_{\sigma}(\C)$ is a standard Levi subgroup, and let $P_2=M_2U_2$ be the corresponding standard parabolic. 
Then for every $u\in U_2(\C)$, $w\in U_0(\C)\cap M_2(\C)$ we have
\[
 \|u\|_{G(\C)^1}\leq c_1 \|\sigma^{-1}u^{-1} \sigma wu\|_{G(\C)^1}^{c_2}\Delta_\C^-(\sigma)^{c_3}
\]
for $c_1,c_2, c_3>0$ constants depending only on $n$.
\end{lemma}
\begin{proof}
Note that for any $\alpha\in\Phi^+\backslash\Phi^{M_2,+}$ we have $\alpha(\sigma)\neq1$. Let $u,w$ be as in the lemma.
For each $\alpha\in\Phi\backslash\Phi^{M_2,+}$ let $X_{\alpha}\in\C$ be the matrix entry in $u$ corresponding to $\alpha$, and for $\beta\in\Phi^{M_2,+}$ let $Y_{\beta}\in\C$ be the entry in $w$ corresponding to $\beta$.
Then for every $\alpha\in\Phi^+\backslash\Phi^{M_2,+}$ the entry in the matrix $\sigma^{-1}u^{-1}\sigma w u$ corresponding to $\alpha$ equals
\[
 (1-\alpha(\sigma)^{-1}) X_{\alpha} + P_{\alpha}
\]
where $P_{\alpha}$ is a polynomial in $Y_{\beta}$, $\beta\in\Phi^{M_2,+}$, and those $\alpha'(\sigma)^{-1}$, $X_{\alpha'}$ with $\alpha'<\alpha$ (with $<$ the usual partial ordering on $\Phi^+\backslash \Phi^{M_2,+}$). The coefficients and degrees of the polynomials $P_{\alpha}$ depend only on $n$ and $M_2$ of course. 
On the other hand, for $\beta\in\Phi^{M_2,+}$ the entry in $\sigma^{-1}u^{-1}\sigma w u$ corresponding to $\beta$ equals $Y_{\beta}$.

Now let $k_1, k_2\in \cpt_{\C}$ and $H\in\aaa$ be such that $x:=\sigma^{-1}u^{-1} \sigma w u=k_1e^H k_2$. Then, writing $x_{\alpha'}$ for the matrix entry in $x$ corresponding to $\alpha'\in\Phi^+$, Lemma~\ref{est_on_norms}~\ref{lemma_est_on_norms_iwasawa} implies that for all $\alpha\in\Phi^+$ we have
\[
 |x_{\alpha'}|_\C\le n^4\|x\|_{G(\C)^1}^4
\]
so that
\[
 |(1-\alpha(\sigma)^{-1}) X_{\alpha} + P_{\alpha}|_{\C}
\leq n^4 \|\sigma^{-1}u^{-1} \sigma w u\|_{G(\C)^1}^4
\]
for every $\alpha\in\Phi^+\backslash\Phi^{M_2,+}$, and
\[
 |Y_{\beta}|_{\C}
\leq n^4 \|\sigma^{-1}u^{-1} \sigma w u\|_{G(\C)^1}^4
\]
for every $\beta\in\Phi^{M_2,+}$.
Hence proceeding inductively (with respect to the order $<$ on $\Phi^+$), we can find constants $a_1, a_2, a_3>0$ depending only on $n$ such that
\[
 |X_{\alpha}|_{\C}
\leq a_1 \|\sigma^{-1}u^{-1} \sigma w u\|_{G(\C)^1}^{a_2} \Delta_\C^-(\sigma)^{a_3}.
\]
Arguing as in the proof of Lemma~\ref{est_on_norms}~\ref{lemma_est_on_norms_iwasawa} we can conclude that
\[
 \|u\|_{G(\C)^1}
 \le b_1 \|\sigma^{-1}u^{-1} \sigma w u\|_{G(\C)^1}^{b_2} \Delta_\C^-(\sigma)^{b_3}
\]
with $b_1, b_2, b_3>0$ suitable constants depending only on $n$ as asserted.
\end{proof}

\begin{cor}
Let $\ooo\in\OOO_{R, \bkappa}$ and let $\sigma\in G(\F)_{\text{ss}}$ be a representative of the semisimple conjugacy class associated with $\ooo$.
Let $\sigma_1\in G(\C)\cap T_0(\C)$ be an arbitrary element in the $G(\C)$-conjugacy class of $\sigma$ intersected with $T_0(\C)$. Let $M_2(\C)=G_{\sigma_1}(\C)$, and $P_2=M_2U_2$ a parabolic with Levi  component $M$.
Then for every  $y\in G(\C)$ and $u\in \UUU_{G_{\sigma_1}}(\C)$ with $\|y^{-1}\sigma_1 uy\|_{G(\C)^1}\leq e^R$ we have
\begin{align*}
 \|m^{-1} y\|_{G(\C)^1} 
&\leq c_1  \Pi_{\bkappa}^{c_2}, \;\;\text{ and }\\
\|m^{-1} u m\|_{G(\C)^1} 
&\leq c_1 
\end{align*}
for $c_1, c_2>0$ suitable constants depending only on $n$ and $R$, and $y=mvk\in M_2(\C)U_2(\C)\cpt_{\C}$ the Iwasawa decomposition of $y$ with respect $P_2$.
\end{cor}
\begin{proof}
 This follows from Lemma~\ref{lemma:bound:norm:arch} together with Corollary~\ref{cor:bound:ev}.
\end{proof}

\subsection{Bounds on orbits}\label{subsec:bounds:orb}
As explained in the introduction, we want to bound from above the absolute value of $\bfc(\lambda)^{-2}\phi_{\lambda}(x)$ at every point $x$ in certain orbits in $G(\C)$ to bound the orbital integrals at the archimedean places we are interested in.
The results in Section~\ref{sec:spherical:fcts} imply that it therefore suffices to find a bound for the distance of every point of an conjugacy class to $\cpt_\C$.
This is the content of this section.
However, not every conjugacy class can be bounded away from $\cpt_{\C}$ (the class of $1$ is an obvious example). We will later distinguish between ``almost unipotent'' elements (for which this is not possible) and the remaining orbits. The crucial results in this section are Corollary~\ref{cor:upper:bound:spherical:fct:orbit} and Corollary~\ref{cor:upper:bound:spherical:fct:arch}.

\begin{lemma}\label{lemma:lower:bounds:norm}
Let $\sigma\in G(\C)$ be a diagonal matrix with eigenvalues $\sigma_1, \ldots, \sigma_n\in\C$.
 \begin{enumerate}[label=(\roman{*})]
  \item Suppose $\gamma\in G(\C)$ has eigenvalues $\sigma_1, \ldots, \sigma_n$. 
Then for every $x\in G(\C)$ we have 
\[\|x^{-1}\gamma x\|_{G(\C)^1}\geq |\det\sigma|_{\C}^{-1/2n} \sqrt{\frac{\bar\sigma_1\sigma_1+\ldots+\bar\sigma_n\sigma_n}{n}}.\]

\item\label{lower_bounds:second} Let $\sigma_{\max}, \sigma_{\min}\in\{\sigma_1, \ldots, \sigma_n\}$ be such that $|\sigma_{\max}|_{\C}=\max_{i=1, \ldots, n}|\sigma_i|_{\C}$ and $|\sigma_{\min}|_{\C}=\min_{i=1, \ldots, n}|\sigma_i|_{\C}$.
Then
\[
\frac{\bar\sigma_1\sigma_1+\ldots+\bar\sigma_n\sigma_n}{n}- |\det\sigma|_{\C}^{1/n}
\geq 
\max\left\{
\frac{\big(|\sigma_{\max}|_{\C}-|\det\sigma|_{\C}^{1/n}\big)^2}{2n |\sigma_{\max}|_{\C}},
\frac{\big(|\det\sigma|_{\C}^{1/n}-|\sigma_{\min}|_{\C}\big)^2}{2n |\sigma_{\max}|_{\C}}
\right\}
\]
Moreover, if the right hand side of this inequality vanishes, then $\sigma_1\bar\sigma_1=\ldots=\sigma_n\bar\sigma_n$.

\item\label{lower:bound:difference:ev} Suppose the characteristic polynomial of $\sigma$ has coefficients in the ring of integers of an imaginary quadratic number field $\F$, and that $\sigma$ has at least two eigenvalues of distinct absolute value. Then
\[
|\det\sigma|_{\C}^{-1/2n}  \sqrt{\frac{\bar\sigma_1\sigma_1+\ldots+\bar\sigma_n\sigma_n}{n}}
\geq 1+ \frac{1}{8n^2} |\sigma_{\max}|_{\C}^{-2n-4}.
\]
 \end{enumerate}
\end{lemma}
\begin{proof}
 \begin{enumerate}[label=(\roman{*})]
  \item Without loss of generality we can assume that $x\in U_0(\C)$ and $x^{-1}\gamma x=\sigma u$ for some $u\in U_0(\C)$. 
Write $x^{-1}\gamma x=s k_1 e^X k_2$ with $k_1, k_2\in\cpt_{\C}$, $X\in \overline{\aaa^+}$, and $s\in\R_{>0}$. Let $\tilde\gamma=|\det \gamma|_\C^{-1/2n}\gamma$ so that $\tilde\gamma$ has eigenvalues $\tau_i:=|\det \gamma|_\C^{-1/2n}\sigma_i=|\det \sigma|_\C^{-1/2n}\sigma_i$ and so that $\|x^{-1}\gamma x\|_{G(\C)^1} =\|x^{-1}\tilde\gamma x\|_{G(\C)^1}= e^{\max\{X_1,-X_n\}}$. 
Let $\tau$ be the diagonal matrix with entries $\tau_1, \ldots, \tau_n$ so that $x^{-1}\tilde\gamma x=\tau \tilde u$ for some $\tilde u\in U_0(\C)$. Then
\[
 \sum_{i=1}^n \bar\tau_i \tau_i 
 =\tr\bar\tau^t\tau 
 \leq \tr\overline{(x^{-1}\tilde\gamma x)}^t(x^{-1}\tilde\gamma x)
 = e^{2X_1}+\ldots+ e^{2X_n}
\]
so that there exists $i_0$ with
\[
2 X_{i_0}\geq\log\frac{\bar\tau_1\tau_1+\ldots+\bar\tau_n\tau_n}{n}.
\]
Note that because of the arithmetic geometric mean inequality we have $(\bar\tau_1\tau_1+\ldots+\bar\tau_n\tau_n)/n\ge 1$.
Since by assumption $X_1\geq\ldots\geq X_n$, this inequality holds in particular for $i_0=1$. Since $X_1+\ldots+X_n=0$, we get
\[
\max\{ X_1,-X_n\}\geq\frac{1}{2}\log\frac{\bar\tau_1\tau_1+\ldots+\bar\tau_n\tau_n}{n}
\]
so that $\|x^{-1}\gamma x\|_{G(\C)^1} \geq \big(\frac{\bar\tau_1\tau_1+\ldots+\bar\tau_n\tau_n}{n}\big)^{1/2}$ as asserted.

\item  
Using a refined version of the arithmetic-geometric mean inequality from~\cite{Alz97}, we get
\begin{align*}
 \frac{\bar\sigma_1\sigma_1+\ldots+\bar\sigma_n\sigma_n}{n}
& \geq \big(\prod_{i=1}^{n} \bar\sigma_i\sigma_i\big)^{1/n} + \frac{1}{2n |\sigma_{\max}|_{\C}} \sum_{j=1}^n \left(|\sigma_j|_{\C}-\big(\prod_{i=1}^{n} \bar\sigma_i\sigma_i\big)^{1/n}\right)^2 \\
& = |\det\sigma|_{\C}^{1/n} + \frac{1}{2n |\sigma_{\max}|_{\C}} \sum_{j=1}^n \left(|\sigma_j|_{\C}-|\det\sigma|_{\C}^{1/n}\right)^2 \\
& \geq |\det\sigma|_{\C}^{1/n} + \frac{\left(|\sigma_{i}|_{\C}-|\det\sigma|_{\C}^{1/n}\right)^2}{2n |\sigma_{\max}|_{\C}}  
\end{align*}
for any $i\in\{1, \ldots, n\}$, so in particular the stated inequality holds.

For the last claim suppose that $|\zeta|_{\C}-|\det\sigma|_{\C}^{1/n}=0$ where $\zeta=\sigma_{\max}$ if 
\[
|\sigma_{\max}|_{\C}-|\det\sigma|_{\C}^{1/n}>|\det\sigma|_{\C}^{1/n}-|\sigma_{\min}|_{\C}
\]
and $\zeta=\sigma_{\min}$ otherwise. Then in the first case,
\[
 \prod_{i=1}^n|\sigma_i|_{\C}=|\det\sigma|_{\C}= |\sigma_{\max}|_{\C}^n
\]
which implies $|\sigma_i|_{\C}=|\sigma_{\max}|_{\C}$ for all $i$ since by definition $|\sigma_i|_{\C}\leq |\sigma_{\max}|_{\C}$.
If $\zeta=\sigma_{\min}$, then similarly,
\[
 \prod_{i=1}^n|\sigma_i|_{\C}=|\det\sigma|_{\C}= |\sigma_{\min}|_{\C}^n,
\]
which again implies that $|\sigma_i|_{\C}=|\sigma_{\min}|_{\C}$ for all $i$ since $|\sigma_i|_{\C}\geq |\sigma_{\min}|_{\C}$.

\item 
Suppose $s_0=1< s_1< s_2<\ldots<s_r\leq n$ are such that $\sigma_{s_i}, \ldots, \sigma_{s_{i+1}-1}$ are the conjugates (over $\F$) of $\sigma_{s_i}$. (So $r=1$ and $s_r=n$ if the characteristic polynomial of $\sigma$ is irreducible over $\F$.)
Let $\zeta$ be as before.
Without loss of generality we assume that $\zeta=\sigma_1$.

Note that $|\det\sigma|_{\C}=\prod_{j=1}^n \sigma_j\bar\sigma_j$ is the product of the norms of the $\sigma_{s_i}$, $i=0, \ldots,r-1$, over $\Q$. In particular, $|\det\sigma|_{\C}\in\Z$ and $\sigma_j, \bar\sigma_j$ are algebraic integers for every $j$. Hence
\[
 \pi_i:=\prod_{j=s_i}^{s_{i+1}-1} \big((\sigma_j\bar\sigma_j)^n-|\det\sigma|_{\C}\big)
\]
is an algebraic integer, and moreover $\pi_i$ is invariant under every $\Q$-automorphism of $\Q(\sigma_{s_i})$ so that $\pi_i\in\Z$.
We claim that $\pi_0\neq0$ so that $|\pi_0|\geq1$. 
By assumption there are two eigenvalues of $\sigma$ of distinct absolute value so that the last claim in~\ref{lower_bounds:second} implies that 
$|\sigma_1|_{\C}^{n}=|\zeta|_{\C}^{n}>|\det\sigma|_{\C}^{1/n}$ if $\zeta=\sigma_{\text{max}}$ or $|\sigma_1|_{\C}^{n}<|\det\sigma|_{\C}^{1/n}$ if $\zeta=\sigma_{\min}$. Suppose that for one $i\in\{2, \ldots, s_1-1\}$ we have
$|\sigma_i|_{\C}^{n}=|\det\sigma|_{\C}$.  
Since $|\det\sigma|_{\C}\in\Z$, this implies that every conjugate of $\sigma_i$ satisfies this equality, that is $|\sigma_j|_{\C}^{n}=|\det\sigma|_{\C}$ for every $j\in\{s_0=1, \ldots, s_1-1\}$ in contradiction with the above. Hence $|\pi_0|\geq1$, and
\[
 1\leq |\pi_0| = \prod_{j=1}^{s_1-1} \big|(\sigma_j\bar\sigma_j)^n-|\det\sigma|_{\C}\big|
\leq \max_{1\leq i\leq n} \big||\sigma_j|_{\C}^{n}-|\det\sigma|_{\C}\big|^{s_1}
= \big||\zeta|_{\C}^{n}-|\det\sigma|_{\C}\big|^{s_1}
\]
so that $||\zeta|_{\C}^{n}-|\det\sigma|_{\C}|\geq1$.
Using the usual formula $a^n-b^n=(a-b)\sum_{k=0}^{n-1}a^k b^{n-1-k}$, we obtain
\[
\big| |\zeta|_{\C}-|\det\sigma|_{\C}^{1/n}\big|
\geq \frac{1}{\sum_{k=0}^{n-1} |\zeta|_{\C}^{k}  |\det\sigma|_{\C}^{(n-1-k)/n}}
\geq \frac{1}{n} |\sigma_{\max} |_{\C}^{-n-1}
\]
where for the last inequality we used $|\det\sigma|_{\C}\leq |\sigma_{\max}|_{\C}^{n}$.
Hence~\ref{lower_bounds:second} implies
\[
 \frac{\bar\sigma_1\sigma_1+\ldots+\bar\sigma_n\sigma_n}{n|\det\sigma|_{\C}^{1/n}}
\geq 1 + \frac{1}{2n^2} \frac{|\sigma_{\max}|_{\C}^{-2n-2}}{|\sigma_{\max}|_{\C}|\det\sigma|_{\C}^{1/n}}
\geq 1 + \frac{1}{2n^2} |\sigma_{\max}|_{\C}^{-2n-4}.
\]
Taking square-roots on both sides and using $\sqrt{1+\eps}\geq 1+\eps/4$ for $0<\eps\leq 1$, we obtain
\[
|\det\sigma|_{\C}^{-1/2n}  \sqrt{\frac{\bar\sigma_1\sigma_1+\ldots+\bar\sigma_n\sigma_n}{n}}
\geq 1+ \frac{1}{8n^2} |\sigma_{\max}|_{\C}^{-2n-4}.
\]
This finishes the proof.
\end{enumerate}
\end{proof}

\begin{cor}\label{cor:rootvalue:bounded:away:0}
 Let $\gamma\in G(\C)$ and suppose that the $G(\C)$-orbit of $\gamma$ intersects some $\ooo\in \OOO_{R,\bkappa}$ non-trivially. Write $\gamma=s k_1e^X k_2$ with $s\in\R_{>0}$, $k_1, k_2\in\cpt_{\C}$ and $X\in\overline{\aaa^+}$. If $X\neq0$, there exists $\alpha\in\Phi^+$ such that
\begin{equation}\label{eq:rootvalue:bounded:away:0}
 \alpha(X)\geq c_1\Pi_{\bkappa}^{-c_2}
\end{equation}
for suitable constants $c_1, c_2>0$ depending only on $n$, $R$, and $\F$.
\end{cor}
\begin{proof}
 By Proposition~\ref{prop_of_contr_classes} there exist $c_1', c_2'$ depending only on $n$, $R$, and $F$ such that 
\[
 |\sigma_i|_{\C}\leq c_1' \Pi_{\bkappa}^{-c_2'}
\]
for all eigenvalues $\sigma_1, \ldots, \sigma_n$ of $\gamma$. The assumption $X\neq0$ implies that there exist at least two eigenvalues of $\gamma$ of different absolute value so that we can apply Lemma~\ref{lemma:lower:bounds:norm}, and in particular its part~\ref{lower:bound:difference:ev}. Combining this with the upper bound on the absolute value of the eigenvalues just stated, we get
\[
\max_{\alpha\in\Phi^+}\alpha(X)
=X_1-X_n
\ge \max\{X_1,-X_n\}
= \|X\|_W
=\log \|\gamma\|_{G(\C)^1}
\geq \log\Big(1+ c_1'' \Pi_{\bkappa}^{-c_2''}\Big)
\]
for suitable constants $c_1'', c_2''$ depending only on $n$, $R$, and $\F$. Using $\log(1+\eps)\geq \eps/2$ if $\eps<1$ and choosing a smaller $c_2''$ if necessary, we obtain the assertion.
\end{proof}

If $\gamma=s k_1 e^X k_2$ is as in the corollary, i.e.\ in particular $X\in\overline{\aaa^+}$ and $X\neq0$,  then $\alpha(X)\geq0$ for all $\alpha\in\Delta_0$. Hence there exists $\alpha_0\in\Delta_0$ such that
\[
 \alpha_0(X)\geq \frac{c_1}{n-1}\Pi_{\bkappa}^{-c_2},
\]
and therefore~\eqref{eq:rootvalue:bounded:away:0} already holds for some $\alpha_0\in\Delta_0$ (with a possibly different constant).
 Let $\tilde\Delta=\Delta_0\minus\{\alpha_0\}$, and $\tilde\Phi^+\subseteq \Phi^+$ the set of positive roots corresponding to the set of simple roots $\tilde\Delta$. (We use the notation introduced in Section~\ref{subsec:sphericalfcts}.) Then $\tilde G$ is a maximal Levi subgroup, $\Delta_1=\{\alpha_0\}$, and every $\alpha\in \Phi^+_1=\Phi^+\minus\tilde\Phi^+$ satisfies 
\[
 \alpha(X)\geq \frac{c_1}{n-1}\Pi_{\bkappa}^{-c_2}.
\]
Hence
\[
\Big( \prod_{\alpha\in\Phi^+_1}\sinh\alpha(X)\Big)^{-1}
\leq \Big(\prod_{\alpha\in\Phi^+_1} \alpha(X)\Big)^{-1}
\leq c_3 \Pi_{\bkappa}^{c_2|\Phi^+_1|}.
\]
for some constant $c_3>0$ depending only on $n$, $R$, and $\F$. Together with Corollary~\ref{cor:est:spherical_times_plancherel} this implies:

\begin{cor}\label{cor:upper:bound:spherical:fct:orbit}
 Suppose $\gamma=sk_1e^X k_2$, $X\neq0$, is as in Corollary~\ref{cor:rootvalue:bounded:away:0}. Then
\[
 \big| \bfc(\lambda)^{-2} \phi_{\lambda}(\gamma)\big|
\leq c_1 (1+\|\lambda\|)^{d-r-|\Phi_1^+|}
\Pi_{\bkappa}^{c_2}.
\]
for all $\lambda\in i\aaa^*$, where $c_1, c_2>0$ are suitable constants depending only on $n$, $R$, and $\F$.
\end{cor}

\begin{rem}
 The condition $X\neq0$ depends only on the equivalence class $\ooo$ which the $G(\C)$-orbit of $\gamma$ intersects non-trivially, but not on $\gamma$ itself: We have $X=0$ if and only if $|\det\gamma|_{\C}^{-1/n}\gamma\in\cpt_{\C}$, and this is the case only if all eigenvalues of $\gamma$ have the same absolute value, i.e.\ if $\gamma$ is almost unipotent in the terminology introduced in Definition~\ref{def:almost:unipotent} below. 
However, the set of eigenvalues of $\gamma$ determines $\ooo$ uniquely (and vice versa). Hence if $\ooo$ is not almost unipotent, then $X\neq0$.
\end{rem}

\begin{cor}\label{cor:upper:bound:spherical:fct:arch}
 If $h\in C_c^{\infty}(\aaa)^W$ and $\gamma$ is as in Corollary~\ref{cor:upper:bound:spherical:fct:orbit}, then for all $t\geq1$ we have
\[
\int_{t\Omega} \big| f_{\C}^{\mu}(\gamma)\big| \,d\mu
\leq \int_{t\Omega} \int_{i\aaa} \big|\hat{h}(\lambda+\mu) \phi_{-\lambda}(\gamma) \bfc(\lambda)^{-2}\big| \,d\lambda\,d\mu
\leq c_1 \Pi_{\bkappa}^{c_2} t^{d-|\Phi^+_1|},
\]
and for all $t\geq1$ and $\mu\in i\aaa^*$,
\[
 \big| f_{\C}^{t, \mu}(\gamma)\big|
\leq c_3 \Pi_{\bkappa}^{c_2} (t+\|\mu\|)^{d-r-|\Phi_1^+|}
\]
for $c_2,c_3>0$ constants depending only on $n$, $R$, and $\F$, and $c_1>0$ a constant depending only on $n$, $R$, $\F$, $h$, and $\Omega$.
(Recall that $d=\dim X=n^2-1$ and $r=\dim\aaa=n-1$.)
\end{cor}
\begin{proof}
Both estimates are direct consequences of Corollary~\ref{cor:upper:bound:spherical:fct:orbit} after unfolding the definition of $f_{\C}^{t,\mu}$ and changing the order of integration (which we are allowed to do, since everything is absolutely convergent).
\end{proof}

\section{Weighted orbital integrals at the archimedean place}\label{sec:weighted:orb:int:arch}
The goal in this section is to prove archimedean analogues of the results of Section~\ref{subsec:weighted:orb:int:non-arch}, i.e.\
we want to bound for every $M\in\LLL$ with $\gamma_{\C}\in M(\C)$ and all $t\geq1$ the integral
\begin{equation}\label{eq:int:over:Omega}
 \int_{t\Omega} \Big| J_M^{L_0}(\gamma_{\C}, (f_{\C}^{\mu})^{(Q_0)}) \Big|\,d\mu
\end{equation}
for every $L_0\in\LLL(M)$ and $Q_0\in\PPP(L_0)$.
 To that end we distinguish two different types of $\gamma_{\C}$.

\subsection{Almost unipotent orbits} 
\begin{definition}\label{def:almost:unipotent}
\begin{enumerate}[label=(\roman{*})]
\item 
We call an equivalence class $\ooo\in \OOO$ \emph{central} if it corresponds to the semisimple conjugacy class of a central element $\sigma=s\One_n\in T_0(\F)$ for some $s\in\F^{\times}$. (Then clearly $\ooo= s\UUU_{G}(\F)$.)

\item 
We call an element $\gamma=\sigma \nu\in G(\C)$  \emph{almost unipotent} if the eigenvalues of $\sigma$ all have the same absolute value. 

\item
We call  $\ooo\in\OOO$ \emph{almost unipotent} if one (and hence any) element in the  $G(\C)$-orbit of $\ooo$ is almost unipotent. 
\end{enumerate}
\end{definition}

If $\gamma$ is almost unipotent, we may assume - after possibly conjugating $\gamma$ by an element in $G(\C)$ - that it is of the form $\gamma=\sigma\nu=s \tau \nu$ with $ s=|\det\sigma|^{1/n}\in\R^{\times}$,
\begin{equation}\label{eq:tau:roots:unity}
\tau=s^{-1}\sigma= \diag(\eps_1, \ldots, \eps_1, \ldots, \eps_r, \ldots, \eps_r),
~~~\text{ and }~~~~
\nu\in U_0^{M_2}(\C)=U_0(\C)\cap M_2(\C),
\end{equation}
where $\eps_1, \ldots, \eps_r\in\C^1$ are suitable numbers of absolute value $1$, $\eps_i\neq\eps_j$ if $i\neq j$,  $\tau$ has $n_i$-many $\eps_i$'s on the diagonal, and $M_2(\C)=\GL_{n_1}(\C)\times\ldots\times\GL_{n_r}(\C)=G_{\sigma}(\C)$ diagonally embedded in $G(\C)$. 
If the $G(\C)$-orbit of $\gamma$ intersects some equivalence class $\ooo$ non-trivially, the numbers $n_1, \ldots, n_r$ are independent on $\gamma$ and only depend on $\ooo$.
In particular, there exists $i$ with $1\leq n_i\leq n-1$ if $\gamma$ does not correspond to a central class $\ooo$. 

Let $M\in\LLL$ be arbitrary with $\gamma\in M(\C)$. 
Let $L_0\in\LLL(M)$ and $Q_0=L_0V_0\in \PPP(L_0)$. 
The weighted orbital integrals $J_M^{L_0}(\gamma, f^{(Q_0)})$, $f\in C_c^{\infty}(G(\C)^1)$, do not depend on the central component $s$ so that
\[
 J_M^{L_0}(\gamma, f^{(Q_0)})
= J_M^{L_0}(\tau\nu, f^{(Q_0)}).
\]
We use the description of weighted orbital integrals from Section~\ref{sec:weighted:orb:int:non-arch}.
\begin{definition}
 Let $P=M_PU_P$ be a semi-standard parabolic subgroup. We call a function $w: U_P(\C)\longrightarrow\C$ a \emph{logarithmic polynomial weight function for $P$} if there exist integers $k, l, m\geq0$ and polynomials $p_1, \ldots, p_k: U_P(\C)\longrightarrow\C^m$ and $q:\C^k\longrightarrow\C$ such that for every $u\in U(\C)$,
\[
 w(u)=q\big(\log\|p_1(u)\|_{\C}, \ldots, \log\|p_k(u)\|_{\C}\big)
\]
where $\|\cdot\|_{\C}$ denotes the usual vector norm on $\C^m$.
\end{definition}

\begin{lemma}\label{almost_unipotent_weight_int}
Let $\gamma=s\tau\nu$ be as above. There exist a semi-standard parabolic subgroup $P=M_PU_P$, an integer $N>0$, and logarithmic polynomial weight functions for $P$
\[
 w_{\bfk} :U_P(\C)\longrightarrow\C
\]
for every tuple $\bfk=(k_{\alpha})_{\alpha\in\Phi^+\minus\Phi^{M_2,+}}$, $k_{\alpha}\in\Z_{\geq0}$, with $|\bfk|=\sum_{\alpha}k_{\alpha}\leq N$ such that for every  bi-$\cpt_{\C}$-invariant function $f\in C^{\infty}(G(\C))$ of almost compact support we have
\[
 J_M^{L_0}(\gamma, f^{(Q_0)})
=  |D^{L_0}(\tau)|_{\C}^{1/2}\sum_{\bfk:~|\bfk|\leq N} \Big(\prod_{\alpha\in\Phi^+\minus\Phi^{M_2, +}} (\log|1-\alpha(\tau)|_{\C})^{k_{\alpha}}\Big)
\int_{U_P(\C)} f(u) w_{\bfk}(u) \,du.
\]
Moreover, $P$, $N$, and the functions $w_{\bfk}$ depend only on the partition $n_1, \ldots, n_r$ of $n$, the Levi subgroup $M$, the parabolic $Q_0$, and the conjugacy class of $\nu$ in $M$ (but not on the values of the $\eps_i$'s).

Further, $\dim U_P\geq1$ unless $\gamma$ is central and $M=G$.
\end{lemma}

\begin{proof}
Let $L_1=L_0\cap M_2$ be the centraliser of $\tau$ in $L_0$, and let $Q_1=L_1V_1\in\PPP^{L_0}(L_1)$. Put $M_1=M\cap M_2$. Then, since $f$ is bi-$\cpt_{\C}$-invariant, we have
\[
 J_M^{L_0}(\gamma, f^{(Q_0)})
= |D^{L_0}(\tau)|_{\C}^{1/2}\int_{V_1(\C)} \sum_{R_1\in \FFF^{L_1}(M_1)} J_{M_1}^{M_{R_1}}(\nu, \Phi_{R_1, y}^{L_0}) \,dy,
\]
where for $m\in M_{R_1}(\C)$ we have
\begin{align*}
 \Phi_{R_1,y}^{L_0}(m)
&= \delta_{R_1}^{L_0}(m)^{1/2} \int_{\cpt_{\C}^{L_1}}\int_{N_{R_1}(\C)} f^{(Q_0)} (y^{-1}k^{-1}\tau m n ky) v_{R_1}'(ky) \,dn\,dk\\
&= \delta_{R_1}^{L_0}(m)^{1/2} \int_{\cpt_{\C}^{L_1}}\int_{N_{R_1}(\C)} \delta_{Q_0}(y^{-1}k^{-1}\tau m n k y)^{1/2}\int_{V_0(\C)}  f(y^{-1}k^{-1}\tau m n  k y v) v_{R_1}'(ky)\,dv\,dn\,dk\\
&= \delta_{R_1}^{L_0}(m)^{1/2}  \delta_{Q_0}(\tau m)^{1/2} \int_{\cpt_{\C}^{L_1}} \int_{N_{R_1}(\C)} \int_{V_0(\C)}  f(y^{-1}k^{-1} \tau m n k y v)v_{R_1}'(ky)\,dv\,dn\,dk,\\
\end{align*}
where  the first equality is the definition of the function $\Phi_{R_1, y}$, and the remaining equalities follow from unfolding the definition of $f^{(Q_0)}$.

The function $\Phi_{R_1, y}^{L_0}$ is clearly $\cpt_{\C}^{M_{R_1}}$-conjugation invariant. 
The unipotent class generated by $\nu$ in $M_{R_1}$ is of Richardson type so that there is a parabolic subgroup $S\subseteq M_{R_1}$ with unipotent radical $N_S$ such that the class of $\nu$ corresponds to the class associated with $N_S$.
By~\cite[\S 5]{LaMu09} there exist logarithmic polynomial weight functions $w^{R_1}$ for $S$ such that
\[
 J_{M_1}^{M_{R_1}}(\nu, f)=\int_{N_S(\C)} f(u) w^{R_1}(u)\,du
\]
for every $\cpt_{\C}^{M_{R_1}}$-conjugation invariant function $f\in C^{\infty}(M_{R_1}(\C))$ of almost compact support.
Note that the number of possibly occurring different Levi and parabolic groups, unipotent classes, and logarithmic polynomial weight functions is bounded by a constant depending only on $n$.

Hence $J_{M_1}^{M_{R_1}}(\nu, \Phi_{R_1, y}^{L_0})$ equals
\[
 \int_{N_S(\C)}w^{R_1}(u) \delta_{Q_0}(\tau)^{1/2} \int_{\cpt_{\C}^{L_1}} \int_{N_{R_1}(\C)} \int_{V_0(\C)}  f(y^{-1}k^{-1} \tau u n k y v)v_{R_1}'(ky) \,dv\,dn\,dk\,du.
\]
Changing $y$ to $kyk^{-1}$, the integral $J_M^{L_0}(\gamma, f^{(Q_0)})$ becomes
\[
 |D^{L_0}(\tau)|_{\C}^{1/2}\sum_{R_1\in \FFF^{L_1}(M_1)}
\int_{V_1(\C)}\int_{N_S(\C)} \int_{N_{R_1}(\C)} \int_{V_0(\C)}  f(y^{-1}\tau u n y  v) w^{R_1}(u) v_{R_1}'(y)\,dv\,dn\,du\,dy.
\]
The map  $V_1(\C)\ni y\mapsto \phi_{\tau u}(y):= (\tau u n)^{-1}y^{-1}(\tau u n) y\in V_1(\C)$ is an isomorphism (and the absolute value of its Jacobian is trivial) so that $J_M^{L_0}(\gamma, f^{(Q_0)})$ equals
\[
|D^{L_0}(\tau)|_{\C}^{1/2}
\sum_{R_1\in \FFF^{L_1}(M_1)}
\int_{V_1(\C)}\int_{N_S(\C)} \int_{N_{R_1}(\C)} \int_{V_0(\C)}  f(\tau un y v) w^{R_1}(u) v_{R_1}'(\phi_{\tau u}^{-1} (y))\,dv\,dn\,du\,dy. 
\]
Let $\tilde S= S N_{R_1} V_1V_0$. This is an element in $\FFF$ with the same Levi component as $S$. Denote its unipotent radical by $U_{\tilde S}$. From the definition of the weight functions it is clear that we can write 
\[
 w^{R_1}(u) v_{R_1}'(\phi_{\tau u}^{-1} (y))
= \sum_{\tilde\bfk:|\tilde\bfk|\leq \tilde N}\; \prod_{\alpha\in\Phi^{L_0,+}\minus \Phi^{L_1,+}} 
\big(\log(1-\alpha(\tau))\big)^{\tilde k_{\alpha}} \tilde w^{R_1}_{\tilde \bfk}(u,y)
\]
for a suitable integer $\tilde N$, and suitable logarithmic polynomial weight functions $\tilde w^{R_1}_{\tilde\bfk}(u,y)$ which do not depend on $\tau$, but only on the different Levi and parabolic subgroups involved. We can extend $\tilde w^{R_1}_{\tilde\bfk}(u,y)$ trivially to all of $U_{\tilde S}$ and write $\bar w^{R_1}_{\tilde\bfk} (\tilde u)$ for this new function, $\tilde u\in U_{\tilde S}(\C)$.  
Hence $J_M^{L_0}(\gamma, f^{(Q_0)})$ equals
\begin{align*}
& |D^{L_0}(\tau)|_{\C}^{1/2}
\sum_{R_1\in \FFF^{L_1}(M_1)}\;
\sum_{\tilde\bfk:|\tilde\bfk|\leq \tilde N}\; \prod_{\alpha\in\Phi^{L_0,+}\minus \Phi^{L_1,+}} \big(\log|1-\alpha(\tau)|\big)^{\tilde k_{\alpha}} \int_{U_{\tilde S}(\C)} f(\tau \tilde u) \bar w^{R_1}_{\tilde\bfk} (\tilde u)
,d\tilde u\\
= & |D^{L_0}(\tau)|_{\C}^{1/2}
\sum_{\bfk:|\bfk|\leq  N} \;\prod_{\alpha\in\Phi^{+}\minus \Phi^{M_2,+}} \big(\log|1-\alpha(\tau)|\big)^{ k_{\alpha}} \int_{U_{\tilde S}(\C)} f(\tilde u) w_{\bfk} (\tilde u)\,d\tilde u
\end{align*}
with $N=\max_{R_1} \tilde N$, and 
\[
 w_{\bfk}=\sum_{R_1\in \FFF^{L_1}(M_1)} \bar w^{R_1}_{\bfk}.
\]
For the last equality we also used that $f$ is bi-$\cpt_{\C}$-invariant and $\tau\in\cpt_{\C}$.  This proves the assertion.
\end{proof}

\begin{proposition}\label{almost_unip_int_prop}
If $M\in\LLL$, $L_0\in\LLL(M)$, $Q_0\in\PPP(L_0)$, and $\gamma=s\tau\nu\in M(\C)$ is almost unipotent, but $\tau\not\in Z(\C)$ or $M\neq G$, then
\begin{equation}\label{almost_unip_int}
 \left|\int_{t\Omega} J_M^{L_0}(\gamma, (f_{\C}^{\mu})^{(Q_0)})d\mu\right|
\leq c_1 \Big(\prod_{\alpha\in \Phi^+:\,\alpha(\tau)\neq1} \max\big\{1,|\log|1-\alpha(\tau)|_{\C}|\big\}\Big)^{c_2} t^{d-1} (\log t)^{n+1}
\end{equation}
as $t\rightarrow\infty$, where $c_1, c_2>0$ are suitable constants depending only on $n$, $R$, and $\Omega$.
\end{proposition}
\begin{proof}
Using Lemma~\ref{almost_unipotent_weight_int} we can bound the left hand side of~\eqref{almost_unip_int} by 
\[
\Big(\prod_{\alpha\in \Phi^+:\,\alpha(\tau)\neq1} \max\big\{1,|\log|1-\alpha(\tau)|_{\C}|\big\}\Big)^{c_2}
\sum_{\bfk:~|\bfk|\leq N} \bigg|\int_{t\Omega}  \int_{U_P(\C)} f_{\C}^{\mu}(u) w_{\bfk}(u) \,du\,d\mu\bigg|
\]
with $\dim U_P\geq1$ because $\tau$ is non-central or $M\subsetneq G$.
This last double integral now is of the same form as the unipotent weighted orbital integrals (with a different weight function) so that we can apply the estimate from~\cite{LaMu09} for each weighted orbital integral. (In~\cite{LaMu09} the assertion was actually proven over $\R$, but the arguments easily carry over to the complex situation.) The occuring constant depend only on $n$, $R$, $\Omega$, $L_0$, $M$, $P$, and the weight function $w_{\bfk}$. Since there are only finitely many different weight functions appearing in all possible integrals (i.e., as $\tau$ varies over all elements of the form~\eqref{eq:tau:roots:unity}), the constant $c_1$ depends only on $n$, $R$, and $\Omega$.
\end{proof}

\begin{lemma}
 If $\ooo\in\OOO_{R, \bkappa}$ is almost unipotent and $\gamma=s\tau\nu$ lies in the $G(\C)$-conjugacy class of some element of $\ooo$, then
\[
 \prod_{\alpha\in\Phi^+:\, \alpha(\tau)\neq1} \max\big\{1, |\log|1-\alpha(\tau)|_{\C}|\big\}
\leq c_1 \Pi_{\bkappa}^{c_2}
\]
for constants $c_1, c_2>0$ depending only on $n$, $R$, and $\F$.
\end{lemma}
\begin{proof}
Let $\mu(x)\in\OOO_{\F}[x]$ be the minimal polynomial of $s\tau$. Then the discriminant of $\mu$ is non-zero and has coefficients in $\OOO_{\F}$, hence has norm at least $1$. Using that our assumption on $\ooo$ implies that $s^{2n}=|\det\gamma|_\C\in\Z_{\ge1}$ we get
\begin{align*}
2^{|\Phi^+|-1}|\tau_i-\tau_j|_\C
&\ge|\tau_i-\tau_j|_\C \prod_{k<l:\, \tau_k\neq\tau_l, (i,j)\neq (k,l)}(|\tau_k|_\C+|\tau_l|_\C) 
\ge \prod_{k<l:\, \tau_k\neq\tau_l}|\tau_k-\tau_l|_\C 
 \geq \prod_{k<l:\, \tau_k\neq\tau_l}s^{-1}\\
& \ge s^{-|\Phi^+|}
\end{align*}
for any pair $k<l$ with $\tau_k\neq\tau_l$.
Hence
\[
\big|1-\frac{\tau_l}{\tau_k}\big|_{\C}
= |\tau_k-\tau_l|_{\C}
\geq s^{-|\Phi^+|} 2^{-|\Phi^+|+1}
\geq (2s)^{-|\Phi^+|}
\]
 The determinant of $\gamma$ can be bounded from above by Proposition~\ref{prop_of_contr_classes} so that the assertion follows.
\end{proof}

Recall the definition of the global test function $f_{\bxi}^{\mu}$ from~\eqref{eq:def:global:test:fct}.
 Suppose $\bkappa$ is a tuple of non-negative integers with $\kappa_v=0$ for almost all $v$.

\begin{cor}\label{cor:estimate:global:weighted:orb:unip}
There exist $c_1, c_2,c_3>0$ depending only on $n$, $R$, $\F$, and $\Omega$ but not on $\bkappa$ or $\bxi$ such that the following holds.
 Then:
\begin{enumerate}[label=(\roman{*})]
\item
 If $\ooo\in\OOO_{R,\bkappa}$ is almost unipotent but not central and $\bxi$ is such that $f_{\bxi}^{\mu}\in C_{R, \bkappa}^{\infty}(G(\A)^1)$, then
\[
\bigg|\int_{t\Omega}J_{\ooo}(f^{\mu}_{\bxi})\,d\mu\bigg|
\leq c_1 t^{d-1}(\log t)^{n+1} \Pi_{\bkappa}^{c_2} \Pi_{S, \bkappa, \ooo}^{c_3}
\]
as $t\rightarrow\infty$ for every finite set of places $S$ of $\F$ with $S^{\bkappa,\ooo}\subseteq S$.

\item
If $\ooo\in\OOO_{R,\bkappa}$ is central and corresponds to $\sigma=s\One_n\in Z(\F)$, and if $\bxi$ is such that $f_{\bxi}^{\mu}\in C_{R, \bkappa}^{\infty}(G(\A)^1)$, then
\[
\bigg|\int_{t\Omega}\big(J_{\ooo}(f^{\mu}_{\bxi})-\vol(G(\F)\backslash G(\A)^1)f^{\mu}_{\bxi}(\sigma)\big) \,d\mu\bigg|
\leq c_1 t^{d-1}(\log t)^{n+1} \Pi_{\bkappa}^{c_2}\Pi_{S,\bkappa,\ooo}^{c_3}
\]
as $t\rightarrow\infty$ for every finite set of places $S$ of $\F$ with $S^{\bkappa,\ooo}\subseteq S$.
\end{enumerate}
\end{cor} 

\begin{proof}
\begin{enumerate}[label=(\roman{*})]
\item
In the following, the $c_i\geq0$ denote suitable constants depending at most on $n$, $\F$, $R$, and $\Omega$.
Our estimates for the local weighted orbital integrals yields for any $\gamma=\sigma\nu\in M(\F)\cap \ooo$,
\[
\left|\int_{t\Omega} J_M^G(\gamma, f^{\mu}_{\bxi})\,d\mu \right|
\leq c_1\Pi_{\bkappa}^{c_2} t^{d-1}(\log t)^{n+1} \prod_{v\in S_{\bkappa}} q_v^{c_3\kappa_v}\prod_{v\in S_{\ooo}} |D^G(\sigma)|_v^{-c_4} \prod_{v\in S\minus S^{\bkappa, \ooo}} q_v^{c_5}.
\]
Now 
\[
 \prod_{v\in S_{\ooo}} |D^G(\sigma)|_v^{-c_4} 
=|D^G(\sigma)|_{\C}^{c_4}\ll_R \Pi_{\bkappa}^{c_6}
\]
for some $c_6>0$ by Lemma~\ref{prop_of_contr_classes}, and $\prod_{v\in S_{\bkappa}} q_v^{c_3\kappa_v}=\Pi_{\bkappa}^{c_3}$. Hence
\[
 \left| \int_{t\Omega} J_M^G(\gamma, f^{\mu}_{\bxi})\,d\mu \right|
\leq c_7  t^{d-1}(\log t)^{n+1} \Pi_{\bkappa}^{c_8} \Pi_{S,\bkappa,\ooo}^{c_5}.
\]
Similar estimates of course hold for the weighted orbital integrals occuring in the splitting formula.
Hence the assertion follows from this estimate together with  the fine expansion expansion of $J_{\ooo}$ in~\eqref{globdistr_wgt_int}, the estimate for the coefficients $|a^M(\gamma,S)|$ from~\eqref{eq:coeff:est}, and the splitting formula for the weighted orbital integral from Lemma~\ref{reduction_to_local_case}.

\item
This follows as the first part by noting that if $\sigma$ is central, then
\[
J_G^G(\sigma, f^{\mu}_{\bxi})
=a^G(\sigma,S)J_G^G(\sigma, f^{\mu}_{\bxi})
= \vol(G(\F)\backslash G(\A)^1) f^{\mu}_{\bxi}(\sigma)
\]
so that after subtracting $\vol(G(\F)\backslash G(\A)^1) f^{\mu}_{\bxi}(\sigma)$ from the fine expansion of $J_{\ooo}$ only terms with $M\subsetneq G$ survive. Hence we can apply the previous estimates again.
\end{enumerate}
\end{proof}

\subsection{Remaining classes}
Now suppose that $\ooo\in\OOO_{R,\bkappa}$ is an arbitrary but not almost unipotent equivalence class. Let $\sigma\in G(\F)$ be a semisimple element in $\ooo$, and  let $\sigma_\C\in T_0(\C)$ have the same eigenvalues as $\sigma$. Again we may assume that $M_2(\C)=G_{\sigma_{\C}}(\C)$ is such that we can choose  $P_2=M_2U_2$ to be standard. Let $\nu_{\C}\in M_2(\C)$ be a unipotent element and put $\gamma_{\C}=\sigma_{\C}\nu_{\C}$. 

The first calculations of the last section stay valid in our case (with $\sigma_{\C}$ in place of $\tau$) so that $\big|J_M^{L_0}(\gamma_{\C}, (f_{\C}^{\mu})^{(Q_0)})\big|$ is bounded by
\begin{multline*}
 \delta_{Q_0}(\sigma_{\C})^{1/2} |D^{L_0}(\sigma_{\C})|_{\C}^{1/2} \sum_{R_1\in \FFF^{L_1}(M_1)}
\int_{V_1(\C)}\int_{W(\C)} \int_{N_{R_1}(\C)} \int_{V_0(\C)}\\
  \big|f_{\C}^{\mu}(y^{-1}\sigma_{\C} u y  vn)\big| \big|w^{R_1}(u)\big| \big|v_{R_1}'(y)\big|\,dv\,dn\,du\,dy.
\end{multline*}
By Lemma~\ref{lemma:bound:norm:arch} and Lemma~\ref{est_on_norms} the condition $f_{\C}^{\mu}(y^{-1}\sigma_{\C} u y  vn) \neq0$ implies that
\[
 \|y\|_{G(\C)^1},\; \|u\|_{G(\C)^1},\; \|v\|_{G(\C)^1},\; \|n\|_{G(\C)^1}
\leq c_1 \Delta_{\C}^-(\sigma_{\C})^{c_2} 
\]
for constants $c_1, c_2>0$ depending only on $n$, $\F$, and $R$.
In particular, for any such $y$ we have by Corollary~\ref{cor:weight:bound:by:norm} that
\[
 |v_{R_1}'(y)|
\leq c_3 \Big(1+\log \big(\Delta^-_{\C}(\sigma_{\C}) \big)\Big)^{n-1}
\]
for all $R_1$ and some constant $c_3>0$ depending only on $n$.

By Corollary~\ref{cor:upper:bound:spherical:fct:arch}   we have for all $t\geq1$, and all $u,y,v,n$ that 
\[
  \int_{t\Omega}\big|f_{\C}^{\mu}(y^{-1}\sigma_{\C} u y  vn)\big| \,d\mu
\leq  C_1 \Pi_{\bkappa}^{C_2} t^{d-1} \chi_1(y^{-1}\sigma_{\C} u y  vn)
\]
for $C_1, C_2>0$ constants depending only on $n$, $\F$, $h$, and $\Omega$, where $\chi_y$, $\chi_u$, $\chi_v$, $\chi_n$ are the characteristic functions (on the obvious groups) of all elements having norm bounded by $c_1 \Delta_{\C}^-(\sigma_{\C})^{c_2} $.

Hence~\eqref{eq:int:over:Omega} is bounded by a constant multiple of the product of
\[
\Pi_{\bkappa}^{C_2} t^{d-n+1} \Big(1+\log \big(\Delta^-_{\C}(\sigma_{\C}) \big)\Big)^{n-1}
  \delta_{Q_0}(\sigma_{\C})^{1/2} |D^{L_0}(\sigma_{\C})|_{\C}^{1/2}
\]
with
\[
 \sum_{R_1\in \FFF^{L_1}(M_1)}
\int_{V_1(\C)}\int_{W(\C)}\int_{N_{R_1}(\C)} \int_{V_0(\C)} \chi_y(y)\chi_v(v)\chi_u(u)\chi_n(n) \big|w^{R_1}(u)\big|\,dv\,dn\,du\,dy.
\]
As in the non-archimedean case it follows from the results of Sections~\ref{sec:int:log:pol} and~\ref{subsec:meas:unip:orb} (cf.~Remark~\ref{rem:arch:case}) that this last integral is bounded by
\[
 c_4 \Big(\Delta^-_{\C}(\sigma_{\C})^{c_2} \Big)^{c_5}
\]
for $c_4, c_5>0$ constants depending only on $n$ and $R$.

Putting everything together and using Proposition~\ref{cor:bound:ev} to bound $\Delta_\C^-(\sigma_{\C})$, we find constants $c_6, c_7>0$ depending only on $n$, $\F$, $h$, and $\Omega$ such that for all $t\geq1$,
\begin{equation}\label{eq:first:bound:arch:arbitrary:class}
 \int_{t\Omega} \Big| J_M^{L_0}(\gamma_{\C}, (f_{\C}^{\mu})^{(Q_0)}) \Big|\,d\mu
\leq  c_6 \Pi_{\bkappa}^{c_7} t^{d-n+1}.
\end{equation}
Using Corollary~\ref{cor:upper:bound:spherical:fct:orbit} instead of Corollary~\ref{cor:upper:bound:spherical:fct:arch} above, we similarly find $c_8, c_9>0$ such that for all $t\geq1$, 
\begin{equation}
 \Big| J_M^{L_0}(\gamma_{\C}, (f_{\C}^{\mu})^{(Q_0)}) \Big|\,d\mu
\leq c_8 \Pi_{\bkappa}^{c_9} (t+\|\mu\|)^{d-2n+2}.
\end{equation}
Combining~\eqref{eq:first:bound:arch:arbitrary:class} with the estimates for the non-archimedean integrals from Corollary~\ref{cor:bound:weighted:orb:int:after:const:term:non.arch} and the splitting formula from Lemma~\ref{reduction_to_local_case} we have proven an upper bound for the weighted orbital integrals $\int_{t\Omega} J_M^G(\gamma, f_{\C}^{\mu}\cdot \tau_{\bxi})\,d\mu$. Combining this bound with the fine geometric expansion and the estimate for the global coefficients $a^M(\gamma, S)$, we get:

\begin{proposition}\label{prop:estimate:global:weighted:orb:non:unip}
Suppose $\ooo\in\OOO_{R, \bkappa}$ is not almost unipotent. 
Then for every $\bxi$ such that $f_{\bxi}^{\mu}\in C_{R, \bkappa}^{\infty}(G(\A)^1)$, every finite set of places $S$ of $\F$ with $S^{\bkappa,\ooo}\subseteq S$, and every $t\geq1$ we have
\[
 \int_{t\Omega} \Big| J_{\ooo}(f^{\mu}_{\bxi}) \Big|\,d\mu
\leq c_1 \Pi_{\bkappa}^{c_2} \Pi_{S,\bkappa,\ooo}^{c_3} t^{d-n+1},
\]
for constants $c_1, c_2,c_3>0$ depending only on $n$, $\F$, $R$, $h$, and $\Omega$.
\end{proposition}
\begin{proof}
The proof is similar as the proof of the first part of Corollary~\ref{cor:estimate:global:weighted:orb:unip} when we replace the estimate for the archimedean distributions by the respective estimate for non-almost unipotent classes.
\end{proof}

Recall that for every $x\in G(\C)$ we have the trivial estimate $|\phi_{\lambda}(x)|\leq1$ for all $\lambda\in i\aaa^*$. Using this estimate in our arguments to bound the orbital integrals at the archimedean place (for any (possibly almost unipotent) class $\ooo\in\OOO_{R,\bkappa}$),
we obtain the trivial (i.e., not using any non-trivial upper bounds for the spherical functions) upper bounds
\begin{equation}\label{eq:triv:upper:bound:weighted:orb}
 \Big| J_M^G(\gamma, f^{t,\mu}_{\bxi}) \Big|\,d\mu
\leq C_1 \Pi_{\bkappa}^{C_2} \Pi_{S,\bkappa,\ooo}^{C_3} t^{n-1}\hat\beta(t,\mu)
\end{equation}
and 
\begin{equation}\label{eq:triv:upper:bound:geom:distr}
 \Big| J_{\ooo}(f^{t,\mu}_{\bxi}) \Big|\,d\mu
\leq C_1 \Pi_{\bkappa}^{C_2} \Pi_{S,\bkappa,\ooo}^{C_3} t^{n-1}\hat\beta(t,\mu)
\end{equation}
for every $t\geq1$.

\section{Summary for the geometric side}\label{sec:summary:geom}
Combining Corollary~\ref{cor:estimate:global:weighted:orb:unip} and Proposition~\ref{prop:estimate:global:weighted:orb:non:unip} with the estimate on $|\OOO_{R, \bkappa}|$ from Proposition~\ref{prop_of_contr_classes}, we can can summarise our final result for the geometric side of the trace formula as follows:

\begin{theorem}\label{theorem:est:geom:side}
There exist constants $c_1, c_2,c_3>0$ such that the following holds:  For every sequence of integers $\bkappa=(\kappa_v)$ 
with $\kappa_v=0$ for almost all $v$,
and every $\bxi$ such that $f_{\bxi}^{\mu}\in C_{R, \bkappa}^{\infty}(G(\A)^1)$ we have
\begin{multline}\label{eq:est:geom:side}
\Big| \int_{t\Omega} J_{\text{geom}}(f^{\mu}_{\bxi})\,d\mu 
-\vol(G(\F)\backslash G(\A)^1)\int_{t\Omega}f_{\C}^{\mu}(1)\,d\mu \sum_{z\in Z(\F)} \tau_{\bxi}(z)\Big|\\
\leq c_1 \Pi_{\bkappa}^{c_2} \Pi_{S,\bkappa,\ooo}^{c_3}~ t^{d-1} (\log t)^{n+1}
\end{multline}
for every $t\geq2$ and every finite set of places $S$ of $\F$ containing $S^{\bkappa,\ooo}$.\qed
\end{theorem}

The ``trivial estimate`` \eqref{eq:triv:upper:bound:geom:distr} similarly yields for every  $\bxi$ as before and $t\geq2$,
\begin{equation}\label{eq:coarse:est:geom:side}
 \Big|J_{\text{geom}}(f^{t,\mu}_{\bxi})\Big|
\leq c_4 \Pi_{\bkappa}^{c_5} \Pi_{S,\bkappa,\ooo}^{c_6} t^{n-1}\hat\beta(t, \mu)
\end{equation}
for every finite set of places $S\supseteq S^{\bkappa,\ooo}$, where $c_4, c_5,c_6>0$ are constants depending only on $n$, $\F$, $R$, and $h$.

\begin{rem}\label{rem:geom:side:est:general}
 Theorem~\ref{theorem:est:geom:side} stays true if we replace $\tau_{\bxi}$ by any compactly supported smooth function $\tau:G(\A_f)\longrightarrow\C$ with $\supp\tau\subseteq \tau_{\bxi}$ and $|\tau|\leq 1$ (cf.\ Remark~\ref{rem:estimate:weighted:orb:non-arch:general}).
\end{rem}

We can express the constant $\Pi_{\bkappa}$ in terms of $\tau_{\bxi}$ without reference to $\bkappa$: Suppose that $\xi_v$ is such that $\xi_{v1}\ge\ldots\ge\xi_{vn}\ge0$. Then by Lemma~\ref{bounded_by_degree}
\[
\Pi_{\bkappa}= \prod_{v\in S_{\bkappa}} q_v^{\kappa_v}
= \prod_{v\in S_{\bkappa}} q_v^{\xi_{v1}}
\leq \prod_{v\in S_{\bkappa}} q_v^{\xi_{vn}} \big(\deg\tau_{\bxi}\big)^{a}
\]
for some $a>0$ depending only on $n$.

Hence in this case we can state~\eqref{eq:est:geom:side} also as 
\begin{multline}\label{eq:cor:to:est:geom:side}
\Big| \int_{t\Omega} J_{\text{geom}}(f^{\mu}_{\bxi})\,d\mu 
-\vol(G(\F)\backslash G(\A)^1)\int_{t\Omega}f_{\C}^{\mu}(1)\,d\mu \sum_{z\in Z(\F)} \tau_{\bxi}(z)\Big|\\
\leq c_1  \prod_{v\in S_{\bkappa}} q_v^{\xi_{vn}} \big(\deg \tau_{\bxi}\big)^{c_2} \Pi_{S,\bkappa, \ooo}^{c_3} ~t^{d-1} (\log t)^{n+1}.
\end{multline}

We finally state an estimate for the integral over $\mu$ of the geometric side for a characteristic function $\tau$ of a general compact set $\Xi\subseteq G(\A_f)$. 

\begin{cor}\label{cor:estimate:geom:general:hecke:fct}
There exist constants $c_1, c_2, c_3>0$ such that the following holds: Let $K_f\subseteq \cpt_f$ be a subgroup of finite index, and $\Xi\subseteq G(\A_f)$ a compact subset which is bi-$K_f$-invariant. 
Then for $\tau\in\HHH(K_f)$ the characteristic function of $\Xi$ normalised by $\vol(K_f)^{-1}$ we have
\begin{multline}\label{eq:estimate:geom:general:hecke:fct}
\Big| \int_{t\Omega} J_{\text{geom}}(f_{\C}^{\mu}\cdot\tau)\,d\mu 
-\vol(G(\F)\backslash G(\A)^1)\int_{t\Omega}f_{\C}^{\mu}(1)\,d\mu \sum_{z\in Z(\F)} \tau(z)\Big|\\
\leq c_1 \bigg(\frac{\vol(\Xi)}{\vol(\cpt_f)}\bigg)^{c_2} [\cpt_f:K_f]^{c_3} ~t^{d-1} (\log t)^{n+1}
\end{multline}
 for all $t\geq2$. 
We further have
\[
 \Big|J_{\text{geom}}(f_{\C}^{t, \mu}\cdot \tau)\Big|
\leq c_1 \bigg(\frac{\vol(\Xi)}{\vol(\cpt_f)}\bigg)^{c_2} [\cpt_f:K_f]^{c_3} ~t^{n-1} \hat\beta(t, \mu) 
\]
for all $t\geq1$ and $\mu\in i\aaa^*$.
\end{cor}

\begin{proof}
It suffices to consider $\tau=\chi_{K_f, g}$ for some $g\in G(\A_f)$ (recall the definition of $\chi_{K_f, g}$ from Section~\ref{subsec:general:hecke:alg}).
Let $\bxi=(\xi_v)_v$ be such that $g_v\in \cpt_v\varpi_v^{\xi_v} \cpt_v$ for every non-archimedean place $v$ of $\F$, and put $\kappa_v=\|\xi_v\|_W$.
Recall the definition of the set of integers $P_v$, $v<\infty$, from Section~\ref{sec:notation} which depends only on $\F$ and satisfies $P_v=0$ for almost all $v$. 
Since the left hand side of~\eqref{eq:estimate:geom:general:hecke:fct} stays unchanged if we replace the test function $f:=f_{\C}^{\mu}\cdot\tau$ by $f'(x):=f(zx)$, $x\in G(\A)$, for any $z\in Z(\F)$, we may assume by definition of the $P_v$'s that for every non-archimedean place $v$ the smallest entry of the vector $\xi_v$ is contained in the interval $[0,P_v]$.  Let $S_{K_f}$ be the set of all non-archimedean places $v$ with $K_f\cap\cpt_v\neq\cpt_v$, and put $S=S_{K_f}\cup S^{\bkappa,\ooo}$. Then clearly for any $\ooo\in\OOO_{R,\bkappa}$,
\[
\Pi_{S,\bkappa,\ooo}= \prod_{v\in S\minus S^{\bkappa,\ooo}}q_v
\leq \prod_{v\in S_{K_f}} q_v
\leq [\cpt_f:K_f].
\]
Hence by~\eqref{eq:cor:to:est:geom:side} and Remark~\ref{rem:geom:side:est:general} there are constants $c_1, c_2>0$ depending only on $n$, $\F$, $h$, and $\Omega$ such that for all $t\geq2$ we have
\begin{multline*}
\vol(K_f) \Big| \int_{t\Omega} J_{\text{geom}}(f_{\C}^{\mu}\cdot\tau)\,d\mu 
-\vol(G(\F)\backslash G(\A)^1)\int_{t\Omega}f_{\C}^{\mu}(1)\,d\mu \sum_{z\in Z(\F)} \tau(z)\Big|\\
\leq c_1 \big(\deg \tau_{\bxi}\big)^{c_2} [\cpt_f:K_f]^{c_2} ~t^{d-1} (\log t)^{n+1}.
\end{multline*}
Now by Section~\ref{subsec:general:hecke:alg} we have 
\[
\deg\tau_{\bxi}
\leq [\cpt_f:K_f] \|\tau\|_{L^1(G(\A_f))}
= [\cpt_f:K_f]^2 \frac{\vol(\Xi)}{\vol(\cpt_f)}
\]
which proves~\eqref{eq:estimate:geom:general:hecke:fct}. The second estimate follows similarly by using~\eqref{eq:coarse:est:geom:side} instead of~\eqref{eq:est:geom:side}.
\end{proof}

\section{The spectral side of the trace formula}\label{sec:spec:side}
In this and the next section we study the integral over $\mu\in t\Omega$ of $J_{\text{spec}}(f_{\C}^{\mu}\cdot\tau)$ as $t\rightarrow\infty$. The strategy for this will be similar to the analysis in~\cite{LaMu09} but we need to be more careful about the dependence of the bounds for the error term on our test function $\tau$. The basic idea is to show that the contribution of the discrete part of the spectrum essentially equals the left hand side of~\eqref{thm_main_eq}, and that the continuous spectrum only contributes to the error term in~\eqref{thm_main_error}. In this section we give some general properties of the spectral side of the trace formula and collect some auxiliary results, before proving the main result for the spectral side in Section~\ref{sec:spectral:est} below.

\subsection{Automorphic representations}
We use the following notation:
\begin{itemize}
 \item $W^M$ is the Weyl group of $M$ with respect to $T_0$ if $M\in\LLL$, $W=W^G$.
\item If $L\in\LLL(M)$,  $W^L(M)$ is the set of all $w\in W^L$ such that $w$ induces an isomorphism $\aaa_M\longrightarrow \aaa_M$, and $W^L(M)_{\text{reg}}:=\{w\in W^L(M)\mid \ker w=\aaa_L\}$.
\item $\Pi_{\text{disc}}(M(\A)^1)$ (resp.\ $\Pi_{\text{cusp}}(M(\A)^1)$) is the set of irreducible unitary representations occurring in the discrete (resp.\ cuspidal) part of $L^2(A_MM(\F)\backslash M(\A))$.
\item $\AAA^2(P)$ is the space of all
$\varphi: U(\A)M(\F)\backslash G(\A)^1\longrightarrow \C$
 such that 
$\varphi_x\in L^2(A_MM(\F)\backslash M(\A))$  for all $x\in G(\A)$, where
$\varphi_x(g)=\delta_P(g)^{-\frac{1}{2}}\varphi(gx)$ for $g\in M(\A)$.

\item $\bar\AAA^2(P)$ denotes the Hilbert space completion of $\AAA^2(P)$.
\end{itemize}

For $\lambda\in \aaa_{P,\C}^*$ we have the induced representation $\rho(P,\lambda,\cdot)$ of $G(\A)$ on $\AAA^2(P)$ given by
\[
 (\rho(P,\lambda, y)\varphi)(x)=\varphi(xy)e^{\langle\lambda+\rho_P, H_P(xy)\rangle}e^{-\langle\lambda+\rho_P,H_P(x)\rangle}.
\]
Thus for sufficiently nice functions $f:G(\A)\longrightarrow \C$ we get an operator $\rho(P,\lambda, f):\AAA^2(P)\longrightarrow~\AAA^2(P)$. 

\begin{itemize}
 \item For $\pi\in\Pi_{\text{disc}}(M(\A)^1)$ let $\AAA_{\pi}^2(P)\subseteq \AAA^2(P)$ be the space of all $\varphi$ such that $\varphi_x $ transforms according to $\pi$ for every $x\in G(\A)$. 
\item $\AAA^2(P)^{K}$ (resp.\ $\AAA^2(P)^{K_f}$, resp.\ $\AAA^2(P)^{\cpt_{\infty}}$)  is the space of $K$-invariant (resp.\ $K_f$-invariant, resp.\ $\cpt_{\infty}$-invariant) vectors in $\bar\AAA^2(P)$.
\item $\Pi_{\cpt_{\infty}}:\bar\AAA^2_{\pi}(P)\longrightarrow \AAA_{\pi}^2(P)^{\cpt_{\infty}}$ (resp.\ $\Pi_{K}:\bar\AAA_{\pi}^2(P)\longrightarrow \AAA^2(P)^{K}$)
is the orthogonal projection onto the space of $\cpt_{\infty}$-invariant (resp.\ $K$-invariant) vectors.
\item Each $\pi\in\Pi_{\text{disc}}(M(\A)^1)$ factorises as a product $\prod_v \pi_v$ of irreducible unitary representations of $G(\F_v)$, and $\lambda_{\pi_{\infty}}\in\aaa_{M, \C}^*/ W_M$ denotes the infinitesimal character of $\pi_{\infty}$.
\item $\HHH_{\pi}$ (resp.\ $\HHH_{\pi_{\infty}}$) is the representation space of $\pi$ (resp.\ $\pi_{\infty}$).

\item $\pi^{\cpt_{\infty}}$ (resp.\ $\pi_{\infty}^{\cpt_{\infty}}$) denotes the restriction of $\pi$ (resp.\ $\pi_{\infty}$) to the space $\HHH_\pi^{\cpt_{\infty}}$ (resp.\ $\HHH_{\pi_{\infty}}^{\cpt_{\infty}}$) of $\cpt_{\infty}$ fixed vectors in $\HHH_{\pi}$ (resp.\ $\HHH_{\pi_{\infty}}$).
\end{itemize}

Let $\tau\in\HHH(K_f)$.
If $\pi\in \Pi_{\text{disc}}(M(\A)^1)$ and $\rho_{\pi}(P, \lambda, \cdot)$ denotes the restriction of $\rho(P, \lambda, \cdot)$ to $\AAA^2_{\pi}(P)$, then by definition of $f_{\C}^{\mu}$ we have 
\begin{equation}\label{ind_rep_on_pi}
 \rho_{\pi}(P, \lambda, f_{\C}^{\mu}\cdot \tau)
= \hat{h}(\lambda -\mu +\lambda_{\pi_{\infty}}) \rho_{\pi}(P, \tau) \Pi_{\cpt_{\infty}}
\end{equation}
where
$\rho(P, \tau): \AAA^2(P)\longrightarrow \AAA^2(P)^{K_f}$
is defined as usual by
\[
\big(\rho(P, \tau)\varphi\big)(y)
= \int_{G(\A_f)} \tau(x) \varphi(yx) \,dx,
\]
and $\rho_{\pi}(P,\tau)$ is the restriction of $\rho(P,\tau)$ to $\AAA_{\pi}^2(P)$. 
In particular, if $P=G$ and $\varphi\in\AAA^2_{\pi}(G)$, then $\lambda=0$ and 
\[
\big(\rho_{\pi}(G, f^{\mu}_{\C}\cdot \tau)\varphi\big)(y)
:=\big(\rho_{\pi}(G, 0, f^{\mu}_{\C}\cdot\tau)\varphi\big)(y)
 = \hat{h}(-\mu+\lambda_{\pi_{\infty}}) \int_{G(\A_f)} \tau(x)\varphi(yx)\,dx
\]
If $\pi\in\Pi_{\text{disc}}(G(\A)^1)$, we write $\pi(\tau):=\rho_{\pi}(G,\tau)$  so that 
\[
\tr \rho_{\pi}(G, f^{\mu}_{\C}\cdot \tau) 
=
\begin{cases}
 \hat{h}(-\mu + \lambda_{\pi_{\infty}}) \tr \pi^{\cpt_{\infty}}(\tau)				&\text{if } \pi_{\infty}\text{ is }\cpt_{\infty}\text{-invariant},\\
0												&\text{else.}
\end{cases}
\]
Moreover, for every $P$ and $\lambda$ we get 
\begin{equation}\label{trace_triv_upper_bound}
\big|\tr \rho_{\pi}(P, \lambda, f^{\mu}_{\C}\cdot \tau) \big|
\leq \big|\hat{h}(\lambda-\mu+\lambda_{\pi_{\infty}})\big|  \|\tau\|_{L^1(G(\A_f))} \dim \AAA_{\pi}^2(P)^K.
\end{equation}

\begin{example}
 If $\tau=\vol(K_f)^{-1}\chi_{K_f}$, then $\rho(P, \tau)$ is the projection of $\AAA^2(P)$ onto the subspace of $K_f$-fixed vectors $\AAA^2(P)^{K_f}$. In particular, 
$\rho_{\pi}(P, \lambda, f^{\mu}_{\C}\cdot \tau) = \hat{h}(\lambda-\mu+\lambda_{\pi_{\infty}}) \Pi_{\cpt_{\infty}} \Pi_K$.
\end{example}

\subsection{The spectral side of the trace formula}
By~\cite[Theorem 4]{FiLaMu_limitmult} (cf.~\cite{FiLaMu}) the spectral side of Arthur's trace formula can be written as
\[
 J_{\text{spec}}(f)=\sum_{[M]} J_{\text{spec},M}(f),  
\]
 where
\begin{equation}\label{eq:defspecdist}
 J_{\text{spec},M}(f)
=\frac{1}{|W(\aaa_M)|} \sum_{\sigma\in W(\aaa_M)} \iota_{\sigma}\sum_{\underline{\beta}\in\mathfrak{B}_{P,L_{\sigma}}}\int_{i\left(\aaa_{L_{\sigma}}^G\right)^*} \tr\left(\Delta_{\mathcal{X}_{L_{\sigma}}(\underline{\beta})}(P,\lambda)M(P,\sigma)\rho(P,\lambda,f)\right) \,d\lambda
\end{equation}
and this last integral is absolutely convergent with respect to the trace norm for every test function $f\in C_c^{\infty}(G(\A)^1)$ (actually, this holds for a larger class of test functions as shown in~\cite{FiLaMu}). 
Here the notation is as follows (see~\cite{FiLaMu,FiLaMu_limitmult} for details): 
\begin{itemize}
 \item $[M]$ runs over conjugacy classes of Levi subgroups of $G$,
\item  $M\in[M]\cap\LLL$ is a representative of the class $[M]$,
\item  $P\in\PPP(M)$ is an arbitrary parabolic subgroup with Levi component $M$,
\item $W(\aaa_M)$ is the set of all linear isomorphisms $\aaa_M\longrightarrow\aaa_M$ obtained from restricting  an element of $W$ to $\aaa_M$,
\item if $\sigma\in W(\aaa_M)$, then $L_{\sigma}\in \LLL $ denotes the smallest Levi subgroup containing $\sigma$,
\item $\mathfrak{B}_{P,L_{\sigma}}$ is a certain set of ordered subsets of $\Sigma_P^{\vee}$,
\item $\mathcal{X}_{L_{\sigma}}$ associates with any $\underline{\beta}$ a certain tuple of parabolic subgroups in $\FFF(M)$,
\item  $\Delta_{\mathcal{X}_{L_{\sigma}}(\underline{\beta})}(P,\lambda)$ is a certain operator essentially being the consecutive application of rank one intertwining operators and their logarithmic derivatives.
\end{itemize}

 In particular, only such tuples $\underline{\beta}=(\beta_1^{\vee}, \ldots, \beta_r^{\vee})$ appear in $\mathfrak{B}_{P, L_{\sigma}}$ for which the projection of $\beta_1^{\vee}, \ldots, \beta_r^{\vee}$ onto $\aaa_L^G$ span a lattice of full rank. 
Note that the operator $M(P,\sigma)$ is unitary for any $\sigma\in W(\aaa_M)$ and commutes with $\rho(P,\lambda, f)$.

\begin{rem}
 If $[M]=[G]$, then $J_{\text{spec},G}(f)=:J_{\text{disc}}(f)=\sum_{\pi\in \Pi_{\text{disc}}(G(\A)^1)} \tr \pi(f)$ is the discrete part of the trace formula.
\end{rem}

\subsection{Intertwining operators}
We need to better understand the operator $\Delta_{\mathcal{X}_{L_{\sigma}}(\underline{\beta})}(P,\lambda)$ appearing in~\eqref{eq:defspecdist} and follow~\cite{FiLaMu_limitmult} for this. Let $\Delta_{\mathcal{X}_{L_{\sigma}}(\underline{\beta})}(P,\lambda)_{\AAA_{\pi}^2(P)^K}$ denote the restriction of $\Delta_{\mathcal{X}_{L_{\sigma}}(\underline{\beta})}(P,\lambda)$ to $\AAA_{\pi}^2(P)^K$.
For $\mu\in  i\aaa^*$ we let $B(\mu)$ denote the ball of radius $1$ around $\mu\in\aaa_{\C}^*$.
Suppose that $M\in\LLL$ and $L\in\LLL(M)$ are Levi subgroups, and $P\in\PPP(M)$. 
Then $\big(\aaa_{L}^G\big)_{\C}^*\subseteq \aaa_{\C}^*$. 
For $\pi\in\Pi_{\text{disc}}(M(\A)^1)$ recall the definition of the constant $\Lambda_{\pi_{\infty}}$ from~\cite[(4.5)]{Mu02}; it depends on the Casimir eigenvalues of $\pi$ and the minimal $\cpt_{\infty}$-types of the induced representation $\Ind_{P(\F_{\infty})}^{G(\F_{\infty})}\pi_{\infty}$.
Our goal in this section is to show the following estimate:

\begin{lemma}\label{lemma:est:int:ball}
For all $\pi\in\Pi_{\text{disc}}(M(\A)^1)$ and $\mu$ as before,
\begin{equation}\label{eq:est_intop_ball}
 \int_{B(\mu)\cap i\big(\aaa_L^G\big)^*} \big\|\Delta_{\mathcal{X}}(P,\lambda)\big|_{\AAA^{2}_{\pi}(P)^K}\big\|\,d\lambda
\leq c \big( 1+\log[\cpt: K] +\log (1+\|\mu\|) + \log (1+\Lambda_{\pi_{\infty}}) \big)^{2r_L}
\end{equation}
where $c>0$ is a constant depending only on $n$ and $\F$.
Here $\Delta_{\mathcal{X}}(P,\lambda)\big|_{\AAA^{2}_{\pi}(P)^K}$ denotes the restriction of $\Delta_{\mathcal{X}}(P,\lambda)$ to $\AAA^2_{\pi}(P)^K$, and $r_L=\dim\aaa_L^G$.
\end{lemma}

We first note that it suffices to consider the case that $K_f=\prod_{v<\infty} K_v$ is a direct product of finite-index subgroups $K_v\subseteq\cpt_v$: If $\nnn\subseteq \OOO_{\F}$ is an ideal, let $K_f(\nnn)\subseteq \cpt_f$ denote the principal congruence subgroup of level $\nnn$. The largest ideal $\nnn$ with $K_f(\nnn)\subseteq K_f$ is called the level $\lev(K_f)$ of $K_f$. Then by~\cite[Lemma 1.6]{Lu95} (cf.\ also~\cite[Remark 4]{FiLaMu_limitmult}) there exists $c_1>0$ depending only on $n$ and $\F$ such that $\N_{\F/\Q}(\lev(K_f))\leq [\cpt_f:K_f]^{c_1}$. On the other hand, there are $a, b>0$ depending only on $n$ and $\F$ such that 
$\N_{\F/\Q}(\nnn)\leq [\cpt_f:K_f(\nnn)]^a\leq \N_{\F/\Q}(\nnn)^b$ for $\nnn=\lev(K_f)$. Hence
\[
 [\cpt_f:K_f(\nnn)]^{c_2}\leq [\cpt_f:K_f]\leq [\cpt_f:K_f(\nnn)]^{c_3} 
\]
for suitable $c_2, c_3>0$ depending only on $n$ and $\F$ so that we may replace $K_f$ by $K_f(\nnn)$.
Hence we will assume from now on that $K_f=\prod_{v<\infty}K_v$ is a direct product of local compact subgroups $K_v\subseteq \cpt_v$.

To prove the above lemma, we break up the operator $\Delta_{\mathcal{X}}(P,\lambda)$ into smaller pieces.
We freely use the notation from~\cite{FiLaMu_limitmult}. Recall in particular the definition of the intertwining operator 
\[ 
M_{Q|P}(\lambda):\AAA^2(P)\longrightarrow\AAA^2(Q)
\]
 for $P,Q\in\PPP(M)$. If $P$ and $Q$ are adjacent along the root $\alpha\in\Sigma_P$, the rank one intertwining operator $M_{Q|P}(\lambda):\AAA^2(P)\longrightarrow\AAA^2(Q)$, $\lambda\in \aaa_{M,\C}^*$, is in fact a meromorphic operator-valued function of the scalar variable $\langle\lambda,\alpha^{\vee}\rangle\in\C$. We then also write $M_{Q|P}(\langle\lambda, \alpha^{\vee}\rangle)=M_{Q|P}(\lambda)$, and can compute the derivative $M_{Q|P}'(\langle\lambda, \alpha^{\vee}\rangle)$ as the ordinary derivative of a function in one complex variable. We define the ``logarithmic derivative'' of $M_{Q|P}$ by 
\[
 \delta_{Q|P}(\lambda):=M_{Q|P}(\langle\lambda, \alpha^{\vee}\rangle)^{-1} M_{Q|P}'(\langle\lambda, \alpha^{\vee}\rangle): \AAA^2(P)\longrightarrow\AAA^2(P).
\]
Fix $\sigma\in W(\aaa_M)$ and $\underline{\beta}=(\beta_1^{\vee}, \ldots, \beta_{r_L}^{\vee})\in\mathfrak{B}_{P,L_{\sigma}}$, and write $L=L_{\sigma}$ and $\mathcal{X}=\mathcal{X}_{L_{\sigma}}(\underline{\beta})$. Let $P_1,P_1', \ldots, P_{r_L}, P_{r_L}'\in \PPP(M)$ be the tuple associated with $\mathcal{X}$ so that $P_1|^{\beta_1}P_1', \ldots, P_{r_L}|^{\beta_{r_L}}P_{r_L}'$.
Then by definition~\cite[p. 179]{FiLaMu} (cf.\ also~\cite[\S 3.3]{FiLaMu_limitmult}) the operator $\Delta_{\mathcal{X}}(P,\lambda)$ equals
\[
 \frac{\vol(\underline{\beta})}{r_L!} M_{P_1'|P}(\lambda)^{-1}\delta_{P_1|P_1'}(\lambda)M_{P_1'|P_2'}(\lambda)\cdot\ldots\cdot \delta_{P_{r_L-1}|P_{r_L-1}'}(\lambda)M_{P_{r_L-1}'|P_{r_L}'}(\lambda) \delta_{P_{r_L}|P_{r_L}'}(\lambda) M_{P_{r_L}'|P}(\lambda),
\]
where $\vol(\underline{\beta})$ denotes the covolume of the lattice in $\aaa_L^G$ spanned by the projection of the coroots $\beta_1^{\vee}, \ldots, \beta_{r_L}^{\vee}$ onto $\aaa_L^G$. 
Note that the operators $M_{Q|P}(\lambda)$ are unitary for all $P,Q\in\PPP(M)$ and $\lambda\in i\big(\aaa_L^G\big)^*$. Hence to bound the operator norm of $\Delta_{\mathcal{X}}(P,\lambda)$, it will suffice to bound the operator norm of $\delta_{Q|P}(\lambda)$ for every pair of adjacent $Q,P\in\PPP(M)$.

To that end, suppose that $Q, P\in\PPP(M)$ are adjacent along $\alpha$, and let $\pi\in\Pi_{\text{disc}}(M(\A)^1)$. Let $s\in\C$ and denote by $M_{Q|P}(\pi,s)$ the restriction of $M_{Q|P}(s)$ to $\AAA_{\pi}^2(P)$.
The operator $M_{Q|P}(\pi,s)$ can be normalised by a normalising factor $n_{\alpha}(\pi, s)$. For $G=\GL_n$ it is given by a quotient of certain global Rankin-Selberg $L$-function.
Let $K^M_f\subseteq \cpt_f^M=\cpt_f\cap M(\A_f)$ be an open-compact subgroup of $M(\A_f)$. 
Then  by~\cite[Proposition 1]{FiLaMu_limitmult}
\begin{equation}\label{eq:norm_factor}
 \int_{T}^{T+1} \frac{n_{\alpha}'(\pi, is)}{n_{\alpha}(\pi,is)} \,ds 
\ll \big(1+\log(1+|T|) + \log(1+\Lambda_{\pi_{\infty}}) +\log [\cpt^M_f: K^M_f]\big)
\end{equation}
for all $T\in\R$ and $\pi\in\Pi_{\text{disc}}(M(\A)^1)$ having a $\cpt_{\infty}^M K^M_f$-invariant vector.
Using the canonical isomorphism of $G(\A_f)\times (\ggG_{\C}, \cpt_{\infty})$-modules 
\[
\AAA_{\pi}^2(P)\xrightarrow{\;\;\;\simeq\;\;\;} \Hom(\pi, L^2(M(\Q)\backslash M(\A)^1)\otimes \Ind_{P(\A)}^{G(\A)} \pi,
\]
the image of the quotient $n_{\alpha}(\pi,s)^{-1} M_{Q|P}(\pi,s)$ equals the normalised intertwining operators 
\[
\id\otimes R_{Q|P}(\pi,s)=\id\otimes\bigotimes_{v} R_{Q|P}(\pi_v,s)
\]
 with $R_{Q|P}(\pi_v, s)$ being defined on the local space $\Ind_{P(\F_v)}^{G(\F_v)} (\pi_v)$. In particular, $R_{Q|P}(\pi_v,s)\varphi_v=\varphi_v$ if $\varphi_v\in\big(\Ind_{P(\F_v)}^{G(\F_v)} (\pi_v)\big)^{\cpt_v}$ and $v$ is non-archimedean so that this product of operators is in fact a finite product over the places $v$ with $K_v\neq\cpt_v$ when we restrict to $K$-invariant vectors (cf.~\cite[Theorem 2.1]{Ar89a}).
We denote by $R_{Q|P}(\pi_v,s)^{-1} R_{Q|P}'(\pi_v,s)\big|_{K_v}$ the restriction of $R_{Q|P}(\pi_v,s)^{-1} R_{Q|P}'(\pi_v,s)$ to the $K_v$-invariant vectors $\big(\Ind_{P(\F_v)}^{G(\F_v)}(\pi_v)\big)^{K_v}$, and use similar notations in the global setting.

Thus, if $Q,P$ are adjacent along $\alpha$, the operator $\delta_{P|Q}(\lambda)$  restricted to $\AAA_{\pi}^2(P)^K$ equals
\begin{align*}
 \delta_{P|Q}(is)\big|_{\AAA_{\pi}^2(P)^K}
& = \frac{n_{\alpha}'(\pi, is)}{n_{\alpha}(\pi, is)} +  R_{Q|P}(\pi,is)^{-1} R_{Q|P}'(\pi,is)\big|_{K}\\
& = \frac{n_{\alpha}'(\pi, is)}{n_{\alpha}(\pi, is)} + \sum_{v\leq \infty} R_{Q|P}(\pi_v,is)^{-1} R_{Q|P}'(\pi_v,is)\big|_{K_v}.
\end{align*}

\begin{lemma}\label{lem:est_norm_intop}
 There exists a constant $c>0$ depending only on $n$ and $\F$ such that for all $\pi\in\Pi_{\text{disc}}(M(\A)^1)$ and all $T\in\R$ we have
\[
 \int_{T}^{T+1} \big\|R_{Q|P}(\pi_v,is)^{-1} R_{Q|P}'(\pi_v,is)\big|_{K_v}\big\| \,ds
\leq c\log[\cpt_v:K_v]
\]
if $v$ is non-archimedean, and, if $v$ is archimedean,
\[
 \int_{T}^{T+1} \big\|R_{Q|P}(\pi_v,is)^{-1} R_{Q|P}'(\pi_v,is)\big|_{\cpt_v}\big\| \,ds
\leq c.
\]
Hence, 
\[
 \int_{T}^{T+1} \big\|R_{Q|P}(\pi,is)^{-1} R_{Q|P}'(\pi,is)\big|_{K}\big\| \,ds
\leq c\big(1+\log[\cpt:K]\big)^2.
\]
\end{lemma}

\begin{proof}
First note that $R_{Q|P}(\pi_v,is)$ is unitary for every $s\in\R$ and every $v$. Hence the assertion in the archimedean case follows from~\cite[Proposition 0.4]{MuSp04}, where the norm of $R_{Q|P}'(\pi_v,is)$ restricted to the space of $\cpt_v$-invariant vectors is bounded uniformly in $s$ and $\pi$.

If $v$ is non-archimedean, $R_{Q|P}(\pi_v,is)$ is moreover $\frac{2\pi}{\log q_v}$-periodic so that
\begin{align*}
 \int_{T}^{T+1} \big\|R_{Q|P}(\pi_v,is)^{-1} R_{Q|P}'(\pi_v,is)\big|_{K_v}\big\| \,ds
& \leq \int_{T}^{T+ \frac{2\pi}{\log q_v} a(q_v)} \big\|R_{Q|P}'(\pi_v,is)\big|_{K_v}\big\|\,ds\\
& = a(q_v) \int_{T}^{T+ \frac{2\pi}{\log q_v} } \big\|R_{Q|P}'(\pi_v,is)\big|_{K_v}\big\|\,ds
\end{align*}
where $a(q_v)=\lceil\log q_v\rceil$ denotes the smallest integer $\geq \log q_v$, and $R_{Q|P}'(\pi_v,is)\big|_{K_v}$ the restriction of $R_{Q|P}'(\pi_v,is)$ to the $K_v$-invariant vectors in $\Ind_{P(\F_v)}^{G(\F_v)}\pi_v$. By~\cite[\S 4.2]{FiLaMu_limitmult} this last expression is bounded by 
\[
 a(q_v) c \log_{q_v}[\cpt_v:K_v]
= c \frac{a(q_v)}{\log q_v} \log[\cpt_v: K_v]
\leq 2 c \log[\cpt_v:K_v]
\]
for $c>0$ a constant depending only on $n$ and $\F$.
The last assertion then follows from 
\[R_{Q|P}(\pi,it)^{-1} R_{Q|P}'(\pi,it)\big|_{K}= \sum_{v\leq\infty} R_{Q|P}(\pi_v,it)^{-1} R_{Q|P}'(\pi_v,it)\big|_{K_v},\] in which at most $1+\log [\cpt:K] /\log 2$ terms are non-zero together with the previous estimates.
\end{proof}

We can now finish the proof  of Lemma~\ref{lemma:est:int:ball}:
Note that for $K_f^M=K_f\cap M(\A_f)$ we can bound $[\cpt_f^M: K_f^M]$ by $c_1[\cpt_f:K_f]^{c_2}$ for suitable $c_1, c_2\geq0$ depending only on $n$ and $\F$: This follows from the fact that we can compare  the index of $[\cpt_f^M:K_f^M]$ (resp.\ $[\cpt_f:K_f]$) and with powers of the level of the respective subgroups as explained above. 
 With this, the lemma is a direct consequence of the definition of $\Delta_{\mathcal{X}}(P,\lambda)$, the estimate~\eqref{eq:norm_factor}, and Lemma~\ref{lem:est_norm_intop}.

\section{Spectral estimates}\label{sec:spectral:est}
By Corollary~\ref{cor:estimate:geom:general:hecke:fct} there exist constants $c_1, c_2>0$ depending only on $n$, $\F$, and $h$ such that for every compact subgroup $K_f\subseteq\cpt_f$ of finite index the following holds:
For all $t\geq1$ and $\mu\in i\aaa^*$ we have
\begin{equation}\label{eq:geom:hypothesis:for:spec:est}
 |J_{\text{geom}}(f_{\C}^{t,\mu}\cdot \chi_{K_f}) |
\leq c_1 [\cpt_f:K_f]^{c_2} t^{r} \hat{\beta}(t,\mu),
\end{equation}
where $f_{\C}^{t,\mu}$ is associated with $h_{t,\mu}$ as usual, $\chi_{K_f}=\chi_{K_f,1}$ is the characteristic function of $K_f$ normalised by $\vol(K_f)^{-1}$, and $\hat\beta(t, \mu)$ is defined in~\eqref{eq:def:of:tildebeta}.
Hence the trace formula tells us that $|J_{\text{spec}}(f_{\C}^{t,\mu}\cdot\chi_{K_f})|$ can also be estimated by the right hand side of~\eqref{eq:geom:hypothesis:for:spec:est}.

The goal of this section is to show the following analogue of~\cite[Corollary 4.4]{LaMu09}.

\begin{proposition}\label{prop:growth_spec}
There exist constants $c_1, c_2, c_3>0$ depending only on $n$, $\F$, and $h$ such that the following holds: Let $\Xi\subseteq G(\A_f)$ be a compact bi-$K_f$-invariant set, and $\tau\in\HHH(K_f)$ the characteristic function of $\Xi$ normalised by $ \vol(K_f)^{-1}$. Then for every $t\geq1$, we have
\begin{equation}\label{eq:growth_disc}
| J_{\text{disc}}(f_{\C}^{t,\mu}\cdot \tau)|
\leq c_1 [\cpt_f:K_f]^{c_2} \bigg(\frac{\vol(\Xi)}{\vol(\cpt_f)}\bigg)^{c_3} ~ t^{n-1} \hat{\beta}(t,\mu),  
\end{equation}
Moreover, there exists $\tilde c_1, \tilde c_2, \tilde c_3>0$ depending only on $n$, $\F$, $h$, and $\Omega$ such that
\begin{equation}\label{eq:growth_nondisc}
 \bigg|\int_{t\Omega} J_{\text{spec}}(f_{\C}^{\mu}\cdot \tau) \,d\mu -\int_{t\Omega} J_{\text{disc}}(f_{\C}^{\mu}\cdot \tau) \,d\mu\bigg|\\
\leq \tilde c_1 [\cpt_f:K_f]^{\tilde c_2} \bigg(\frac{\vol(\Xi)}{\vol(\cpt_f)}\bigg)^{\tilde c_3} ~ t^{d-2} (\log t)^n
\end{equation}
as $t\rightarrow\infty$.
\end{proposition}

We first need an auxiliary result. 
For $t\geq1$ denote by $B_t(\mu)\subseteq \aaa_{\C}^*$ the ball of radius $t$ around $\mu\in\aaa_{\C}^*$ (so $B(\mu)=B_1(\mu)$ in our previous notation). Let $K=\cpt_{\infty} K_f$ and set
\begin{equation}\label{eq:def:spectral:mult}
 m_K(B_t(\mu))= \sum_{\substack{\pi\in\Pi_{\text{disc}}(G(\A)^1) \\ \lambda_{\pi_{\infty}}\in B_t(\mu) }} \dim \HHH_{\pi}^K.
\end{equation}

\begin{lemma}\label{lem:spec_meas_of_ball}
There exist constants $c_1, c_2>0$ depending only on $n$ and $\F$ (but not on $K_f$) such that for all $\mu\in i\aaa^*$ we have
\begin{equation}\label{eq:spec:meas:ball2}
 m_K(B_t(\mu))\leq c_1 [\cpt:K]^{c_2}  ~t^r\hat{\beta}(t,\mu)
\end{equation}
and
\begin{equation}\label{eq:spec:meas:ball1}
 m_K(\{\nu\in i\aaa^*\mid \|\Im\nu-\mu\|\leq t\})
\leq c_1[\cpt:K]^{c_2}  ~t^r\hat{\beta}(t,\mu)
\end{equation}
for all $t\geq1$ and all $K_f\subseteq \cpt_f$ of finite index for which~\eqref{eq:growth_disc} of Proposition~\ref{prop:growth_spec} holds with $\tau=\chi_{K_f}$.
\end{lemma}

\begin{proof}
 Without the explicit dependence on $[\cpt:K]$, the assertion is contained in~\cite[Proposition 4.5]{LaMu09}. Note that we only need to show~\eqref{eq:spec:meas:ball2}, since if $\nu=\lambda_{\pi_{\infty}}$ for some $\pi\in \Pi_{\text{disc}}(G(\A)^1)$, then $\|\Re \lambda_{\pi_{\infty}}\|\leq \|\rho\|$ by~\cite[Proposition 3.4]{DKV79}. Suppose that $M\in\LLL$ and $\mu\in i\big(\aaa_M^G\big)^*$. We prove~\eqref{eq:spec:meas:ball2} by induction on $\dim\big(\aaa_M^G\big)^*$. 

By definition of $\chi_{K_f}$ we have  $\tr \rho_{\pi}(P,\chi_{K_f})=\dim\AAA_{\pi}^2(P)^K=\dim\HHH_{\pi}^K$ because of multiplicity one for $\GL_n$.
Suppose $M=G$ so that $\mu=0$ because of $\big(\aaa_G^G\big)^*=\{0\}$. As explained in~\cite[\S 6]{DKV79} (cf.~\cite[p. 135]{LaMu09}), we can choose $h\in C_c^{\infty}(\aaa)^W$ such that $|\hat{h}|\geq 1$ on $B_1(0)$. Let $f_{\C}^{t,\mu}$ denote the function defined by $h_{t,\mu}$ as usual. Then 
$m_K(B_t(0))\leq J_{\text{disc}}(f_{\C}^{t,0}\cdot \chi_{K_f})$ and by~\eqref{eq:growth_disc} there are constants $c_1, c_2\geq0$ such that 
\[
\Big|J_{\text{disc}}(f_{\C}^{t,0}\cdot \chi_{K_f})\Big|
\leq c_1 [\cpt:K]^{c_2} ~t^r \hat{\beta}(t,\mu=0)
\]
 which is the assertion for $M=G$.  
Using this as the initial step of the induction, the induction step can be carried out exactly as explained in~\cite[pp. 135-136]{LaMu09} and we omit the details.
\end{proof}

\begin{proof}[Proof of Proposition~\ref{prop:growth_spec}]
We prove the proposition by induction on $n$.

If $n=1$, there is no continuous spectrum, and the trace formula reduces to the usual Poisson summation formula so that
\[
J_{\text{spec}}(f_{\C}^{0}\cdot\tau)
=J_{\text{disc},G}(f^{0}_{\C}\cdot \tau)
=\vol(\F^{\times}\backslash\A^1)\sum_{z\in\F^{\times}} \tau(z).
\]
Hence
\[
\Big|J_{\text{disc},G}(f^{0}_{\C}\cdot \tau)\Big|
\leq \vol(\F^{\times}\backslash\A^1)\sum_{z\in\F^{\times}} |\tau(z)|
=\vol(\F^{\times}\backslash\A^1) [\cpt:K] \frac{\vol(\Xi)}{\vol(\cpt_f)}.
\]
Now suppose that $n>1$ and assume that the proposition holds for all $\GL_m$ with $m<n$.
Then the assertion of the proposition and also Lemma~\ref{lem:spec_meas_of_ball} hold in particular for all semi-standard Levi subgroups properly contained in $G$. 
  It will suffice to bound
\[
\bigg|\int_{t\Omega}J_{\text{spec},M}(f^{\mu}_{\C}\cdot \tau)\,d\mu\bigg|
\]
as $t\rightarrow\infty$ for semi-standard Levi subgroups $M\in\LLL$ with $M\neq G$.
 Fix $\sigma\in W(\aaa_M)$ and $\underline{\beta}\in \mathfrak{B}_{P,L_{\sigma}}$, and write $L=L_{\sigma}$ and $\mathcal{X}=\mathcal{X}_{L}(\underline{\beta})$. By the last section, it suffices to estimate the integral over $\mu\in t\Omega$ of
\begin{align*}
 &\bigg|\int_{i\big(\aaa_L^G\big)^*}  \tr\big(\Delta_{\mathcal{X}}(P,\lambda) M(P,\sigma) \rho(P, \lambda, f^{\mu}_{\C}\cdot \tau)\big) \,d\lambda\bigg|\\
& \leq \sum_{\pi\in \Pi_{\text{disc}}(M(\A)^1)} \big|\tr\rho_{\pi}(P,\tau)\big| \int_{i\big(\aaa_L^G\big)^*}\big|\hat{h}(-\mu+\lambda+\lambda_{\pi_{\infty}})\big| 
\big\|\Delta_{\mathcal{X}}(P,\lambda)\big|_{\pi,K}\big\|  \,d\lambda\\
& \leq  \|\tau\|_{L^1(G(\A_f))} \sum_{\pi\in \Pi_{\text{disc}}(M(\A)^1)} \dim \AAA_{\pi}^2(P)^K \int_{i\big(\aaa_L^G\big)^*}\big|\hat{h}(-\mu+\lambda+\lambda_{\pi_{\infty}})\big| 
\big\|\Delta_{\mathcal{X}}(P,\lambda)\big|_{\pi,K}\big\|  \,d\lambda,
\end{align*}
where we used~\eqref{trace_triv_upper_bound} for the last inequality.
As in the proof of~\cite[Proposition 4.3]{LaMu09} we write $\aaa_M^L=\aaa_M\cap \aaa^L$ and let $\big(\aaa_M^L\big)^{\perp}=\big(\aaa_0^M\big)^*\oplus \big(\aaa_L^G\big)^*$ be the annihilator of $\aaa_M^L$ in $\aaa^*$. Choose a lattice $\Lambda\subseteq i\big(\aaa_M^L\big)^{\perp}$ such that $\bigcup_{\lambda\in\Lambda} \big(B_1(\lambda)\cap i\big(\aaa_M^L\big)^{\perp}\big)=i \big(\aaa_M^L\big)^{\perp}$.
Then
\begin{align*}
& \sum_{\pi\in \Pi_{\text{disc}}(M(\A)^1)} \dim \AAA_{\pi}^2(P)^K \int_{i\big(\aaa_L^G\big)^*}\big|\hat{h}(-\mu+\lambda+\lambda_{\pi_{\infty}})\big| 
\big\|\Delta_{\mathcal{X}}(P,\lambda)\big|_{\pi,K}\big\|  \,d\lambda           \\
& \leq \sum_{\nu\in \Lambda} \sup_{\lambda\in B_{1+\|\rho\|}(\nu)}\big|\hat{h}(-\mu+\lambda)\big|
\sum_{\substack{ \pi\in \Pi_{\text{disc}}(M(\A)^1): \\ \lambda_{\pi_{\infty}}\in B^M_1(\nu^M)}} \dim \AAA_{\pi}^2(P)^K
 \int_{B_1(\nu_L)\cap i\big(\aaa_L^G\big)^*} 
\big\|\Delta_{\mathcal{X}}(P,\lambda)\big|_{\pi,K}\big\|  \,d\lambda,
\end{align*}
where we write $\nu=\nu^M+\nu_L\in i\big(\aaa_0^M\big)^*\oplus i\big(\aaa_L^G\big)^*$.
The last integral can be bounded by Lemma~\ref{lemma:est:int:ball}, and the sum over the dimensions of the spaces of automorphic forms can be bounded as in~\cite[p. 151]{LaMu09} by using Lemma~\ref{lem:spec_meas_of_ball} and the induction hypothesis (with characteristic function of $K_f^M=K_f\cap M(\A_f)$ normalised by $\vol(K_f^M)^{-1}$ as the test function at the non-archimedean places). Altogether, we end up with
\[
c_1 [\cpt:K]^{c_2}\|\tau\|_{L^1(G(\A_f))}
 \sum_{\nu\in \Lambda} \big(1+\log(2+\|\nu\|)\big)^r~\hat{\beta}^M(\nu)
\sup_{\lambda\in B_{1+\|\rho\|}(\nu)}\big|\hat{h}(-\mu+\lambda)\big|.
\]
Since $\hat{h}\in \PPP(\aaa_{\C}^*)^W$, the last sum is bounded by a finite constant that depends only on $h$ and $n$.
Note that $\hat{\beta}^M(\nu)=O\big(1+\|\nu\|\big)^{d-r-2}$. Hence integrating over $\mu\in t\Omega$, we can find constants $c_1, c_2>0$ depending only on $n$, $h$, $\F$, and $\Omega$ such that
\[
 \bigg|\int_{t\Omega}J_{\text{spec},M}(f^{\mu}_{\C}\cdot \tau)\,d\mu\bigg|
\leq c_1~ [\cpt:K]^{c_2}\|\tau\|_{L^1(G(\A_f))}  ~ t^{d-2} \log^r (2+t)
\]
finishing the proof of~\eqref{eq:growth_nondisc}.
Using~\eqref{eq:geom:hypothesis:for:spec:est} the first part of the proposition follows as well.
\end{proof}

\section{Proof of Theorem \ref{thm:main}}\label{sec:finish:main:res}
We now bring together all our previous results and finish the proof of our main theorem.
For $\lambda\in\aaa_{\C}^*$ we define in generalisation of~\eqref{eq:def:spectral:mult}
\[
 m_K(\lambda)=\sum_{\substack{\pi\in\Pi_{\text{disc}}(G(\A)^1): \\ \lambda_{\pi_{\infty}}=\lambda}} \dim\HHH_{\pi}^{K},\;\;\;\text{ and }\;\;\;
m_K(\lambda, \tau)=\sum_{\substack{\pi\in\Pi_{\text{disc}}(G(\A)^1): \\ \lambda_{\pi_{\infty}}=\lambda}}  \tr\pi^{\cpt_{\infty}}(\tau)
\]
for $\tau\in\HHH(K_f)$ so that $m_K(\lambda)=m_K(\lambda, \chi_{K_f})$.
If $\Omega'\subseteq \aaa_{\C}^*$ is any set, we put $m_K(\Omega')=\sum_{\lambda\in\Omega'} m_K(\lambda)$ and $m_K(\Omega', \tau)=\sum_{\lambda\in\Omega'} m_K(\lambda, \tau)$ so that
\[
 m_{K}(t\Omega, \tau)
=\sum_{\substack{\pi\in\Pi_{\text{disc}}(G(\A)^1): \\ \lambda_{\pi_{\infty}}\in t\Omega}} \tr\pi^{\cpt_{\infty}}(\tau).
\]

\begin{rem}
 \begin{enumerate}[label=(\roman{*})]
  \item Note that $\dim\HHH_{\pi_{\infty}}^{\cpt_{\infty}}\leq 1$ for every $\pi\in \Pi_{\text{disc}}(G(\A)^1))$.

\item We have  $m_K(\lambda)=0=m_K(\lambda, \tau)$ if there does not exist any irreducible representation $\pi\subseteq L_{\text{disc}}^2(G(\F)\backslash G(\A)^1)^{K_f}$ with $\lambda_{\pi_{\infty}}=\lambda$.
In particular, $m_K(\lambda, \tau)=0$ except for a  $\lambda$ in a discrete set in $\aaa_{\C}^*$.
 \end{enumerate}
\end{rem}

Let $\Xi\subseteq G(\A_f)$ be a compact, bi-$K_f$-invariant set, and let $\tau$ be the characteristic function of $\Xi$ normalised by $\vol(K_f)^{-1}$ as in Theorem~\ref{thm:main}.
Note that
\[
 \Big|m_K(\Omega', \tau)\Big|
\leq \|\tau\|_{L^1(G(\A_f))} m_K(\Omega')
= [\cpt:K] \frac{\vol(\Xi)}{\vol(\cpt_f)} m_K(\Omega').
\]

To finish the proof of our main theorem we need one last auxiliary result:

\begin{lemma}
Fix $h\in C_R^{\infty}(\aaa)^W$ with $h(0)=1$.
 There exist constants $c_1, c_2>0$ depending only on $n$, $\F$, $h$, and $\Omega$ such that for all $t\geq1$ we have
\begin{equation}\label{eq:spec:est1}
 m_K\big((B_t(0)\minus (i\aaa^*\cap B_t(0))\big)
\leq c_1 [\cpt:K]^{c_2} t^{d-2},
\end{equation}
 and moreover,
\begin{align}
 \int_{t\Omega}\int_{i\aaa^*} \hat{h}(\lambda-\mu) \beta(\lambda) \,d\lambda\,d\mu - \int_{t\Omega}\beta(\lambda)\,d\lambda
& \leq c_1 t^{d-1}, \;\;\;\text{ and } \label{eq:spec:est2} \\
\bigg| \sum_{\lambda\in\aaa_{\C}^*} m_K(\lambda, \tau) \int_{t\Omega}\hat{h}(\lambda-\mu) \,d\mu - m_K(t\Omega, \tau)\bigg|
&\leq c_1 [\cpt:K]^{c_2} \frac{\vol(\Xi)}{\vol(\cpt_f)}~ t^{d-1}.\label{eq:spec:est3}
\end{align}
for all $t\geq1$ and all $\tau$ as above.
\end{lemma}
\begin{proof}
The first inequality~\eqref{eq:spec:est1} is proven in~\cite[Corollary 4.6]{LaMu09} for the base field $\Q$ and without explicit dependence on $K$. However, their proof easily carries over to our situation if we use our estimate~\eqref{eq:spec:meas:ball1} in place of~\cite[Proposition 4.5]{LaMu09}.
The second estimate~\eqref{eq:spec:est2} does not make any reference to $K$ and its proof in the case $\F=\Q$ in~\cite[pp. 137-138]{LaMu09} also works for our field $\F$.

For the last estimate~\eqref{eq:spec:est3} we also closely follow~\cite[pp. 137-138]{LaMu09}: The left hand side of~\eqref{eq:spec:est3} can be written as
\begin{align*}
 -\sum_{\lambda\in\aaa_{\C}^*, ~\lambda\in t\Omega} m_K(\lambda, \tau)\int_{t\Omega}\hat{h}(\lambda-\mu)\,d\mu
& +\sum_{\lambda\in\aaa_{\C}^*, ~\lambda\in t\Omega'} m_K(\lambda, \tau)\int_{t\Omega}\hat{h}(\lambda-\mu)\,d\mu\\
& +\sum_{\lambda\in\aaa_{\C}^*, ~\lambda\not\in i\aaa^*} m_K(\lambda, \tau)\int_{t\Omega}\hat{h}(\lambda-\mu)\,d\mu
\end{align*}
for $\Omega'\subseteq i\aaa^*$ the complement of $\Omega$ in $i\aaa^*$. Let $T>0$ be such that $\Omega\subseteq B_T(0)$, and define for $k>0$ the set
\[
 \partial_k(t\Omega)=\{\nu\in i\aaa^*\mid \inf_{\mu\in t\partial\Omega}\|\nu-\mu\|\leq k\}.
\]
Then the absolute value of the above sum can be bounded by
\[
\|\tau\|_{L^1(G(\A_f))} \sum_{k=1}^{\infty} k^{-N} \big(m_K(\partial_k(t\Omega)) + m_K(B_{T(t+k)}(0)\minus i\aaa^*)\big) k^{-n+1}.
\]
This in turn can be bounded by using~\eqref{eq:spec:meas:ball1} (together with the fact that $\partial\Omega$ is piecewise $C^2$), and~\eqref{eq:spec:est1} to obtain the desired result.
\end{proof}

We can now finish the proof of Theorem~\ref{thm:main}: Fix a function $h\in C_R^{\infty}(\aaa)^W$ with $h(0)=1$ and define $f_{\C}^{\mu}$ as usual.
By Plancherel inversion we have
\[
 \vol(G(\F)\backslash G(\A)^1)\int_{t\Omega}f_{\C}^{\mu}(1)\,d\mu
=|W|^{-1} \vol(G(\F)\backslash G(\A)^1) \int_{t\Omega}\int_{i\aaa^*}\hat{h}(\lambda-\mu)\beta(\lambda)\,d\lambda\,d\mu.
\]
Hence combining Corollary~\ref{cor:estimate:geom:general:hecke:fct} with~\eqref{eq:spec:est2} and Proposition~\ref{prop_of_contr_classes} we get
\begin{align*}
&\Big| \int_{t\Omega} J_{\text{geom}}(f_{\C}^{\mu}\cdot\tau)\,d\mu - \Lambda_0(t)\sum_{z\in Z(\F)}\tau(z)\Big|\\
= &\Big| \int_{t\Omega} J_{\text{geom}}(f_{\C}^{\mu}\cdot\tau)\,d\mu 
-|W|^{-1}\vol(G(\F)\backslash G(\A)^1)\int_{t\Omega}\beta(\lambda)\,d\lambda \sum_{z\in Z(\F)} \tau(z)\Big|\\
\leq & c_1 [\cpt:K]^{c_2} \bigg(\frac{\vol(\Xi)}{\vol(\cpt_f)}\bigg)^{c_3}~ t^{d-1} (\log t)^{n+1}.
\end{align*}
On the other hand, using~\eqref{eq:spec:est3} we can approximate $\int_{t\Omega}J_{\text{spec}}(f_{\C}^{\mu}\cdot \tau)\,d\mu$ by $m_K(t\Omega,\tau)$. Hence by Proposition~\ref{prop:growth_spec} we have
\[
 \bigg|\int_{t\Omega} J_{\text{spec}}(f_{\C}^{\mu}\cdot \tau) \,d\mu -m_K(t\Omega, \tau)\bigg|
\leq \tilde c_1 [\cpt:K]^{\tilde c_2} \bigg(\frac{\vol(\Xi)}{\vol(\cpt_f)}\bigg)^{\tilde c_3} ~ t^{d-1} (\log t)^n.
\]
Using the trace formula $J_{\text{geom}}(f_{\C}^{\mu}\cdot\tau)=J_{\text{spec}}(f_{\C}^{\mu}\cdot\tau)$, we therefore get 
\[
\lim_{t\rightarrow\infty} \frac{m_K(t\Omega, \tau)}{\Lambda_0(t)}
= \sum_{z\in Z(\F)}\tau(z), 
\]
  and 
\[
 \big|m_K(t\Omega, \tau)-\Lambda_0(t)\sum_{z\in Z(\F)} \tau(z)\big|
\leq c_1' [\cpt:K]^{c_2'} \bigg(\frac{\vol(\Xi)}{\vol(\cpt_f)}\bigg)^{c_3'}~ t^{d-1} (\log t)^{n+1}
\]
for all $t\geq 2$. 
By Moeglin-Waldspurger's classification of the residual spectrum of $\GL_n(\A)$~\cite{MoWa89}, we know that if $\pi\in\Pi_{\text{disc}}(G(\A)^1)\minus\Pi_{\text{cusp}}(G(\A)^1)$, then $\lambda_{\pi_{\infty}}\in \aaa_{\C}^*\minus i\aaa^*$. Since $t\Omega\subseteq i\aaa^*$ and $\tr\pi^{\cpt_{\infty}}(\tau)\neq0$ implies that $\pi$ has a $K$-fixed vector, we therefore get $m_K(t\Omega, \tau)=\sum_{\pi\in\FFF_{K, t}}\tr\pi^{\cpt_{\infty}}(\tau)$ which finishes the proof of Theorem~\ref{thm:main}.

\bibliographystyle{amsalpha}
\bibliography{bibzetafcts1}
\end{document}